\documentclass[reqno]{amsart}     

\usepackage{amsmath,amsfonts,latexsym,amssymb}
\usepackage{graphicx} 
\usepackage{enumerate}
\usepackage{stmaryrd}
\usepackage{tikz}

\allowdisplaybreaks

\definecolor{fg}{rgb}{0.0,0.5,0.0}
\definecolor{bb}{rgb}{0.54, 0.81, 0.94}
\definecolor{gsa}{rgb}{0.66, 0.89, 0.63}

\newcommand{\hF}{\widehat{F}}

\newcommand{\cS}{\mathcal{S}}
\newcommand{\cT}{\mathcal{T}}

\newcommand{\cE}{\mathcal{E}}
\newcommand{\cM}{\mathcal{M}}

\newcommand{\Ome}{\Omega}
\newcommand{\oOme}{\overline{\Omega}}
\newcommand{\NJ}{\mathbb{N}_J}
\newcommand{\bn}{\mathbf{n}}

\newcommand{\del}{\delta}
\newcommand{\ou}{\overline{u}}
\newcommand{\uu}{\underline{u}}

\newcommand{\p}{\partial}

\newcommand{\bh}{\mathbf{h}}
\newcommand{\bx}{\mathbf{x}}

\newcommand{\by}{\mathbf{y}}
\newcommand{\bfv}{\mathbf{v}}
\newcommand{\bz}{\mathbf{z}}

\newcommand{\bk}{\mathbf{k}}
\newcommand{\bW}{\mathbf{W}}

\newcommand{\poly}{\mathbb{P}}

\newtheorem{remark}{Remark}[section]
\newtheorem{theorem}{Theorem}[section]

\newtheorem{definition}{Definition}[section]
\newtheorem{lemma}{Lemma}[section]

\begin{document}

\title[A narrow-stencil numerical framework]{A narrow-stencil framework for convergent numerical approximations of fully nonlinear second order PDEs}

\author[X. Feng]{Xiaobing Feng}  
\address{Department of Mathematics, The University of Tennessee, Knoxville, TN 37996}
\email{xfeng@utk.edu}
\thanks{The work of the first author was partially supported by the NSF grants DMS-1620168 and DMS-2012414. The work of the second and third  authors was partially supported by the NSF grant DMS-2111059.}

\author[T. Lewis]{Thomas Lewis}  
\address{Department of Mathematics and Statistics, The University of North Carolina at Greensboro, Greensboro, NC 27412}
\email{tllewis3@uncg.edu}
%\thanks{The work of the second and third  authors was partially supported by the NSF grant DMS-2111059.}

\author[K. Ward]{Kellie Ward}  
\address{Department of Mathematics and Statistics, The University of North Carolina at Greensboro, Greensboro, NC 27412}
\email{kmward7@uncg.edu}
%\thanks{The work of the third author was partially supported by the NSF grant DMS-2111059.}

\date{\today}



%%%%%%%%%%%%%%%%%%%%%%%%%

\begin{abstract}
This paper develops a unified general framework for designing convergent 
finite difference and discontinuous Galerkin methods 
for approximating viscosity and regular solutions of fully nonlinear second order PDEs. 
Unlike the well-known monotone (finite difference) framework, 
the proposed new framework allows for the use of narrow stencils and unstructured grids 
which makes it possible to construct high order methods. 
%The main building blocks for the new narrow-stencil framework are various one-sided first order 
%and various symmetric second order numerical derivative operators.  
The general framework is based on the concepts of consistency and g-monotonicity 
which are both defined in terms of various numerical derivative operators.  
Specific methods that satisfy the framework are constructed using numerical moments. 
Admissibility, stability, and convergence properties are proved,  
and numerical experiments are provided along with some computer implementation details.  
\end{abstract}

\maketitle

%%%%%%%%%%%%%%%%%%%%%%%%%%%%%%%%%%%%%

\section{Introduction} \label{intro_sec}
This paper develops a unified general framework for designing convergent narrow-stencil 
finite difference (FD) and discontinuous Galerkin (DG) methods for approximating 
the viscosity (and regular) solution to the following fully nonlinear second order 
Dirichlet boundary value problem:
\begin{subequations}\label{bvp}
\begin{alignat}{2}
F[u](\bx)\equiv F \left( D^2 u, \nabla u, u, \bx \right) 
	& = 0, &&\qquad \forall \bx \in \Omega,
	\label{FD_pde} \\
u(\bx) & = g(\bx), && \qquad \forall \bx \in \partial \Omega,  \label{FD_bc} 
\end{alignat}
\end{subequations}
where $\Omega \subset \mathbb{R}^d$ ($d=2,3$) is a bounded domain and $D^2u(\bx)$ 
denotes the Hessian matrix of $u$ at $\bx$. 
The partial differential equation (PDE) operator 
$F: \cS^{d\times d}\times \mathbb{R}^d\times\mathbb{R} \times \oOme \to \mathbb{R}$, 
where $\cS^{d\times d} \subset \mathbb{R}^{d \times d}$ denotes the set
of $d\times d$ symmetric real matrices, 
is a fully nonlinear second order differential operator in the sense that $F$ is nonlinear 
in at least one component of the Hessian $D^2 u$.  
The boundary data $g$ is assumed to be continuous with $F$ Lipschitz with respect to its 
first three arguments. 
Moreover, $F$ is assumed to be {\em uniformly and proper elliptic} and satisfy a comparison principle 
(see Section~\ref{PDEtheory_sec} for the definitions).  
In this paper we focus 
our attention on two main classes of fully nonlinear second order PDEs, namely, the 
Monge-Amp\`ere-type and Hamilton-Jacobi-Bellman-type equations 
(cf. \cite{Gilbarg_Trudinger01,Feng_Glowinski_Neilan10}). 

Fully nonlinear second order PDEs arise from many scientific and engineering applications 
such as antenna design, astrophysics, economics, differential geometry, stochastic optimal control, 
and optimal mass transport; yet, they are a class of PDEs which are
difficult to study analytically and even more challenging to approximate numerically.  
Due to the fully nonlinear structure, there is no general variational (or weak) formulation.  
As a result, its weak solution concept 
(called {\em viscosity solutions}, see Section~\ref{PDEtheory_sec} for the definition) 
is complicated and, in particular, very difficult to address numerically. 
Nevertheless, driven by the need for solving many emerging and intriguing application problems, 
numerical fully nonlinear PDEs has garnered a lot attention and experienced rapid developments 
in recent years.   
See \cite{Feng_Glowinski_Neilan10,NSZ16} and the references therein
for an overview of various numerical methods that have been proposed, analyzed, and tested. 

To the best of our knowledge, there are only two main approaches in the literature which aim to approximate viscosity solutions.  
Both approaches have been used successfully to 
design and analyze (practical) numerical methods for approximating second order fully nonlinear PDEs. 
The first approach, which was adopted in the overwhelming majority of the existing works, 
is the {\em Barles-Souganidis'  monotone (wide-stencil) finite difference framework} 
(cf. \cite{Barles_Souganidis91,Oberman08,Debrabant_Jakobsen13,Feng_Jensen,Jensen_Smears,NNZ17}).  
We call the second approach the {\em numerical moment-enhanced  g-monotone (narrow-stencil) finite difference
and DG framework} (cf. \cite{FengKaoLewis13,FengLewis18,FengLewis21,Lewis_dissertation} 
and also the original vanishing moment method \cite{Feng_Neilan08}). 
It is well-known that monotone (in the sense of Barles-Souganidis \cite{Barles_Souganidis91}) methods are difficult to construct; 
moreover, they are intrinsically low order, require the use of wide stencils, and yield strongly coupled nonlinear algebraic problems that need to be solved. 
The narrow-stencil approach aims to sidestep these limitations of the wide-stencil approach 
so high order methods can be constructed on both structured and non-structured grids.  
On the other hand, since the narrow-stencil approach abandons the standard 
monotonicity requirement, it prevents one to directly use the powerful Barles-Souganidis'  framework 
for the convergence analysis. 
Consequently, new machineries and techniques must be developed for the analysis of the proposed 
narrow-stencil methods which, so far, has only been done on a case-by-case basis in \cite{FengKaoLewis13,FengLewis18,FengLewis21,Lewis_dissertation}. 
The primary goal of this paper is to re-examine (with a top-down view) 
and refine the {\em numerical moment-enhanced  g-monotone  (narrow-stencil) finite difference and DG approach}, 
which was initiated by us in \cite{FengKaoLewis13,FengLewis18,FengLewis21,Lewis_dissertation}, 
and to formulate it into a unified framework which is parallel to the Barles-Souganidis' framework.  
A far-reaching goal is to provide a blueprint/framework for designing and analyzing practical 
convergent numerical methods for approximating viscosity (and regular) solutions of fully nonlinear second order PDEs.

The remainder of this paper is organized as follows.  
In Section \ref{PDEtheory_sec}, we first recall the basics of viscosity solution theory and elliptic operators as well as the necessary notations. 
We then introduce various FD and DG finite element 
numerical derivative operators (cf. \cite{FengLewis21,DG_calc}) and unify the notations. 
Those discrete derivative operators are the building blocks for our narrow-stencil framework. 
In Section~\ref{framework_sec}, we formulate our abstract framework and introduce the key concepts of 
{\em numerical operators, consistency, g-monotonicity,} and {\em numerical moments}. 
We then state a few structure conditions/assumptions for numerical operators. 
It should be noted that all of these concepts and conditions are motivated by and abstractions of similar ones in our earlier works 
\cite{FengKaoLewis13,FengLewis18,FengLewis21,Lewis_dissertation}. 
The {\em numerical moment} will play a critical role in the 
specific examples of numerical operators that are presented.  
In Sections \ref{convergence_sec}--\ref{stability_sec}, we present a complete 
convergence, admissibility, and stability analysis for the narrow-stencil FD 
methods proposed in Section \ref{framework_sec}. 
Unlike the Barles-Souganidis'  monotone  (wide-stencil) framework where admissibility 
and the $\ell^\infty$-norm stability of the underlying numerical methods are almost free to 
obtain (thanks to the monotonicity), our results require entirely new techniques 
and their proofs given in Sections \ref{convergence_sec}--\ref{stability_sec} are more  technical and involved (as well as much longer).   
These technical issues are precisely the price to pay for using narrow stencils. 
Assuming the admissibility and $\ell^\infty$ stability, we first establish the convergence of the numerical solution to the viscosity solution of the underlying PDE problem in Section \ref{convergence_sec}. The proof is adapted from the much more detailed version in \cite{FengLewis21}. 
We then prove the desired admissibility and $\ell^\infty$ stability in Sections \ref{admissible_sec} and \ref{stability_sec}. 
Our main idea of proving the admissibility is to use the Contractive Mapping Theorem in $\ell^2$ instead of $\ell^\infty$.
The $\ell^\infty$ stability is obtained by a novel numerical embedding technique first introduced in \cite{FengLewis21}. 
Finally, in Section \ref{experiments_sec}, we present some numerical experiments to 
demonstrate the effectiveness of the proposed framework and to address some computer implementation issues.

%%%%%%%%%%%%%%%%%%%%%%%%%%%%%%%%%%%%%

%\section{PDE Definitions and Notation} \label{PDEtheory_sec}
\section{Preliminaries and numerical derivatives} \label{PDEtheory_sec}

\subsection{Notation and definitions}
The narrow-stencil framework will rely upon two different partial orderings for matrices.  
We will utilize the convention that 
$A \geq B$ if $A - B$ is symmetric nonnegative definite for symmetric matrices $A$ and $B$.  
We will also introduce the alternative convention that  
$A \succeq B$ if each component of $A-B$ is nonnegative.  
Note that the partial ordering induced by $\succeq$ does not require symmetric matrices.  
We also let $A:B$ denote the Frobenius inner product with $A:B \equiv \sum_{i=1}^d \sum_{j=1}^d a_{ij} b_{ij}$ 
for all matrices $A, B \in \mathbb{R}^{d \times d}$. 

For a bounded open domain $\Ome\subset\mathbb{R}^d$, let $B(\Ome)$, 
$USC(\Ome)$, and $LSC(\Ome)$ denote, respectively, the spaces of bounded,
upper semi-continuous, and lower semicontinuous functions on $\Ome$.
For any $v\in B(\Ome)$, we define
\[
v^*(\bx):=\limsup_{\by \to \bx} v(\by) \qquad\mbox{and}\qquad
v_*(\bx):=\liminf_{\by\to \bx} v(\by). 
\]
Then, $v^*\in USC(\Ome)$ and $v_*\in LSC(\Ome)$, and they are called
{\em the upper and lower semicontinuous envelopes} of $v$, respectively.

Below we use the convention of writing the boundary condition as a
discontinuity of the PDE (cf. \cite[p.274]{Barles_Souganidis91}).
The Dirichlet boundary condition is assumed to hold in the viscosity sense.  
The following two definitions can be found in \cite{Gilbarg_Trudinger01,
Caffarelli_Cabre95,Barles_Souganidis91}.

\begin{definition}\label{proper_elliptic_def}
Equation \eqref{bvp} is said to be {\em proper elliptic} if for  all $(\mathbf{q}, \bx) \in\mathbb{R}^d \times \oOme$, there holds
\begin{align*}
F(A,\mathbf{q}, v, \bx) \leq F(B,\mathbf{q}, w, \bx) \qquad\forall 
A,B\in \cS^{d\times d},\, A\geq B, \;
v,w \in \mathbb{R},\, v \leq w . 
\end{align*}
\end{definition}

\begin{definition}\label{uniformly_elliptic_def}
Equation \eqref{bvp} is said to be {\em uniformly elliptic} if 
there exists $\Lambda \geq \lambda > 0$ such that, 
for all $(\mathbf{q},v, \bx) \in\mathbb{R}^d\times \mathbb{R}\times \oOme$, there holds
\begin{align*}
	0 \geq - \lambda\, \mbox{\rm tr}(A-B)  
	\geq F(A,\mathbf{q},v,\bx) - F(B,\mathbf{q},v,\bx) 
	\geq - \Lambda\, \mbox{\rm tr}(A-B)  
\end{align*}
for all $A,B\in \cS^{d\times d}$ with $A \geq B$.  
\end{definition}

We note that when $F(A,\mathbf{q},v,\bx)$ is differentiable with respect to the first parameter, 
then the proper ellipticity definition is equivalent to 
requiring that the matrix $\frac{\partial F}{\partial A}$
is negative semi-definite 
and the value $\frac{\partial F}{\partial v}$ is nonnegative 
(cf. \cite[p. 441]{Gilbarg_Trudinger01}).
If $F$ is also uniformly elliptic, then there holds 
$0 > - \lambda |\vec{\xi}|^2 \geq \vec{\xi} \cdot \frac{\partial F}{\partial A} \vec{\xi} \geq - \Lambda | \vec{\xi} |^2$ 
for all $\vec{\xi} \neq \vec{0}$.  
Thus, $\lambda I \leq -\frac{\partial F}{\partial A} \leq \Lambda I$.  

\begin{definition}\label{viscosity_soln_def}
A function $u\in B(\Ome)$ is called a viscosity subsolution (resp.
supersolution) of \eqref{bvp} if, for all $\varphi\in C^2(\oOme)$,
if $u^*-\varphi$ (resp. $u_*-\varphi$) has a local maximum
(resp. minimum) at $\bx_0\in \oOme$, then we have
\[
F_*(D^2\varphi(\bx_0),\nabla \varphi(\bx_0),u^*(\bx_0), \bx_0) \leq 0 
\]
(resp. $F^*(D^2\varphi(\bx_0),\nabla \varphi(\bx_0), u_*(\bx_0), \bx_0) \geq 0$).
The function $u$ is said to be a viscosity solution of \eqref{bvp}
if it is simultaneously a viscosity subsolution and a viscosity
supersolution of \eqref{bvp}.
\end{definition}

\begin{definition}\label{comparison_def}
Problem \eqref{bvp} is said to satisfy a {\em comparison 
principle} if the following statement holds. For any upper semicontinuous 
function $u$ and lower semicontinuous function $v$ on $\overline{\Omega}$,  
if $u$ is a viscosity subsolution and $v$ is a viscosity supersolution 
of \eqref{bvp}, then $u\leq v$ on $\overline{\Omega}$.
\end{definition}

Since we assume that $F$ in \eqref{bvp} satisfies the comparison principle, 
we have that the underlying viscosity solution $u$ must be continuous.  
Furthermore, if $F$ is continuous with respect to $\bx$, then by the Lipschitz continuity 
with respect to $D^2 u$ and $u$, we can drop the upper and lower $*$ indices in 
Definition~\ref{viscosity_soln_def}.  

%%%%%%%%%%%%%%%%%%%%%%%%%%%%%%%%%%%%%%%%%

\subsection{Finite difference derivative operators} \label{differences_sec}

We introduce several difference operators for approximating first and second order 
partial derivatives. 
The narrow-stencil framework will use multiple difference operators to help resolve 
the underlying viscosity solution.  
The notation and difference operators used are the same as those in \cite{FengLewis21}.  
The section ends with a formal result for comparing various discrete second order operators.

%%%% 

\subsubsection{\bf Finite difference grids} \label{mesh_sec}

Assume $\Omega$ is a $d$-rectangle, i.e., 
$\Omega = \left( a_1 , b_1 \right) \times \left( a_2 , b_2 \right) \times \cdots \times 
		\left( a_d , b_d \right)$.    
We shall only consider grids that are uniform in each coordinate $x_i$, $i = 1, 2, \ldots, d$.  
Let $J_i (\geq 2)$ be an integer and $h_i = \frac{b_i-a_i}{J_i-1}$ for $i = 1, 2, \ldots, d$. 
Define $\mathbf{h} = \left( h_1, h_2, \ldots, h_d \right) \in \mathbb{R}^d$, 
$h = \max_{i=1,2,\ldots,d} h_i$,  $J = \prod_{i=1}^d J_i$, and   
$\NJ = \{ \alpha = (\alpha_1, \alpha_2, \ldots, \alpha_d) 
\mid 1 \leq \alpha_i \leq J_i, i = 1, 2, \ldots, d \}$.  Then, $\left| \NJ \right| = J$. 
We partition $\Omega$ into $\prod_{i=1}^d \left(J_i-1 \right)$ sub-$d$-rectangles with grid points
$\mathbf{x}_{\alpha} = \bigl (a_1+ (\alpha_1-1)h_1 , a_2 + (\alpha_2-1) h_2 , \ldots , 
		a_d + (\alpha_d-1) h_d \bigr)$
for each multi-index $\alpha \in \NJ$.
We call $\cT_{\mathbf{h}}=\{\mathbf{x}_{\alpha} \}_{\alpha \in \NJ}$ 
a mesh (set of nodes) for $\overline{\Omega}$. 
We also introduce an extended mesh 
$\cT_{\mathbf{h}}'$ which extends  
$\cT_{\mathbf{h}}$ by a collection of ghost grid points that are at most one layer exterior to 
$\overline{\Omega}$ in each coordinate direction. 
In particular, we choose ghost grid points $\bx$ such that 
$\bx = \by \pm 2 h_i \mathbf{e}_i$ for some $\by \in \cT_{\bh} \cap \Omega$, $i \in \{1,2,\ldots,d\}$.  
We set $J_i' = J_i+2$ and $\NJ'$ is defined by replacing $J_i$ by $J_i'$ in 
the definition of $\NJ$ and then removing the extra multi-indices that would correspond to ghost grid points 
that are not in the set $\cT_{\bh}'$ to ensure $| \NJ' | = | \cT_{\bh}' |$.

%%% 

\subsubsection{\bf First order finite difference operators} \label{discrete_gradients_sec}
 
Let $\left\{ \mathbf{e}_i \right\}_{i=1}^d$ denote the canonical basis vectors for $\mathbb{R}^d$. 
Define the (first order) forward and backward difference operators by
\begin{equation} \label{fd_x}
	\del_{x_i,h_i}^+ v(\mathbf{x})\equiv \frac{v(\mathbf{x} + h_i \mathbf{e}_i) - v(\mathbf{x})}{h_i},\qquad
	\del_{x_i,h_i}^- v(\mathbf{x})\equiv \frac{v(\mathbf{x})- v(\mathbf{x}-h_i \mathbf{e}_i)}{h_i}
\end{equation}
for a function $v$ defined on $\mathbb{R}^d$.  
We also consider the central difference operator $\overline{\delta}_{x_i, h_i}$ defined by 
$\overline{\delta}_{x_i, h_i} \equiv \frac12 \delta_{x_i, h_i}^+ + \frac12 \delta_{x_i, h_i}^-$ so that 
\[
	\overline{\del}_{x_i,h_i} v(\mathbf{x})\equiv \frac{v(\mathbf{x} + h_i \mathbf{e}_i) - v(\mathbf{x} - h_i \mathbf{e}_i)}{2h_i} . 
\]
The corresponding forward, backward, and central discrete gradient operators are denoted by 
$\nabla_{\bh}^+$, $\nabla_{\bh}^-$, and $\overline{\nabla}_{\bh}$, respectively.  

%%%%%%

\subsubsection{\bf Second order finite difference operators} \label{discrete_Hessian_sec}

Using the forward and backward 
difference operators introduced in the previous subsection, we have the following
four possible approximations of the second order differential operator
$\partial^2_{x_ix_j}$ given by 
$D_{\mathbf{h},ij}^{\mu\nu}\equiv \delta_{x_j, h_j}^\nu \delta_{x_i, h_i}^\mu$ 
for $\mu,\nu\in \{+,-\}$, 
which in turn leads to the definition of the following four approximations
of the Hessian operator $D^2\equiv [\partial^2_{x_ix_j}]$: 
\begin{equation*}%\label{discrete_Hessian}
D_{\mathbf{h}}^{\mu\nu} \equiv \bigl[D_{\mathbf{h},ij}^{\mu\nu}\bigr]_{i,j=1}^d  \qquad
\mbox{for } \mu,\nu\in \{+,-\}.
\end{equation*}

To analyze our narrow-stencil framework in the next section, we also need to introduce the
following two sets of averaged second order difference operators:
\begin{subequations}\label{FD_partials}
\begin{align}
	\widehat{\delta}_{x_i, x_j; h_i, h_j}^2 
	& \equiv \frac12 \bigl( D_{\mathbf{h},ij}^{+-} + D_{\mathbf{h},ij}^{-+} \bigr)
          =\frac12 \bigl(\delta_{x_j, h_j}^- \delta_{x_i, h_i}^+  
		+ \delta_{x_j, h_j}^+ \delta_{x_i, h_i}^- \bigr) , \\ 
	\widetilde{\delta}_{x_i, x_j; h_i, h_j}^2 
	& \equiv \frac12 \bigl( D_{\mathbf{h},ij}^{--} + D_{\mathbf{h},ij}^{++} \bigr)
           = \frac12 \bigl(\delta_{x_j, h_j}^+ \delta_{x_i, h_i}^+  
		+ \delta_{x_j, h_j}^-\delta_{x_i, h_i}^-  \bigr) ,  
\end{align}
\end{subequations}
for all $i,j = 1,2,\ldots, d$.   
Note that, while the various components of $D_{\bh}^{\mu \nu}$ may not be self-adjoint operators, 
the discrete second order partial derivative operators 
$\widehat{\delta}_{x_i, x_j; h_i, h_j}^2$ and $\widetilde{\delta}_{x_i, x_j; h_i, h_j}^2$ 
defined by \eqref{FD_partials} are self-adjoint 
as can be verified by the component forms in \cite{FengLewis21}.  

Using the above difference operators, we define the following two ``centered" approximations 
of the Hessian operator $D^2\equiv [\partial^2_{x_ix_j}]$:
\begin{equation} \label{fd_hessc}
\widehat{D}_{\mathbf{h}}^2\equiv \bigl[\widehat{\delta}_{x_i, x_j; h_i, h_j}^2 \bigr]_{i,j=1}^d,\qquad
\widetilde{D}_{\mathbf{h}}^2 \equiv \bigl[\widetilde{\delta}_{x_i, x_j; h_i, h_j}^2 \bigr]_{i,j=1}^d . 
\end{equation}
We will also consider the 
average central approximation of the Hessian operator $\overline{D}_{\bh}^2$ defined by 
\begin{equation}
	\overline{D}_{\bh}^2 \equiv \frac12 \widehat{D}_{\bh}^2 + \frac12 \widetilde{D}_{\bh}^2 . 
\end{equation}
A visual representation of the local stencils for the various discrete Hessians 
$\widetilde{D}_{\bh}^2$, $\widehat{D}_{\bh}^2$, and $\overline{D}_{\bh}^2$ 
can be found in Figure \ref{fig:steps}. 

For notation brevity, we set
$\delta_{x_i,h_i}^2\equiv\widehat{\delta}_{x_i,h_i}^2\equiv\widehat{\delta}_{x_i, x_i; h_i, h_i}^2$ 
and $\overline{\delta}_{x_i, h_i}^2 \equiv \overline{\delta}_{x_i , x_j ; h_i , h_j}^2$.  
Then, there holds 
\begin{align*}
	\delta_{x_i,h_i}^2 v(\mathbf{x}) 
	& = \frac{v(\mathbf{x}-h_i \mathbf{e}_i) - 2 v(\mathbf{x}) + v(\mathbf{x}+h_i \mathbf{e}_i)}{h_i^2}  
\end{align*}
and 
\[
	\overline{\delta}_{x_i, h_i}^2 = \delta_{x_i, h_i} \delta_{x_i , h_i} = \delta_{x_i, 2 h_i}^2 
\]
for all $i=1,2,\ldots,d$.  
Lastly, we denote the discrete Laplacian operator by 
$\Delta_{\bh} \equiv \sum_{i=1}^d \delta_{x_i, h_i}^2$.  

\begin{figure}[h]
	\centering 
	\begin{tikzpicture}[scale=0.6] 
%%% SQUARES %%%
    \filldraw[draw=black,fill=bb] (0.75,-0.25) rectangle (1.25,0.25); 
    \filldraw[draw=black,fill=bb] (1.75,-0.25) rectangle (2.25,0.25);
    \filldraw[draw=black,fill=bb] (0.75,-1.25) rectangle (1.25,-0.75);
    \filldraw[draw=black,fill=bb] (1.75,-1.25) rectangle (2.25,-0.75); 
    \filldraw[draw=black,fill=bb] (2.75,-1.25) rectangle (3.25,-0.75); 
    \filldraw[draw=black,fill=bb] (1.75,-2.25) rectangle (2.25,-1.75); 
    \filldraw[draw=black,fill=bb] (2.75,-2.25) rectangle (3.25,-1.75);

    \filldraw[draw=black,fill=bb] (7.75,-0.25) rectangle (8.25,0.25); 
    \filldraw[draw=black,fill=bb] (8.75,-0.25) rectangle (9.25,0.25); 
    \filldraw[draw=black,fill=bb] (6.75,-1.25) rectangle (7.25,-0.75); 
    \filldraw[draw=black,fill=bb] (7.75,-1.25) rectangle (8.25,-0.75); 
    \filldraw[draw=black,fill=bb] (8.75,-1.25) rectangle (9.25,-0.75); 	
    \filldraw[draw=black,fill=bb] (6.75,-2.25) rectangle (7.25,-1.75); 
    \filldraw[draw=black,fill=bb] (7.75,-2.25) rectangle (8.25,-1.75);

    \filldraw[draw=black,fill=bb] (12.75,-0.25) rectangle (13.25,0.25); 
    \filldraw[draw=black,fill=bb] (14.75,-0.25) rectangle (15.25,0.25); 
    \filldraw[draw=black,fill=bb] (12.75,-2.25) rectangle (13.25,-1.75); 
    \filldraw[draw=black,fill=bb] (14.75,-2.25) rectangle (15.25,-1.75); 
    
%%% DIAMONDS %%%
    % Original Point: (2,0)
		\filldraw[draw=black,fill=pink] (1.650,0.00) -- (2.00,0.350) -- (2.350,0.00) -- (2.00,-0.350) -- cycle; 

    % Original Point: (2,-1)
		\filldraw[draw=black,fill=pink] (1.650,-1.00) -- (2.00,-.650) -- (2.350,-1.00) -- (2.00,-1.350) -- cycle; 

    % Original Point: (2,-2)
		\filldraw[draw=black,fill=pink] (1.650,-2.00) -- (2.00,-2.350) -- (2.350,-2.00) -- (2.00,-1.65) -- cycle;

    % Original Point: (8,1)
		\filldraw[draw=black,fill=pink] (7.65,1.00) -- (8.00,1.35) -- (8.35,1.00) -- (8.00,.65) -- cycle; 

    % Original Point: (8,0)
		\filldraw[draw=black,fill=pink] (7.65,0.00) -- (8.00,0.35) -- (8.35,0.00) -- (8.00,-0.35) -- cycle; 
		
    % Original Point: (8,-1)
		\filldraw[draw=black,fill=pink] (7.65,-1.00) -- (8.00,-0.65) -- (8.35,-1.00) -- (8.00,-1.35) -- cycle; 

    % Original Point: (8,-2)
		\filldraw[draw=black,fill=pink] (7.65,-2.00) -- (8.00,-1.65) -- (8.35,-2.00) -- (8.00,-2.35) -- cycle; 

    % Original Point: (8,-3)
		\filldraw[draw=black,fill=pink] (7.65,-3.00) -- (8.00,-2.65) -- (8.35,-3.00) -- (8.00,-3.35) -- cycle;

    % Original Point: (14,1)
		\filldraw[draw=black,fill=pink] (13.65,1.00) -- (14.00,1.35) -- (14.35,1.00) -- (14.00,0.65) -- cycle; 

    % Original Point: (14,-1)
		\filldraw[draw=black,fill=pink] (13.65,-1.00) -- (14.00,-0.65) -- (14.35,-1.00) -- (14.00,-1.35) -- cycle; 

    % Original Point: (14,-3)
		\filldraw[draw=black,fill=pink] (13.65,-3.00) -- (14.00,-2.65) -- (14.35,-3.00) -- (14.00,-3.35) -- cycle; 

%%% CIRCLES %%%
    \filldraw[draw=black,fill=gsa] (1.00,-1) circle (7pt); 
	\filldraw[draw=black,fill=gsa] (2.00,-1) circle (7pt); 
	\filldraw[draw=black,fill=gsa] (3.00,-1) circle (7pt);

	\filldraw[draw=black,fill=gsa] (6.00,-1) circle (7pt); 
	\filldraw[draw=black,fill=gsa] (7.00,-1) circle (7pt); 
	\filldraw[draw=black,fill=gsa] (8.00,-1) circle (7pt); 
	\filldraw[draw=black,fill=gsa] (9.00,-1) circle (7pt); 
	\filldraw[draw=black,fill=gsa] (10.00,-1) circle (7pt);

	\filldraw[draw=black,fill=gsa] (12.00,-1) circle (7pt); 
	\filldraw[draw=black,fill=gsa] (14.00,-1) circle (7pt); 
	\filldraw[draw=black,fill=gsa] (16.00,-1) circle (7pt);

%%% POINTS %%%
	\fill (2,1) circle (3pt); 
	\fill (1,0) circle (3pt); 
	\fill (2,0) circle (3pt); 
	\fill (3,0) circle (3pt); 
	\fill (0,-1) circle (3pt); 
	\fill (1,-1) circle (3pt); 
	\fill (2,-1) circle (3pt); 
	\fill (3,-1) circle (3pt); 
	\fill (4,-1) circle (3pt); 
	\fill (1,-2) circle (3pt); 
	\fill (2,-2) circle (3pt); 
	\fill (3,-2) circle (3pt); 
	\fill (2,-3) circle (3pt); \node[below] at (2,-3.6){\scriptsize $ \widehat{D}_{\bh}^2 $};

	\fill (8,1) circle (3pt); 
	\fill (7,0) circle (3pt); 
	\fill (8,0) circle (3pt); 
	\fill (9,0) circle (3pt); 
	\fill (6,-1) circle (3pt); 
	\fill (7,-1) circle (3pt); 
	\fill (8,-1) circle (3pt); 
	\fill (9,-1) circle (3pt); 
	\fill (10,-1) circle (3pt); 
	\fill (7,-2) circle (3pt); 
	\fill (8,-2) circle (3pt); 
	\fill (9,-2) circle (3pt); 
	\fill (8,-3) circle (3pt); \node[below] at (8,-3.6){\scriptsize $ \widetilde{D}_{\bh}^2 $};

	\fill (14,1) circle (3pt); 
	\fill (13,0) circle (3pt); 
	\fill (14,0) circle (3pt); 
	\fill (15,0) circle (3pt); 
	\fill (12,-1) circle (3pt); 
	\fill (13,-1) circle (3pt); 
	\fill (14,-1) circle (3pt); 
	\fill (15,-1) circle (3pt); 
	\fill (16,-1) circle (3pt); 
	\fill (13,-2) circle (3pt); 
	\fill (14,-2) circle (3pt); 
	\fill (15,-2) circle (3pt); 
	\fill (14,-3) circle (3pt); \node[below] at (14,-3.6){\scriptsize $ \overline{D}_{\bh}^2 $};

%%% KEY %%%
	
	\filldraw[draw=black,fill=gsa] (17.85,0) circle (8pt);
    \node[right] at (18.3,0){{$\, u_{xx} \,$}};
    
	\filldraw[draw=black,fill=pink] (17.5,-1) -- (17.85,-0.65) -- (18.2,-1) -- (17.85,-1.35) -- cycle;  
	\node[right] at (18.2,-1){{$\,\ u_{yy} \,$}};
		
	\filldraw[draw=black,fill=bb] (18.15,-1.75) rectangle (17.6,-2.25); 
	\node[right] at (18.25,-2){{$\,\ u_{xy} \,$}};
	
	\end{tikzpicture} 
	\caption{Illustration of the local stencils for the three discrete Hessian operators 
	$\widehat{D}_{\bh}^2$, $\widetilde{D}_{\bh}^2$, and $\overline{D}_{\bh}^2$.  
	}
\label{fig:steps}
\end{figure}
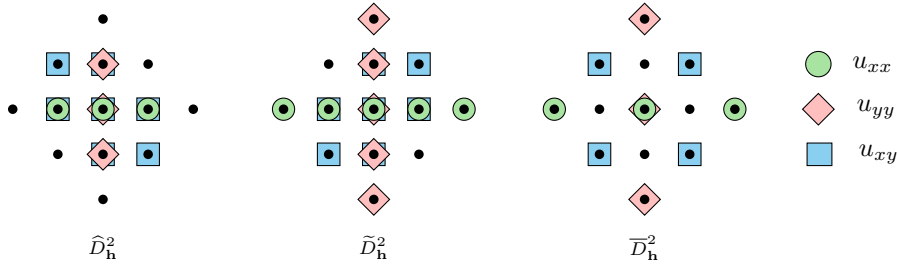

%%%%%%

\subsubsection{\bf Properties of second order finite difference operators} \label{discrete_Hessian_comparison_sec} 

We now derive a relationship between $\widehat{D}_{\bh}^2$ and $\widetilde{D}_{\bh}^2$ 
that will be essential to the analysis of our narrow-stencil methods.  
The following expands upon observations in \cite{FengLewis21}.  

Let $V$ be a grid function defined over $\cT_{\bh}'$.  
Suppose $V = 0$ over $\cT_{\bh} \cap \partial \Omega$
and $V$ satisfies $\Delta_{\bh} V_\alpha = 0$ for all $\bx_\alpha \in \mathcal{S}_{\bh}$, 
where $\mathcal{S}_{\bh} \subset \cT_{\bh} \cap \partial \Omega$ is defined by 
\begin{align}\label{Sh_grid}
	\mathcal{S}_{\bh} \equiv \big\{ \bx_\alpha \in \cT_{\bh} \cap \partial \Omega \mid & \,  
		\bx_{\alpha} + h_i \mathbf{e}_i \in \cT_{\bh} \cap \Omega \text{ or } 
		\bx_{\alpha} - h_i \mathbf{e}_i \in \cT_{\bh} \cap \Omega \\ 
		\nonumber & 
		\text{ for some } i \in \{1,2,\ldots,d\} \big\} . 
\end{align}
Since $V=0$ over $\cT_{\bh} \cap \partial \Omega$, there also holds 
$\delta_{x_i, h_i}^2 V_\alpha = 0$ for all $\bx_\alpha \in \mathcal{S}_{\bh}$ and $i \in \{1,2,\ldots,d\}$ 
such that either $\bx_{\alpha - \mathbf{e}_i} \in \cT_{\bh} \cap \Omega$ or 
$\bx_{\alpha + \mathbf{e}_i} \in \cT_{\bh} \cap \Omega$.  
Note that the additional boundary information is needed to incorporate the ghost points associated with 
the extended grid $\cT_{\bh}'$.  

Let $\widehat{D}_{ij,0}$ and $\widetilde{D}_{ij,0}$ denote the matrix representations of $[ \widehat{D}_{\bh}^2 ]_{ij}$ 
and $[ \widetilde{D}_{\bh}^2]_{ij}$, respectively, 
restricted to grid functions defined over $\cT_{\bh} \cap \Omega$ 
with the boundary assumptions for $V$ built in.  
Notationally, the zero subscript is used to denote the boundary conditions.  
Then $\widehat{D}_{ii,0}$ is symmetric positive definite for all $i = 1,2,\ldots,d$.  
%Let $\vec{v}$ denote the vectorization of $V$.  
%Then 
%$[ \widehat{D}_{ij} \vec{v}]_k = \widehat{\delta}_{x_i, x_j; h_i, h_j}^2 V_{\alpha(k)}$ 
%and 
%$[ \widetilde{D}_{ij} \vec{v}]_k = \widetilde{\delta}_{x_i, x_j; h_i, h_j}^2 V_{\alpha(k)}$ 
%for all $i,j = 1,2,\ldots,d$ and all $k$, 
%where $\alpha(k)$ denotes the multi-index associated with the vector index $k$.  
Lastly, let $\delta_{x_i, h_i;0}^2$ denote the central difference operator $\delta_{x_i, h_i}^2$ with 
the zero Dirichlet boundary data assumption.  

It is easy to check that (cf. \cite{FengLewis21}) there holds 
\[
	\widetilde{\delta}_{x_i, x_j; h_i, h_j}^2 V_\alpha - \widehat{\delta}_{x_i, x_j; h_i, h_j}^2 V_\alpha 
	= \frac{h_i h_j}{2} \delta_{x_i, h_i}^2 \delta_{x_j, h_j}^2 V_\alpha 
	%= \frac{h_i}{h_j}{2} \delta_{x_j, h_j}^2 \delta_{x_i, h_i}^2 V_\alpha 
\]
for all $i,j = 1,2,\ldots,d$ and $\bx_\alpha \in \cT_{\bh} \cap \Omega$.  
Suppose $i \neq j$.  
If $\bx_{\alpha \pm \mathbf{e}_j} \in \cT_{\bh} \cap \partial \Omega$, 
then $\delta_{x_i, h_i}^2 V_{\alpha \pm \mathbf{e}_j} = 0$ 
using the fact that $x_i$ is orthogonal to $x_j$ and $V = 0$ over $\cT_{\bh} \cap \partial \Omega$.  
Then 
\begin{align*}
	\delta_{x_i, h_i; 0}^2 \delta_{x_j, h_j; 0}^2 V_\alpha
	& = \delta_{x_i, h_i; 0}^2 \delta_{x_j, h_j}^2 V_\alpha 
	= \delta_{x_i, h_i}^2 \delta_{x_j, h_j}^2 V_\alpha , 
\end{align*} 
and, by a simple computation, 
\begin{align*}
	\delta_{x_i, h_i}^2 \delta_{x_j, h_j}^2 V_\alpha 	
	= \delta_{x_j, h_j}^2 \delta_{x_i, h_i}^2 V_\alpha 
	& = \delta_{x_j, h_j; 0}^2 \delta_{x_i, h_i; 0}^2 V_\alpha 
\end{align*}
for all $\bx_\alpha \in \cT_{\bh} \cap \Omega$.  
Thus, 
\[
	\widetilde{D}_{ij,0} - \widehat{D}_{ij,0} = \frac{h_i h_j}{2} \widehat{D}_{ii,0} \widehat{D}_{jj,0} 
	= \frac{h_i h_j}{2} \widehat{D}_{jj,0} \widehat{D}_{ii,0} 
\]
for all $i \neq j$, 
and it follows that the matrix is symmetric positive definite.  

Choose $i \in \{1,2,\ldots,d \}$.  
Observe that, since $\delta_{x_i, h_i}^2 V_\alpha = 0$ over $\mathcal{S}_{\bh}$, there holds 
$\delta_{x_i, h_i}^2 \delta_{x_i, h_i}^2 V_\alpha = \delta_{x_i, h_i; 0}^2 \delta_{x_i, h_i}^2 V_\alpha$.  
Thus, by the Dirichlet boundary condition, we have 
\[
	\delta_{x_i, h_i}^2 \delta_{x_i, h_i}^2 V_\alpha = \delta_{x_i, h_i; 0}^2 \delta_{x_i, h_i; 0}^2 V_\alpha , 
\]
and it follows that 
\[
	\widetilde{D}_{ii,0} - \widehat{D}_{ii,0} = \frac{h_i^2}{2} \widehat{D}_{ii,0} \widehat{D}_{ii,0} .  
\]
Therefore, the matrix $\widetilde{D}_{ii,0} - \widehat{D}_{ii,0}$ is symmetric positive definite.  

We have proved the following lemma, where the result for 
$\overline{D}_{ij,0} \equiv \frac12 \widehat{D}_{ij,0} + \frac12 \widetilde{D}_{ij,0}$ 
is an immediate consequence:  

\begin{lemma} \label{hessians_lemma}
The matrix $\widetilde{D}_{ij,0} - \widehat{D}_{ij,0}$ is symmetric positive definite for all $i,j \in \{1,2,\ldots,d \}$.  
Furthermore, there holds 
$- \widehat{D}_{ij,0} > - \overline{D}_{ij,0} > - \widetilde{D}_{ij,0}$.  
\end{lemma}

%%%%
\subsection{Discontinuous Galerkin finite element derivative operators}\label{DG_perators}
We note that the above FD numerical derivative operators are only defined on uniform
Cartesian grids. In order to extend them to arbitrary grids, and, in particular, to triangular/tetrahedral grids, 
we utilize finite element DG numerical derivatives 
which were first introduced in  \cite{DG_calc} (also see \cite{FengLewis18}). Below
we recall their definitions and some useful properties, but we shall use new notations which
are consistent with the above FD discrete derivative operators.

%%%%%%%%%%%%%%
\subsubsection{\bf DG mesh and space notations} \label{sec-2_3}
%%%%%%%%%%%%%%

Let $\Omega$ be a polygonal domain, and let $\mathcal{T}_h$ 
denote a locally quasi-uniform shape-regular partition 
of the domain $\Omega$
with $h \equiv \max_{K \in \cT_h} (\text{diam} K)$.
We introduce the broken $H^1$-space and broken $C^0$-space
\[
H^1(\cT_h) \equiv \prod_{K \in \cT_h} H^1(K) , \qquad 
C^0 (\cT_h) \equiv \prod_{K \in \cT_h} C^0 ( \overline{K} )  
\]
and the broken $L^2$-inner product
\[
(v ,w)_{\mathcal{T}_h} \equiv \sum_{K \in \cT_h} \int_{K} v w\, dx \qquad 
\forall v,w \in L^2(\cT_h) .
\]
Let $\cE_h^I$ denote the set of all interior faces/edges of $\cT_h$, 
$\cE_h^B$ denote the set of all boundary faces/edges of $\cT_h$, 
and $\cE_h \equiv \cE_h^I \cup \cE_h^B$.
Then, for a set $\cS_h \subset \cE_h$, we define the broken $L^2$-inner product over $\cS_h$ by 
\[
\langle v ,w \rangle_{\cS_h} \equiv \sum_{e \in \cS_h} \int_{e} v \, w\, ds \qquad 
\forall v,w \in L^2(\cS_h) .
\]
For a fixed integer $r \geq 0$, we define the standard DG finite element space
$V^h \subset H^1 (\cT_h) \subset L^2(\Omega)$ by
\[
V^h \equiv \prod_{K \in \mathcal{T}_h} \poly_{r} (K),
\]
where $ \poly_{r} (K)$ denotes the set of all polynomials on $K$ with
degree not exceeding $r$.  

For $K, K'\in \cT_h$, let $e=\partial K\cap \partial K' \in \cE^I_h$. Without a loss of
generality, we assume that the global labeling number of $K$ is smaller than 
that of $K'$ and define the following (standard) jump and average notations: 
\begin{equation} \label{DG_jump_avg}
[v] \equiv v|_K-v|_{K'} , \qquad 
\{ v \} \equiv \frac{v|_K + v|_{K'}}{2} 
\end{equation}
for any $v\in H^m(\cT_h)$. We also define $\bn_e \equiv \bn_K=-\bn_{K'}$ as the normal vector to $e$. 

Based on the formulation in \cite{FengLewis18}, we will extend the jump and average operators to the boundary of the 
domain in a nonstandard way.  
Let $K \in \cT_h$ such that $e \subset \partial K \in \cE_h^B$, 
and define $\bn_e$ as the unit outward normal for the underlying boundary simplex. 
To unify notation, we impose the convention that the set exterior to the domain $\Omega$ has 
a global labeling number of 0 with the indexing starting at 1 for the ``first" label for the simplices in $\cT_h$.  
Then, we can use the same convention as for interior edges and define 
\begin{equation} \label{DG_jump_avg_boundary}
[v] \equiv v|_{\Omega^c}-v|_{K} , \qquad 
\{ v \} \equiv \frac{v|_{\Omega^c} + v|_{K}}{2} , 
\end{equation}
where $\Omega^c = \mathbb{R}^d \setminus \Omega$.  
Below we will specify how to choose values for $v|_{\Omega^c}$ 
and how to interpret $v|_{K}$ in order to naturally impose a boundary condition.  

Using the jump and average operators for $\cE_h$, we define the labelling-dependent
trace operators $T_i^\pm(v) : \cE_h \to \mathbb{R}$ for each $i=1,2,\ldots,d$ 
for a given function $v \in H^m(\cT_h)$ by 
\begin{equation} \label{trace_dg}
T_i^\pm(v) \equiv \big\{ v \big\} \mp \frac12 \text{sgn} (n_e^{(i)}) \big[ v \big]  \quad
\mbox{where} \quad
\text{sgn}(y) = \begin{cases}
1 & \text{if } y > 0 , \\ 
-1 & \text{if } y < 0 , \\ 
0 & \text{if } y = 0 
\end{cases} 
\end{equation}
for all $y \in \mathbb{R}$ 
and $n_e^{(i)}$ denoting the $i$-th component of $\bn_e$.  
Note that the exact trace values for $v$ along the boundary still need to be specified 
for the jump and average operators.  

Let $K \in \cT_h$ such that $e \subset \partial K \in \cE_h^B$.  
Suppose we have Dirichlet boundary data for the given function $v$, denoted by $g$.  
Then, for $r \geq 1$, we assume $v|_{\Omega^c} = v|_{K} = g$ so that 
$T_i^\pm(v) = g$.  
Such an assumption yields the standard interpretation as introduced in \cite{DG_calc}.
If $r = 0$, we assume $v|_{\Omega^c} = g$ and $v|_{K}$ is given by the interior limit for $v$.  
Thus, $T_i^\pm(v)$ is given by either $g$ or the interior limit for $v$ 
depending on the choice for $\pm$ and the sign of the $i$-th component of the unit normal vector. 
Such a nonstandard approach allows for weighting degrees of freedom associated with the 
Dirichlet data against degrees of freedom associated with the value on the interior of $K$.  
Since $r=0$ implies only one degree of freedom is available on $K$, such a weighting 
is essential to not overly emphasize the boundary condition.  
Notationally, we write $T_i^{\pm,g}$ to denote the natural enforcement of the Dirichlet boundary data $g$. 
When no boundary data is explicitly given and $r \geq 1$, we assume 
$v|_{\Omega^c} = v|_{K}$ for $v|_{K}$ given by the interior limit for $v$. 
If $r = 0$, we always let $v|_{K}$ be given by the interior limit for $v$ 
and have to assign values for $v|_{\Omega^c}$ appropriately.  
Typically we assume $v|_{\Omega^c} = v|_{K}$ or 
$T_i^+(v) = T_i^-(v)$ with the understanding that $v|_{K}$ is given by the interior limit for $v$.
More explicit values for $v|_{\Omega^c}$ can be assigned based on context 
as in \cite{FengLewis18}.  
Notationally, the use of $T_i^{\pm}$ denotes the lack of given Dirichlet data.

\begin{remark} \ 
\begin{itemize}
\item[(a)] 
The trace operators $T_i^{\pm}$ and $T_i^{\pm,g}$ 
are nonstandard in that their values depend on the individual 
components of the edge normal $\bn_e$. 
The standard definition used for LDG assigns a single-value (called
a numerical flux) based on the edge normal vector as a whole.  
\item[(b)]
A labelling-independent definition can also be used so that 
$T_i^{\pm}$ can be associated with the upwind or downwind direction 
with respect to the $x_i$ axis.  
The conventions are equivalent on a uniform Cartesian mesh 
using the natural ordering.  
See \cite{LewisNeilan14} for more details.  
\end{itemize}
\end{remark}

%%%%%%%%
\subsubsection{\bf First order DG derivative operators}

The main idea in \cite{DG_calc} for defining DG derivative operators 
is to use the following local integration by parts formula 
for a given function $v \in H^1(\cT_{\bh}) \cap C^0 (\overline{\Omega})$: 
\begin{equation} \label{DG_greenes}
\int_K v_{x_i} \, \varphi \, dx
= \int_{\partial K} v \, \varphi(x^I) \, n_i \, ds 
- \int_K v \, \varphi_{x_i} \, dx, \qquad i = 1, 2, \ldots, d,\, K\in \cT_h,  
\end{equation}
with test functions $\varphi$ chosen from the DG space $V^h$, 
 $\phi(x^I)$ denoting the limit from the interior of $K$, 
and $n_i$ denoting the $i$-th component of the unit outward normal vector for $K$.  
Thus, the DG (partial in $x_i$) derivative intends to approximate the weak partial derivative $v_{x_i}$ for all 
$v \in H^1 (\cT_h)$. 
To this end, the trace value $v|_{\p K}$ 
must be appropriately chosen/defined when $v$ is not continuous.  

We define DG first order partial derivative operators $\p^\pm_{x_i,h}$ 
for $i=1,2,\cdots, d$ as follows: for $u\in H^1(\cT_h)$, 
\begin{equation} \label{DG_local_q}
\int_K \p^\pm_{x_i,h} u  \, \phi \, dx 
\equiv \int_{\partial K} T_i^\pm (u) \phi(x^I)\, n_{K}^{(i)} \, ds
-\int_K u \, \phi_{x_i} \, dx \qquad\forall \phi \in V^h
\end{equation}
for all $K \in \cT_h$.  
Notice that the ``forward/backward" DG first order derivative operators $\p^\pm_{x_i,h}$ are different
if the values of $T_i^\pm(u)$ are different due to a discontinuity in $u$.  
It is easy to check that $\p^\pm_{x_i, h}$ coincides with the FD operators $\delta^\pm_{x_i, h_i}$ on 
Cartesian grids when using the natural ordering (cf. \cite{DG_calc}). 
Hence, the forward/backward DG derivative operators are indeed generalizations of the forward/backward 
difference operators to unstructured grids. 
When boundary trace data $g$ is known, we define the 
DG first order partial derivative operators $\p^{\pm,g}_{x_i, h}$ that naturally enforce the boundary data by 
\begin{equation} \label{DG_local_q_bc}
\int_K \p^{\pm,g}_{x_i,h} u  \, \phi \, dx 
=\int_{\partial K} T_i^{\pm,g} (u) \phi(x^I)\, n_{K}^{(i)} \, ds
-\int_K u \, \phi_{x_i} \, dx \qquad\forall \phi \in V^h
\end{equation}
for all $K \in \cT_h$ and $i=1,2,\ldots,d$.  

Using the DG first order partial derivative operators as building blocks, we can define various 
central DG first order derivative operators 
and corresponding DG finite element gradient operators.  
Let $i \in \{1,2,\ldots,d\}$.  
Then, we define 
\[
\p_{x_i,h} \equiv \frac12 \Bigl(\p^{+}_{x_i,h}   +  \p^{-}_{x_i,h}  \Bigr) 
\]
as a generalization of the central difference operator $\overline{\delta}_{x_i, h_i}$.  
If boundary data $g$ is given, then we define the following two central 
DG first derivative operators that naturally enforce the boundary condition: 
\[
	\p_{x_i,h}^g \equiv \frac12 \Bigl(\p^{+,g}_{x_i,h}   +  \p^{-,g}_{x_i,h}  \Bigr) , \qquad 
	\overline{\p}_{x_i,h}^g \equiv \frac12 \Bigl(\p^{g}_{x_i,h}  +  \p_{x_i,h}  \Bigr) . 
\]
The first operator $\p_{x_i,h}^g$ naturally generalizes the form of the central difference operator $\delta_{x_i, h_i}$ 
when acting on a grid function with known boundary values, 
and the operator $\overline{\p}_{x_i,h}^g$ generalizes the central difference operator $\delta_{x_i, h_i}$ 
in the sense that both correspond to antisymmetric matrices when vectorized with $g=0$ 
(see \cite{FengLewisRapp21} for the motivation for $\overline{\p}_{x_i,h}^g$).  
In general, we use $\p_{x_i,h}^g$ when $r=0$ and $\overline{\p}_{x_i,h}^g$ when $r \geq 1$ 
to naturally enforce boundary conditions while appropriately weighting interior degrees of freedom 
versus fixed boundary data.  
Corresponding finite element gradient operators $\nabla_h^\pm$, $\nabla_h$, $\nabla_h^{\pm,g}$, $\nabla_h^{g}$, 
and $\overline{\nabla}_h^{g}$ are naturally defined by letting all components 
be given by the appropriate DG first order partial derivative operator.  
For example, 
$\nabla^\pm_{h} \equiv \bigl(\p^\pm_{x_1,h},\p^\pm_{x_2,h}, \cdots, \p^\pm_{x_d,h} \bigr)^T$.

%%%%%%%%%%%
\subsubsection{\bf Second order DG derivative operators}
Similar to the finite difference (and to the classical calculus) construction, 
using first order DG derivative operators as the building blocks, we can easily 
define their high order extensions. Below we only define the second order operators. 
We also only define the operators that will be used directly in our 
framework for approximating fully nonlinear elliptic equations.  
We will only consider the case when boundary conditions correspond to Dirichlet boundary data.  
More information about Neumann boundary data can be found in 
\cite{FengLewisWise2015,DG_calc}.  

We first define the following one-sided second order DG partial derivatives:
\begin{equation} \label{DG_2d_partials}
\p_{x_ix_j, h}^{\mu\nu} \equiv \p_{x_j,h}^\nu \p_{x_i,h}^\mu, 
\qquad 
\p_{x_ix_j, h}^{\mu\nu,g} \equiv \p_{x_j,h}^\nu \p_{x_i,h}^{\mu,g}
\qquad \mu,\nu\in \{+,-\} , 
\end{equation}
where $g$ corresponds to given Dirichlet boundary data.  
Then we define the eight ``sided" $d \times d$ matrix-valued DG Hessian operators
\begin{equation} \label{DG_Hessian}
D_h^{\mu\nu} \equiv \bigl[\p_{x_ix_j,h}^{\mu \nu} \bigr]_{i,j=1}^d, 
\qquad D_h^{\mu\nu,g} \equiv \bigl[\p_{x_ix_j,h}^{\mu \nu,g} \bigr]_{i,j=1}^d,
\qquad \mu,\nu\in \{+,-\}, 
\end{equation}
and the six central 
$d \times d$ matrix-valued DG Hessian operators
\begin{subequations} \label{DG_Hessianb}
\begin{alignat}{2}
\widehat{D}_h^2 & \equiv \frac12 \left( D_h^{+-} + D_h^{-+} \right), 
\qquad && \widehat{D}_h^{2,g}  \equiv \frac12 \left( D_h^{+-,g} + D_h^{-+,g} \right) , \\ 
\widetilde{D}_h^2 & \equiv \frac12 \left( D_h^{++} + D_h^{--} \right), 
\qquad && \widetilde{D}_h^{2,g}  \equiv \frac12 \left( D_h^{++,g} + D_h^{--,g} \right) , \\ 
\overline{D}_h^2 & \equiv \frac12 \left( \widehat{D}_h^{2} + \widetilde{D}_h^{2} \right), 
\qquad && \overline{D}_h^{2,g} \equiv \frac12 \left( \widehat{D}_h^{2,g} + \widetilde{D}_h^{2,g} \right) 
\end{alignat}
\end{subequations}
that can be used when assuming the underlying method has reduced form as introduced below.  

\begin{remark}
 It can be shown (\cite{DG_calc}) that the above 
second order DG operators coincide with their corresponding FD operators on Cartesian grids. 
%Hence, the DG operators are indeed generalizations of the FD operators to unstructured grids. 
Moreover, it is easy to see that all of the DG operators defined above can be applied to 
any piecewise ``nice" functions on $\cT_h$ including those in $V^h$. 
\end{remark}

%%%%%%%%%%%%%%% 

\section{A narrow-stencil and g-monotone numerical framework} \label{framework_sec}

In this section we formulate a general framework for both FD and DG methods 
that can be used to approximate fully nonlinear elliptic boundary value problems 
using narrow-stencil methods.  
We first introduce the ideas using FD methods.  
We then provide examples and extend the ideas to DG methods.  

\subsection{A narrow-stencil FD framework} 

The narrow-stencil FD schemes that we consider will all correspond to seeking 
a grid function $U_\alpha: \NJ'\to \mathbb{R}$ such that 
\begin{subequations}\label{FD_method}
\begin{alignat}{2} 
\hF[U_\alpha, \bx_\alpha] & =0  &&\qquad\mbox{for } \mathbf{x}_\alpha\in\mathcal{T}_{\mathbf{h}} \cap \Omega, \label{FD_method:1}\\
U_\alpha &= g(\mathbf{x}_\alpha)  &&\qquad\mbox{for } \mathbf{x}_\alpha\in\mathcal{T}_{\mathbf{h}}\cap  \partial\Omega , \label{FD_method:2} \\ 
\Delta_{\mathbf{h}} U_\alpha & = 0 
&& \qquad \mbox{for } \mathbf{x}_\alpha\in \mathcal{S}_{\mathbf{h}} \subset \mathcal{T}_{\mathbf{h}} \cap \partial\Omega 
\label{FD_auxiliaryBC} 
\end{alignat}
\end{subequations}
for all $\alpha \in \NJ$, 
where
\begin{align}\label{hatF}
\hF[U_\alpha, \bx_\alpha] &\equiv \widehat{F}\bigl(D_{\mathbf{h}}^{--}U_\alpha, 
D_{\mathbf{h}}^{-+} U_\alpha, D_{\mathbf{h}}^{+-}U_\alpha, 
D_{\mathbf{h}}^{++} U_\alpha, \overline{\nabla}_{\bh} U_\alpha , U_\alpha,\bx_\alpha \bigr) . 
\end{align}
Since the schemes only depend upon the discrete Hessian operators $D_{\bh}^{\mu \nu}$ 
and the discrete gradient operator $\overline{\nabla}_{\bh}$, 
they are inherently narrow-stencil.  
The multiple Hessian operators are used to avoid the directional resolution approach used 
for monotone schemes that often lead to the use of wide-stencils.  
We also note that the auxiliary boundary condition \eqref{FD_auxiliaryBC} needed to define the ghost points that arise when 
calculating $D_{\bh}^{\pm \pm} U_\alpha$ for nodes $\bx_\alpha$ near the boundary could 
be generalized to setting $\Delta_{\mathbf{h}} U_\alpha = h(\bx_\alpha)$  
for some bounded function $h$.  
A well chosen $h$ can increase the accuracy of the underlying scheme by removing any boundary layer error 
associated with the auxiliary boundary condition.  
Such an $h$ can be chosen using a refining process by solving various iterations of \eqref{FD_method} 
with increasingly better chosen functions $h$ based on the previous iteration.  

The main goal for this paper is to define sufficient conditions that $\hF$ can satisfy 
in order to guarantee the scheme \eqref{FD_method} is admissible and convergent.  
The following definitions are adapted from the 1D definitions presented in \cite{FengKaoLewis13}.  

\smallskip
\begin{definition}\label{framework_def} \
	
\begin{itemize} 
\item[{\rm (i)}] A function 
$\hF: \left( \mathbb{R}^{d \times d} \right)^4 \times \mathbb{R}^d \times \mathbb{R} \times \Omega \to \mathbb{R}$ 
is called  a {\em numerical operator}. 
\item[{\rm (ii)}] A numerical operator $\hF$ is said to be {\em consistent} (with 
the differential operator $F$) if $\hF$ satisfies
\begin{align*}%\label{A1a}
\liminf_{P^{\mu \nu} \to P; \mu, \nu = -, + \atop \mathbf{q} \to \mathbf{v}, \lambda \to v, \xi \to \bx} 
\hF(P^{- -}, P^{- +}, P^{+ -}, P^{+ +} , \mathbf{q}, \lambda, \xi) \geq F_*(P,\mathbf{v}, v,\bx),\\
\limsup_{P^{\mu \nu} \to P; \mu, \nu = -, + \atop \mathbf{q} \to \mathbf{v}, \lambda \to v, \xi \to \bx} 
\hF(P^{- -}, P^{- +}, P^{+ -}, P^{+ +} , \mathbf{q}, \lambda, \xi) \leq F_*(P,\mathbf{v},v,\bx), %\label{A1b}
\end{align*}
where $F_*$ and $F^*$ denote, respectively, the lower and upper
semi-continuous envelopes of $F$.

\item[{\rm (iii)}] A numerical operator $\hF$ is said to be
{\em generalized-monotone or g-monotone} if there holds 
\[
\hF(A^{++}, B^{+-}, B^{-+}, A^{--}, \mathbf{q}, v, \bx) \leq \hF(B^{++}, A^{+-}, A^{-+}, B^{--}, \mathbf{q}, w, \bx) 
\]
for all $A^{\mu \nu},B^{\mu \nu} \in \mathbb{R}^{d \times d}$; $\mathbf{q} \in \mathbb{R}^d$; 
$v,w \in \mathbb{R}$; $\bx \in \Omega$ 
such that $B^{\mu \nu} \succeq A^{\mu \nu}$ and $w \geq v$ 
for all $\mu, \nu \in \{ + , - \}$.  
A numerical operator $\hF$ is said to be {\em uniformly g-monotone} if 
there exists a constant $\kappa_* > 0$ such that 
$\hF(P^{++}, P^{+-}, P^{-+}, P^{--}, \mathbf{q}, v, \bx)$ is increasing in 
$P^{++}$, $P^{--}$, and $v$ at a rate bounded below by $\kappa_*$ 
and decreasing in $P^{+-}$ and $P^{-+}$ at a rate bounded above by $-\kappa_*$ 
using the partial ordering imposed by $\succeq$.  

\item[{\rm (iv)}] 
A numerical operator $\hF$ can be written in {\em reduced form} if there exists a function 
$\widehat{G} : \left( \mathbb{R}^{d \times d} \right)^2 \times \mathbb{R} \times \Omega \to \mathbb{R}$
such that 
\[
	\hF(P^{- -}, P^{- +}, P^{+ -}, P^{+ +} , \mathbf{v} , v, \bx) 
	= \widehat{G} ( \widetilde{P}, \widehat{P} , \mathbf{v} , v, \bx) 
\]
for $\widetilde{P} \equiv \frac12 (P^{--} + P^{++})$ and $\widehat{P} \equiv \frac12 (P^{-+} + P^{+-})$ 
for all $P^{- -}$, $P^{- +}$, $P^{+ -}$, $P^{+ +} \in \mathbb{R}^{d \times d}$;
$\mathbf{v} \in \mathbb{R}^d$; $v \in \mathbb{R}$; and $\bx \in \Omega$.  
 
\end{itemize}
\end{definition}

\begin{remark} \
\begin{itemize}
\item[(a)] 
When $F$ and $\widehat{F}$ are continuous, the definition of consistency can be simplified to
$\hF(P, P, P, P ,\mathbf{v},v, \bx) = F(P,\mathbf{v},v,\bx)$
for all $P \in \mathbb{R}^{d \times d}$, $\mathbf{v} \in \mathbb{R}^d$, $v \in \mathbb{R}$, and $\bx \in \Omega$.  
\item[(b)] 
When $\hF$ is differentiable, g-monotonicity can be defined by requiring that the
matrices $\frac{\partial \hF}{\partial P^{- -}}$ and $\frac{\partial \hF}{\partial P^{+ +}}$ 
have all nonnegative entries, 
the matrices 
$\frac{\partial \hF}{\partial P^{- +}}$ and $\frac{\partial F}{\partial P^{+ -}}$ 
have all nonpositive entries, 
and $\frac{\partial \hF}{\partial v}$ is nonnegative.  
In other words, 
$\hF(\uparrow,\downarrow,\downarrow, \uparrow, \cdot, \uparrow,\cdot)$.
For a uniformly g-monotone numerical operator, 
$\frac{\partial \hF}{\partial P^{- -}} \succeq \kappa_* \mathbf{1}_{d \times d}$, 
$\frac{\partial \hF}{\partial P^{+ +}} \succeq \kappa_* \mathbf{1}_{d \times d}$, 
and $\frac{\partial \hF}{\partial v} \geq \kappa_*$ 
while 
$-\frac{\partial \hF}{\partial P^{+ -}} \succeq \kappa_* \mathbf{1}_{d \times d}$ and 
$-\frac{\partial \hF}{\partial P^{- +}} \succeq \kappa_* \mathbf{1}_{d \times d}$, 
where $\mathbf{1}_{d \times d}$ denotes the matrix with all components equal to 1.  
\item[(c)] 
The g-monotonicity approach for narrow-stencil methods 
uses a component partial ordering instead of the SPD partial ordering 
for symmetric matrices.  
The approach also directly compares high order differences instead of looking directly at function values 
as in the standard monotonicity approach of Barles and Souganidis making it easier 
to design g-monotone schemes for a wide class of problems.  
The consistency of $\hF$ with $F$ will allow the scheme to also take advantage of the 
SPD partial ordering associated with a proper elliptic operator.  
\item[(d)] 
To simplify notation, we will assume $\hF$ can be written in reduced form and write 
$\hF ( \widetilde{P}, \widehat{P} , \mathbf{v}, v, \bx)$ 
instead of introducing the new function $\widehat{G}$.  
\end{itemize}
\end{remark}

A key tool for designing the g-monotone numerical operators in Section~\ref{examples_sec} 
is the introduction of a numerical moment as 
defined in \cite{FengLewis21}: 

\begin{definition} \label{moment_def} 
Let $A : \mathbb{R}^J \times \cT_{\mathbf{h}} \to \mathbb{R}^{d \times d}$ 
and $V$ be a given grid function.  
The discrete operator $M : \mathbb{R}^J \to \mathbb{R}$ defined by 
\[
M[V,\bx_\alpha] \equiv  A \bigl( V_\alpha , \mathbf{x}_\alpha \bigr) 
: \bigl( \widetilde{D}_{\mathbf{h}}^2 V_\alpha - \widehat{D}_{\mathbf{h}}^2 V_\alpha \bigr)
\] 
for all $\bx_\alpha \in \cT_{\bh} \cap \Omega$ 
is called a {\em numerical moment operator}.   
\end{definition}

\subsection{Examples of g-monotone FD methods} \label{examples_sec}

We now introduce particular examples of g-monotone FD methods that 
fulfill the structure assumptions of the narrow-stencil framework.  
The first method is the Lax-Friedrichs-like method proposed in \cite{FengLewis21} 
that uses both a numerical moment and a numerical viscosity 
(where the g-monotone definition could be extended for multiple discrete gradient arguments).  
The (general) method is defined by 
\begin{align*}%\label{hatF}
\hF[U_\alpha, \bx_\alpha] 
&\equiv F \left( \overline{D}_{\mathbf{h}}^2 U_\alpha , \overline{\nabla}_{\mathbf{h}} U_\alpha , U_\alpha , \bx_\alpha \right) 
	+ A(U_\alpha, \bx_\alpha) : 
		\bigl( \widetilde{D}_{\mathbf{h}}^2 U_\alpha - \widehat{D}_{\mathbf{h}}^2 U_\alpha \bigr)  \\ 
	\nonumber &\qquad 
	 	- \vec{\beta}(U_\alpha , \bx_\alpha) \cdot 
	 \bigl( \nabla_{\mathbf{h}}^+ U_\alpha - \nabla_{\mathbf{h}}^- U_\alpha \bigr), 
\end{align*}
where $\vec{\beta} : \mathbb{R}^J \times \cT_{\bh} \to \mathbb{R}^d$ is a vector-valued function 
and $- \vec{\beta}(U_\alpha , \bx_\alpha) \cdot 
	 \bigl( \nabla_{\mathbf{h}}^+ U_\alpha - \nabla_{\mathbf{h}}^- U_\alpha \bigr)$ 
is called a numerical viscosity.  
Note that the method is (globally) g-monotone and consistent using the framework above for 
the particular choices $\vec{\beta} = \vec{0}$ and 
$A = \sigma \mathbf{1}_{d \times d}$ for the constant $\sigma > K/2$, 
where 
$K$ denotes the global Lipschitz constant of $F$ with respect to the Hessian argument 
and $\mathbf{1}_{d \times d}$ denotes the matrix with all entries equal to one.  
The choices $\vec{\beta} = \vec{0}$ and $A = \sigma \mathbf{1}_{d \times d}$ for $\sigma > K/2$ 
were the focus in the admissibility, stability, and convergence analysis in \cite{FengLewis21}.  

In Section \ref{experiments_sec} we test the performance of the consistent 
FD method corresponding to the choice 
\newpage
\begin{align}\label{Fgamma}
	\hF_{\gamma,\sigma} [U_\alpha, \bx_\alpha] 
	& = F \left( \overline{D}_{\bh}^2 U_\alpha , \overline{\nabla}_{\bh} U_\alpha , U_\alpha , \bx_\alpha \right) \\
	&\nonumber\qquad 
		+ \left( M_\alpha + \gamma I_{d \times d} + \sigma \mathbf{1}_{d \times d} \right) 
			: \left( \widetilde{D}_{\bh}^2 U_\alpha - \widehat{D}_{\bh}^2 U_\alpha \right) 
\end{align}
for $\sigma \geq 0$ and $\gamma + \sigma \geq 0$,  
where 
\[
	\left[ M_\alpha \right]_{ij} \equiv \frac12 \left| \frac{\partial F}{\partial P_{ij}} \right|_{\left( \overline{D}_{\bh}^2 U_\alpha , \overline{\nabla}_{\bh} U_\alpha , U_\alpha , \bx_\alpha \right) }
\]
for all $i,j = 1,2,\ldots,d$ 
using the convention that $F = F(P, \mathbf{v}, v, \bx)$ for 
$P \in \mathbb{R}^{d \times d}$, $\mathbf{v} \in \mathbb{R}^d$, $v \in \mathbb{R}$, and $\bx \in \Omega$. 
By construction, the method is only locally g-monotone in the sense that the 
linearization of the method is g-monotone. 
Furthermore, for $\sigma > 0$, the method is only locally uniformly g-monotone.  
The admissibility proof in Section~\ref{admissible_sec} when applied to $\hF_{\gamma,\sigma}$ 
holds for $\gamma \geq 0$ and $\sigma = 0$ and $\gamma \geq - \sigma$ with $\sigma \geq 0$ 
if the problem is uniformly elliptic.  
If the operator $F$ is only degenerate elliptic but globally Lipschitz, 
the proof would hold for either $\gamma > 0$ or $\sigma > 0$.  

Consider the linear problem $A:D^2 u = f$. 
Then, the method $\hF_{0,0}$ is equivalent to 
\[
	\hF_{0,0} [U_\alpha, \bx_\alpha] 
	= F \left( D_{\bh}^2 U_\alpha , \bx_\alpha \right) = A : D_{\bh}^2 U_\alpha - f(\bx_\alpha)
\]
for the discrete Hessian $D_{\bh}^2$ defined by  
\begin{equation} \label{Dhlin}
	\left[ D_{\bh}^2 U_\alpha \right]_{ij} 
	= \begin{cases}
	\widetilde{\delta}_{x_i, x_j; h_i, h_j}^2 U_\alpha & \text{if } a_{ij}(\bx_\alpha) > 0 , \\ 
	\widehat{\delta}_{x_i, x_j; h_i, h_j}^2 U_\alpha & \text{if } a_{ij}(\bx_\alpha) \leq 0  
	\end{cases}
\end{equation}
for all $i,j \in \{1,2,\ldots,d\}$ and $\bx_\alpha \in \cT_{\bh} \cap \Omega$.  
Thus, the method takes the form of an upwinding-type method where, instead of matching the 
choice of the discrete partial derivative approximation to the advection field, 
we match the choice of the discrete second-order partial derivative approximation to the sign 
of the corresponding diffusion coefficient in $A$.  
For $-A$ symmetric nonnegative definite, we would have $a_{ii} \leq 0$ for all $i=1,2,\ldots,d$.  
Consequently, when $\gamma = \sigma = 0$, the method would only have a nine-point stencil instead of a 13-point stencil 
in two-dimensions, 
and the auxiliary boundary condition $\Delta_{\bh} U_\alpha = 0$ would not be required.  
Similarly, the choice $\gamma = - \sigma$ would also only have a nine-point stencil and would not 
require the auxiliary boundary condition.  
We would further reduce the stencil to only seven points if the diffusion coefficient $a_{12}$ has a fixed sign.  
Thus, the methods based on choosing $\gamma = \sigma = 0$ 
or $\gamma = - \sigma$ are of particular interest since   
the choices $\gamma = - \sigma$ or $\gamma \geq 0$ with $\sigma = 0$ represent 
limiting choices for enforcing the g-monotonicity of a numerical operator $\hF$  
while remaining consistent with the PDE operator $F$.  

\subsection{A narrow-stencil DG framework} 

We can also naturally formulate narrow-stencil and g-monotone DG methods by seeking 
a piecewise polynomial function $u_h \in V^h$ such that 
\begin{align} \label{DG_method}
\Bigl( \hF[u_h] , \varphi_h \Bigr)_{\cT_h} 
+ \gamma^B \sum_{e \in \cE^B} \frac{1}{h_e} \Big\langle u_h - g , \varphi_h  \Big\rangle_e 
+ \gamma^I \sum_{e \in \cE^I} \frac{1}{h_e} \Big\langle [ u_h ] , [ \varphi_h ] \Big\rangle_e
= 0 
\end{align}
for all $\varphi_h \in V^h$, where 
$\gamma^B, \gamma^I \geq 0$ and 
\[
	\hF[u_h] = \widehat{F}\bigl(D_{h}^{--,g} u_h , D_{h}^{-+,g} u_h , D_{h}^{+-,g} u_h , D_{h}^{++,g} u_h , \overline{\nabla}_h^g u_h , u_h , \cdot \bigr) 
\]
is the same as the numerical operator used for FD methods but is now evaluated using DG derivatives.  
Notationally, $h_e$ denotes the diameter of $e$ 
and $u_h, \varphi_h$ are evaluated using interior limits over $e \in \cE^B$ 
when $\gamma^B > 0$.  
When $r=0$, we use $\nabla_h^g$ instead of $\overline{\nabla}_h^g$ 
to approximate the gradient operator since the corresponding trace operators 
naturally weight exterior limits versus interior limits when assuming only the exterior limit corresponds to $g$.  
We also set $\gamma^B = \gamma^I = 0$ when $r=0$ to ensure consistency with the DG method 
and the underlying FD method on uniform Cartesian grids.  

Observe that the term $\Bigl( \hF[u_h] , \varphi_h \Bigr)_{\cT_h}$ corresponds to projecting the 
numerical operator $\hF[u_h]$ into the discrete space $V^h$ using an $L^2$ projection.  
For a quasi-uniform mesh, 
the penalty terms can be controlled using the uniform ellipticity assumption for $F$ 
and the properties of the DWDG method for approximating Poisson's equation derived in \cite{LewisNeilan14}.  
Consequently, we can choose $\gamma^B = \gamma^I = 0$ even when $r \geq 1$.  
As such, the formulation for DG methods requires projecting the FD formulation into the discrete space 
and optionally adding penalization.  
We note that the auxiliary boundary condition is not required for $r \geq 1$ based on the definitions 
of the DG derivative operators and, in particular, the way in which the boundary trace operators are defined.  
For $r=0$, explicit rules for defining the exterior values for the boundary trace operators 
are provided in \cite{FengLewis18}, and they are consistent with the Dirichlet boundary data 
and the auxiliary boundary condition \eqref{FD_auxiliaryBC}. 
Letting $q_i^\pm = \partial_{x_i, h}^{\pm,g} u_h$, the difficulty addressed in 
\cite{FengLewis18} is how to define $q_i^{\pm} \big|_{\Omega^c}$, which can be thought of as 
defining ghost points for the partial derivative with respect to $x_i$ when $n_i \neq 0$.  
 We refer the reader to \cite{FengLewis18} for the complete formulation when $r=0$.  

\begin{remark} \
\begin{enumerate}
\item[(a)]
By construction, the proposed DG methods can be considered ``narrow-stencil."   
\item[(b)]
When utilizing $\nabla_h^g$ instead of $\overline{\nabla}_h^g$, 
the DG method \eqref{DG_method} is equivalent to the nonstandard LDG methods in \cite{FengLewis18} written 
in a compact form using the DG finite element calculus.  
For the unified framework we utilize $\overline{\nabla}_h^g$ to more closely mimic the antisymmetric 
property of the FD operator $\delta_{x_i, h_i}$ as inspired by \cite{FengLewisRapp21} 
where DG methods were formulated for approximating stationary Hamilton-Jacobi equations. 
\item[(c)]
The DG method \eqref{DG_method} is equivalent to the FD method \eqref{FD_method} 
when $\cT_h$ is a uniform Cartesian mesh and the natural ordering is used.  
As such, all of the analytical results in Sections~\ref{convergence_sec}, \ref{admissible_sec}, and \ref{stability_sec} 
can be extended to \eqref{DG_method} in this special case 
while, in general, the DG approach formally allows for 
higher degree bases and more general meshes.  
\end{enumerate}
\end{remark}

%%%%%%%%%%%%%%%%%%%%%%%%%%%%%%%%%%%%%%%%%%%%%%%%

\section{Convergence analysis} \label{convergence_sec}

In this section we prove that consistent, stable, and g-monotone methods converge to the 
underlying viscosity solution of \eqref{bvp}.  
Similar to the proof in \cite{FengKaoLewis13}, the result will assume the numerical operator 
can be written in reduced form.  
We also use the definition in \cite{FengLewis21} that defines  
a piecewise constant extension $u_{\mathbf{h}}$ for a given grid function 
$U \in S(\cT_{\mathbf{h}}')$ 
by 
\begin{equation} \label{FD_extension_nd}
	u_{\mathbf{h}} (\mathbf{x}) \equiv 
	U_\alpha , \qquad \mathbf{x} \in B_\alpha  
\end{equation} 
for all $\alpha \in \mathbb{N}_J'$, 
where 
$B_\alpha \equiv \prod_{i=1,2,\ldots,d} 
		\big( \mathbf{x}_\alpha-\frac{h_i}2  { \mathbf{e}_i} , \mathbf{x}_\alpha+ \frac{h_i}2 { \mathbf{e}_i}  \big] $
for all $\mathbf{x} \in \Omega' \equiv \cup_{\alpha \in \mathbb{N}_J'} B_\alpha \supset \overline{\Omega}$.    

For transparency, we will  only explicitly consider operators $F$ that have the form 
\[
	F[u](\bx) = F \left( D^2 u, u, \bx \right) 
\] 
in \eqref{FD_pde}.  
We note that the proof can be readily extended to the more general case 
$F[u](\bx) = F \left( D^2 u, \nabla u , u, \bx \right)$ 
using the techniques in \cite{FengLewis21}.  
Indeed, in the proof below, 
we would have $\overline{\nabla}_{\bh_{\bk}} u_{\bh_{\bk}} (\bz_{\bk}) \to \nabla \varphi(\bx_0)$ in {\em Case {\rm (i)}} 
which exploits the consistency of the scheme
and \eqref{conv_eq2} 
could be rewritten as 
\begin{align*}
	& F_* \bigl(\widehat{D}_{\bh_{\bk}}^2 u_{\bh_{\bk}}(\bx_{\bk}) , \overline{\nabla}_{\bh_{\bk}} u_{\bh_{\bk}}(\bx_{\bk}) , u_{\bh_{\bk}}(\bx_{\bk}) ,\bx_{\bk} \bigr) 
			- F_* \left( D^2 \varphi(\bx_0) , \nabla \varphi(\bx_0) , \varphi(\bx_0) , \bx_{\bk} \right) \\
	\nonumber& \qquad \geq 
		- \lambda \delta_{x_{\ell}, h_{\ell}^{(k_{\ell})}}^2 u_{\bh_{\bk}}(\bx_{\bk}) 
		- \lambda \left| \varphi_{x_\ell, x_\ell}(\bx_0) \right| 
		- K \left(  \left| u_{\bh_{\bk}}(\bx_{\bk}) \right| + \left| \varphi(\bx_0) \right| \right) \\ 
		\nonumber& \qquad \qquad
		- K \sum_{i=1}^d \left( \left| \left[ \nabla_{\bh_{\bk}}^+ u_{\bh_{\bk}}(\bx_{\bk}) \right]_{i} \right| 
			+ \left| \left[ \nabla_{\bh_{\bk}}^- u_{\bh_{\bk}}(\bx_{\bk}) \right]_{i} \right| 
			+ \left| \varphi_{x_i}(\bx_0) \right| \right) \\ 
		\nonumber& \qquad \qquad
		- K \sum_{i=1}^d \sum_{j=1 \atop (i,j) \neq (\ell,\ell)}^d \left( 
			\left| \left[ \widehat{D}_{\bh_{\bk}}^2 u_{\bh_{\bk}}(\bx_{\bk}) \right]_{ij} \right| 
			+ \left| \varphi_{x_i x_j}(\bx_0) \right| \right)  
\end{align*}
so that the fact
\[
	f_\ell \left( h_\ell^{([\bk_0]_\ell)} \right) 
	\sum_{i=1}^d \left( \left| \left[ \nabla_{\bh_{\bk}}^+ u_{\bh_{\bk}}(\bx_{\bk}) \right]_{i} \right| 
			+ \left| \left[ \nabla_{\bh_{\bk}}^- u_{\bh_{\bk}}(\bx_{\bk}) \right]_{i} \right| 
			+ \left| \varphi_{x_i}(\bx_0) \right| \right) 
	\to 0 
\]
for $\min \bk$ sufficiently large and $k_\ell \to \infty$ 
can be exploited in {\em Case {\rm (ii)}}.

\begin{theorem} \label{FD_convergence_nd}
Suppose the operator $F$ in \eqref{bvp} is proper and uniformly elliptic 
with $\lambda > 0$, 
$g$ is continuous on $\partial \Omega$, 
$F$ is Lipschitz continuous with respect to its first two arguments, 
and \eqref{bvp} satisfies the comparison principle.  
Suppose $\widehat{F}$ is consistent, is uniformly g-monotone, and can be written in reduced form, 
and suppose that $\widehat{F}$ is Lipschitz continuous with respect to its first three arguments 
when written in reduced form.  
Let $U\in S(\cT_{\mathbf{h}}')$ be the solution 
to the scheme \eqref{FD_method}, and let $u_{\mathbf{h}}$ denote the piecewise constant extension of 
$U$ defined  by \eqref{FD_extension_nd}.  
If \eqref{FD_method} is admissible and $\ell^\infty$-norm stable, 
then $u_{\mathbf{h}}$ converges to $u$ locally uniformly
as $\mathbf{h} \to  \mathbf{0}^+$.  
\end{theorem}

\begin{proof}
The following is a sketch of the proof that highlights the differences from the complete convergence 
proof for the Lax-Friedrich's-like method given in \cite{FengLewis21}.   

\medskip 
{\em Step 1}: 
Since the underlying FD scheme is assumed to be $\ell^\infty$-norm stable,
there exists a constant $C > 0$ such that 
$\left\| u_{\mathbf{h}} \right\|_{L^\infty(\Omega)} \leq C$ 
independent of $\bh$.   
Define the upper and lower semicontinuous functions $\ou$ and $\uu$ by 
\[
	\ou(\mathbf{x})\equiv\limsup_{ \bh \to \mathbf{0}^+ \atop \xi \to \bx} u_{\bh}(\xi) , 
	\qquad 
	 \uu(\mathbf{x})\equiv\liminf_{ \bh \to \mathbf{0}^+ \atop \xi \to \bx}  u_{\bh}(\xi) , 
\]
where the limits are understood as multi-limits.  
We show $\ou$ is a viscosity subsolution of \eqref{bvp}.  
The proof that $\uu$ is a viscosity supersolution of \eqref{bvp} is analogous.  
By the comparison principle, we have $\ou = \uu$, and 
it follows that $u = \ou = \uu$ is the viscosity solution of \eqref{bvp}.  

Let $\varphi\in C^2(\oOme)$ be a quadratic polynomial  
such that $\ou-\varphi$ takes a strict local maximum at $\mathbf{x}_0\in \overline{\Omega}$ 
with $\ou(\mathbf{x}_0)=\varphi(\mathbf{x}_0)$.  
Then there exists a ball, 
$B_{r_0}(\mathbf{x}_0) \subset \mathbb{R}^d$, 
centered at $\mathbf{x}_0$ with radius $r_0>0$ (in the $\ell^\infty$ metric) 
such that
\begin{equation}\label{e3.9}
\ou(\mathbf{x})-\varphi(\mathbf{x}) < \ou(\mathbf{x}_0)-\varphi(\mathbf{x}_0)= 0  
\qquad\forall \mathbf{x}\in \left( B_{r_0}(\mathbf{x}_0) \cap \overline{\Omega} \right) \setminus \{ \bx_0 \}.
\end{equation}

Suppose $\mathbf{x}_0 \in \Omega$.  
We show that 
\begin{equation}  \label{step3a} 
F_* \bigl( D^2 \varphi(\mathbf{x}_0) , \varphi(\mathbf{x}_0) , 
		\mathbf{x}_0 \bigr) \leq 0 
\end{equation} 
based on various cases determined by the regularity of $\ou$ at $\bx_0$.  
Note that if $\mathbf{x}_0 \in \partial \Omega$, then, by the argument in \cite{FengLewis21}, 
$\ou$ can be shown to satisfy the boundary condition \eqref{FD_bc} 
in the viscosity sense.  

By the definition of $\ou$ and \eqref{e3.9}, there exists (maximizing) 
sequences $\{\bh_{\bk}\}$, $\{\bx_{\bk}\}$, and $\{\bz_{\bk}\}$ 
and a constant $K_0 > 0$ 
such that 
\begin{subequations}\label{max_seq}
\begin{align}
	& \bh_{\bk} \to \mathbf{0}^+ , \\ 
	& \bx_{\bk} \to \bx_0 \text{ with } \bx_{\bk} \in \cT_{\bh_{\bk}} , \\ 
	& u_{\bh_{\bk}}(\bx_{\bk}) \to \ou(\bx_0) , \\ 
	& \bz_{\bk} \to \bx_0  \text{ with } \left| x_i^{(k_i)} - z_i^{(k_i)} \right| \leq \frac12 h_i^{(k_i)} 
		\text{ and } u_{\bh_{\bk}}(\bz_{\bk}) = u_{\bh_{\bk}}(\bx_{\bk}), \\
	& u_{\bh_{\bk}}(\bz) - \varphi(\bz) \text{ is locally maximized at } 
		\bz = \bz_{\bk} \text{ for all } \min \bk \geq K_0 . \label{max_seq_e}
\end{align}
\end{subequations}
Let $H^{(\bk)} \in \mathbb{R}^{d \times d}$ be defined by $H^{(\bk)} = D^2 u^{(\bk)}(\bz_{\bk})$ using the 
convention in \cite{FengLewis21} to define the local interpolation functions $u^{(\bk)}$.  

\smallskip 
\underline{\em Case {\rm (i)}}: $\{ H^{(\bk)} \}$ has a uniformly bounded subsequence. 
In this case, there exists a symmetric matrix $H \in \mathbb{R}^{d \times d}$ 
and a subsequence (not relabeled) such that $H^{(\bk)} \to H$, 
$\widetilde{D}_{\bh_{\bk}}^2 u_{\bh_{\bk}}(\bz_{\bk}) \to H$,   
and $\widehat{D}_{\bh_{\bk}}^2 u_{\bh_{\bk}}(\bz_{\bk}) \to H$ 
with $D^2 \varphi - H$ symmetric positive semidefinite (see \cite{FengLewis21} for details).  
Thus,  
\begin{align*}
	0 & = \lim_{\min \bk \to \infty} 
		\hF \left[ u_{\bh_{\bk}} , \bx_{\bk} \right] \\ 
	& = \lim_{\min \bk \to \infty} 
		\hF \left[ u_{\bh_{\bk}} , \bz_{\bk} \right] \\ 
	& = \lim_{\min \bk \to \infty} \widehat{F}\bigl(\widetilde{D}_{\bh_{\bk}}^2 u_{\bh_{\bk}}(\bz_{\bk}) , 
\widehat{D}_{\bh_{\bk}}^2 u_{\bh_{\bk}}(\bz_{\bk}) , u_{\bh_{\bk}}(\bz_{\bk}) ,\bz_{\bk} \bigr)  \\ 
	& \geq F_* \left( H , \varphi(\bx_0) , \bx_0 \right) \\ 
	& \geq F_* \left( D^2 \varphi(\bx_0) , \varphi(\bx_0) , \bx_0 \right) 
\end{align*}
by the consistency of the scheme and the ellipticity of $F$.

\smallskip 
\underline{\em Case {\rm (ii)}}: $\{ H^{(\bk)} \}$ does not have a uniformly bounded subsequence 
and there is no set of local interpolation functions $\widetilde{u}_{\bh_{\bk}}$ such that the sequence 
$\widetilde{H}^{(\bk)}$ has a bounded subsequence (see \cite{FengLewis21} for the definition of $\widetilde{u}_{\bh_{\bk}}$).  
If such functions $\widetilde{u}_{\bh_{\bk}}$ exist, then the argument in Case (i) can be easily updated 
to show $F_* \left( D^2 \varphi(\bx_0) , \varphi(\bx_0) , \bx_0 \right) \leq 0$.  

There exists a pair of indices $(i,j)$ such that the sequence 
$[ \widetilde{D}_{\bh_{\bk}}^2 u_{\bh_{\bk}}(\bz_{\bk}) ]_{ij}$ 
or $[ \widehat{D}_{\bh_{\bk}}^2 u_{\bh_{\bk}}(\bz_{\bk}) ]_{ij}$
does not have a bounded subsequence.  
Thus, by \eqref{max_seq_e}, there exists an index $\ell \in \{i,j\}$ and a subsequence such that 
$\delta_{\xi_i^j, \bh_{\bk}}^2 u_{\bh_{\bk}}(\bz_{\bk}) \to - \infty$,  
$\delta_{\eta_i^j, \bh_{\bk}}^2 u_{\bh_{\bk}}(\bz_{\bk}) \to - \infty$, 
$\delta_{x_\ell, 2 h_\ell^{(k_\ell)}}^2 u_{\bh_{\bk}}(\bz_{\bk}) \to - \infty$, 
or $\delta_{x_\ell, h_\ell^{(k_\ell)}}^2 u_{\bh_{\bk}}(\bz_{\bk}) \to - \infty$ 
using the notation in \cite{FengLewis21} to rewrite the 
components of $[ \widetilde{D}_{\bh_{\bk}}^2 u_{\bh_{\bk}}(\bz_{\bk}) ]_{ij}$ 
and $[ \widehat{D}_{\bh_{\bk}}^2 u_{\bh_{\bk}}(\bz_{\bk}) ]_{ij}$ 
in terms of central difference operators.   
Since $\hF$ is uniformly g-monotone, 
$\hF$ must be uniformly increasing with respect to 
$\delta_{\xi_i^j, \bh_{\bk}}^2 u_{\bh_{\bk}}(\bz_{\bk}) \to - \infty$,  
$\delta_{\eta_i^j, \bh_{\bk}}^2 u_{\bh_{\bk}}(\bz_{\bk}) \to - \infty$, 
or $\delta_{x_\ell, 2 h_\ell^{(k_\ell)}}^2 u_{\bh_{\bk}}(\bz_{\bk}) \to - \infty$.  
By sending $k_i \to \infty$ and $k_j \to \infty$ while ensuring $\min \bk$ is sufficiently large,  
if there holds $\delta_{\xi_i^j, \bh_{\bk}}^2 u_{\bh_{\bk}}(\bz_{\bk}) \to - \infty$,  
$\delta_{\eta_i^j, \bh_{\bk}}^2 u_{\bh_{\bk}}(\bz_{\bk}) \to - \infty$, 
or $\delta_{x_\ell, 2 h_\ell^{(k_\ell)}}^2 u_{\bh_{\bk}}(\bz_{\bk}) \to - \infty$, 
then there must hold $\delta_{x_\ell, h_\ell^{(k_\ell)}}^2 u_{\bh_{\bk}}(\bz_{\bk}) \to - \infty$ 
to ensure $\hF [ u_{\bh_{\bk}} , \bx_{\bk} ] = 0$ for all $\bk$.  
Therefore, there exists an index $\ell$ such that the sequence 
$\delta_{x_{\ell}, h_{\ell}^{(k_{\ell})}}^2 u_{\bh_{\bk}}(\bz_{\bk})
= \delta_{x_{\ell}, h_{\ell}^{(k_{\ell})}}^2 u_{\bh_{\bk}}(\bx_{\bk})$ 
does not have a bounded subsequence as $\min \bk \to \infty$.  

Choose sequences $\{ \bh_{\bk} \}$, $\{ \bx_{\bk} \}$ that maximize the 
rate at which $\delta_{x_{\ell}, h_{\ell}^{(k_{\ell})}}^2 u_{\bh_{\bk}}(\bx_{\bk}) \to - \infty$.  
By the definition of the scheme, we have 
\begin{align}\label{conv_eq1}
	0 & = \hF [ u_{\bh_{\bk}} , \bx_{\bk} ] \\ 
	\nonumber 
	& = \hF \bigl(\widetilde{D}_{\bh_{\bk}}^2 u_{\bh_{\bk}}(\bx_{\bk}) , 
			\widehat{D}_{\bh_{\bk}}^2 u_{\bh_{\bk}}(\bx_{\bk}) , u_{\bh_{\bk}}(\bx_{\bk}) ,\bx_{\bk} \bigr) \\ 
	\nonumber
	& = F_* \left( D^2 \varphi(\bx_0) , \varphi(\bx_0) , \bx_{\bk} \right) \\ 
		\nonumber& \qquad 
		+ F_* \bigl(\widehat{D}_{\bh_{\bk}}^2 u_{\bh_{\bk}}(\bx_{\bk}) , u_{\bh_{\bk}}(\bx_{\bk}) ,\bx_{\bk} \bigr) 
			- F_* \left( D^2 \varphi(\bx_0) , \varphi(\bx_0) , \bx_{\bk} \right) \\ 
		\nonumber& \qquad 
		+ \hF \bigl(\widehat{D}_{\bh_{\bk}}^2 u_{\bh_{\bk}}(\bx_{\bk}) , 
			\widehat{D}_{\bh_{\bk}}^2 u_{\bh_{\bk}}(\bx_{\bk}) , u_{\bh_{\bk}}(\bx_{\bk}) ,\bx_{\bk} \bigr) 
		- F_* \bigl(\widehat{D}_{\bh_{\bk}}^2 u_{\bh_{\bk}}(\bx_{\bk}) , u_{\bh_{\bk}}(\bx_{\bk}) ,\bx_{\bk} \bigr) \\ 
		\nonumber& \qquad 
		+ \hF \bigl(\widetilde{D}_{\bh_{\bk}}^2 u_{\bh_{\bk}}(\bx_{\bk}) , 
			\widehat{D}_{\bh_{\bk}}^2 u_{\bh_{\bk}}(\bx_{\bk}) , u_{\bh_{\bk}}(\bx_{\bk}) ,\bx_{\bk} \bigr) \\
			\nonumber& \qquad 
		- \hF \bigl(\widehat{D}_{\bh_{\bk}}^2 u_{\bh_{\bk}}(\bx_{\bk}) , 
			\widehat{D}_{\bh_{\bk}}^2 u_{\bh_{\bk}}(\bx_{\bk}) , u_{\bh_{\bk}}(\bx_{\bk}) ,\bx_{\bk} \bigr) . 
\end{align}
Then, by the mean value theorem, the Lipschitz continuity of $F$, and the uniform and proper ellipticity of $F$, 
there exists a constant $K \geq 0$ such that 
\begin{align} \label{conv_eq2}
	& F_* \bigl(\widehat{D}_{\bh_{\bk}}^2 u_{\bh_{\bk}}(\bx_{\bk}) , u_{\bh_{\bk}}(\bx_{\bk}) ,\bx_{\bk} \bigr) 
			- F_* \left( D^2 \varphi(\bx_0) , \varphi(\bx_0) , \bx_{\bk} \right) \\
	\nonumber& \qquad \geq 
		- \lambda \left( \delta_{x_{\ell}, h_{\ell}^{(k_{\ell})}}^2 u_{\bh_{\bk}}(\bx_{\bk}) - \varphi_{x_\ell, x_\ell}(\bx_0) \right) 
		- K \left(  \left| u_{\bh_{\bk}}(\bx_{\bk}) \right| + \left| \varphi(\bx_0) \right| \right) \\ 
		\nonumber& \qquad \qquad
		- K \sum_{i=1}^d \sum_{j=1 \atop (i,j) \neq (\ell,\ell)}^d \left( 
			\left| \left[ \widehat{D}_{\bh_{\bk}}^2 u_{\bh_{\bk}}(\bx_{\bk}) \right]_{ij} \right| 
			+ \left| \varphi_{x_i x_j}(\bx_0) \right| \right) \\ 
	\nonumber& \qquad \geq 
		- \lambda \delta_{x_{\ell}, h_{\ell}^{(k_{\ell})}}^2 u_{\bh_{\bk}}(\bx_{\bk}) 
		- \lambda \left| \varphi_{x_\ell, x_\ell}(\bx_0) \right| 
		- K \left(  \left| u_{\bh_{\bk}}(\bx_{\bk}) \right| + \left| \varphi(\bx_0) \right| \right) \\ 
		\nonumber& \qquad \qquad
		- K \sum_{i=1}^d \sum_{j=1 \atop (i,j) \neq (\ell,\ell)}^d \left( 
			\left| \left[ \widehat{D}_{\bh_{\bk}}^2 u_{\bh_{\bk}}(\bx_{\bk}) \right]_{ij} \right| 
			+ \left| \varphi_{x_i x_j}(\bx_0) \right| \right)  . 
\end{align}
%Choose $\epsilon > 0$.  
Using the consistency of the scheme, there holds 
\begin{align} \label{conv_eq3}
	& \hF \bigl(\widehat{D}_{\bh_{\bk}}^2 u_{\bh_{\bk}}(\bx_{\bk}) , 
			\widehat{D}_{\bh_{\bk}}^2 u_{\bh_{\bk}}(\bx_{\bk}) , u_{\bh_{\bk}}(\bx_{\bk}) ,\bx_{\bk} \bigr) 
		- F_* \bigl(\widehat{D}_{\bh_{\bk}}^2 u_{\bh_{\bk}}(\bx_{\bk}) , u_{\bh_{\bk}}(\bx_{\bk}) ,\bx_{\bk} \bigr) \\
	\nonumber & \qquad \geq 0 .  
\end{align}
%for all $\min{\bk}$ sufficiently large and fixed.  
Lastly, by the mean value theorem, the Lipschitz continuity of $\hF$, 
and the g-monotonicity of $\hF$, there exists a constant $\widehat{K} \geq 0$ and 
a sequence $0 \leq a_{\bk} \leq K$ such that 
\begin{align} \label{conv_eq4}
	& \hF \bigl(\widetilde{D}_{\bh_{\bk}}^2 u_{\bh_{\bk}}(\bx_{\bk}) , 
			\widehat{D}_{\bh_{\bk}}^2 u_{\bh_{\bk}}(\bx_{\bk}) , u_{\bh_{\bk}}(\bx_{\bk}) ,\bx_{\bk} \bigr) \\ 
		\nonumber & \qquad - \hF \bigl(\widehat{D}_{\bh_{\bk}}^2 u_{\bh_{\bk}}(\bx_{\bk}) , 
			\widehat{D}_{\bh_{\bk}}^2 u_{\bh_{\bk}}(\bx_{\bk}) , u_{\bh_{\bk}}(\bx_{\bk}) ,\bx_{\bk} \bigr) \\
	\nonumber & \geq  
		a_{\bk} \left[ \widetilde{D}_{\bh_{\bk}}^2 u_{\bh_{\bk}}(\bx_{\bk}) 
				- \widehat{D}_{\bh_{\bk}}^2 u_{\bh_{\bk}}(\bx_{\bk}) \right]_{\ell \ell} \\ 
		\nonumber & \qquad \qquad 
		- \widehat{K} \sum_{i=1}^d \sum_{j=1 \atop (i,j) \neq (\ell,\ell)}^d \left( 
			\left| \left[ \widetilde{D}_{\bh_{\bk}}^2 u_{\bh_{\bk}}(\bx_{\bk}) \right]_{ij} \right| 
			+ \left| \left[ \widehat{D}_{\bh_{\bk}}^2 u_{\bh_{\bk}}(\bx_{\bk}) \right]_{ij} \right| \right) . 
\end{align}
Plugging \eqref{conv_eq2}, \eqref{conv_eq3}, and \eqref{conv_eq4} into \eqref{conv_eq1}, 
it follows that 
\begin{align}\label{conv_eq}
	0 
	& \geq F_* \left( D^2 \varphi(\bx_0) , \varphi(\bx_0) , \bx_{\bk} \right) \\ 
		\nonumber& \qquad 
		- \lambda \delta_{x_{\ell}, h_{\ell}^{(k_{\ell})}}^2 u_{\bh_{\bk}}(\bx_{\bk}) 
		+ a_{\bk} \left[ \widetilde{D}_{\bh_{\bk}}^2 u_{\bh_{\bk}}(\bx_{\bk}) 
				- \widehat{D}_{\bh_{\bk}}^2 u_{\bh_{\bk}}(\bx_{\bk}) \right]_{\ell \ell} \\ 
		\nonumber & \qquad 
		- \lambda | \varphi_{x_\ell, x_\ell}(\bx_0) | 
		- K \left(  \left| u_{\bh_{\bk}}(\bx_{\bk}) \right| + \left| \varphi(\bx_0) \right| \right) \\ 
		\nonumber& \qquad 
		- K \sum_{i=1}^d \sum_{j=1 \atop (i,j) \neq (\ell,\ell)}^d \left( 
			\left| \left[ \widehat{D}_{\bh_{\bk}}^2 u_{\bh_{\bk}}(\bx_{\bk}) \right]_{ij} \right| 
			+ \left| \varphi_{x_i x_j}(\bx_0) \right| \right) \\ 
		\nonumber & \qquad  
		- \widehat{K} \sum_{i=1}^d \sum_{j=1 \atop (i,j) \neq (\ell,\ell)}^d \left( 
			\left| \left[ \widetilde{D}_{\bh_{\bk}}^2 u_{\bh_{\bk}}(\bx_{\bk}) \right]_{ij} \right| 
			+ \left| \left[ \widehat{D}_{\bh_{\bk}}^2 u_{\bh_{\bk}}(\bx_{\bk}) \right]_{ij} \right| \right) . 
\end{align}

Choose the corresponding optimal function $f_\ell$ (defined in \cite{FengLewis21}) and subsequences 
such that  
\begin{equation} \label{rate_uh}
	\lim_{\min \bk \to \infty} f_\ell \left(h_\ell^{(k_\ell)} \right) \delta_{x_{k_\ell}, h_\ell^{(k_\ell)}}^2 u_{\bh_{\bk}} (\bx_{\bk}) 
	= - C_\ell 
\end{equation}
for some constant $C_\ell > 0$.  
Then, by \cite{FengLewis21}, there holds 
\begin{align*}%\label{momentii} 
	\liminf_{\min \bk \to \infty} f_\ell \left(h_\ell^{(k_\ell)} \right) 
		\left[ \widetilde{D}_{\bh_{\bk}}^2 u_{\bh_{\bk}}(\bx_{\bk}) 
			- \widehat{D}_{\bh_{\bk}}^2 u_{\bh_{\bk}}(\bx_{\bk}) \right]_{\ell \ell} 
	\geq 0 
\end{align*} 
implying 
\begin{align}\label{momentii} 
	f_\ell \left(h_\ell^{(k_\ell)} \right) 
		a_{\bk} \left[ \widetilde{D}_{\bh_{\bk}}^2 u_{\bh_{\bk}}(\bx_{\bk}) 
			- \widehat{D}_{\bh_{\bk}}^2 u_{\bh_{\bk}}(\bx_{\bk}) \right]_{\ell \ell} 
	\geq - \frac{\lambda}{4} C_\ell  
\end{align} 
for all $\min \bk$ sufficiently large.  
Furthermore, by combining the observations in Subcases iia, iib, and iic in the proof of Theorem 6.1 in \cite{FengLewis21}, 
there exists subsequences such that, 
for $\min \bk$ sufficiently large and $k_\ell >> \max k_j$ for all $j \neq \ell$ 
%if $\limsup_{k_\ell \to \infty} \frac{f_\ell \left( h_\ell^{(k_\ell)} \right) }{h_\ell^{(k_\ell)} } < \infty$ 
%(see subcases iia and iib) 
%or $k_j$ chosen to so that $c_h \leq \frac{h_j^{(k)}}{h_\ell^{(k)}} \leq C_h$ for some constants $c_h, C_h$ 
%for a single-indexed subsequence of $\bh^{(\bk)}$ 
%if $\lim_{k_\ell \to \infty} \frac{f_\ell \left( h_\ell^{(k_\ell)} \right) }{h_\ell^{(k_\ell)} } = \infty$
%(see subcase iic), 
there holds 
\begin{align} \label{subseq_bound}
	\frac{\lambda}{4} C_\ell 
	& \geq f_\ell \left( h_\ell^{(k_\ell)} \right) 
	\bigg( \lambda | \varphi_{x_\ell, x_\ell}(\bx_0) | 
		+ K \left(  \left| u_{\bh_{\bk}}(\bx_{\bk}) \right| + \left| \varphi(\bx_0) \right| \right) \bigg) \\ 
		\nonumber& \qquad 
		+ K f_\ell \left( h_\ell^{(k_\ell)} \right) \sum_{i=1}^d \sum_{j=1 \atop (i,j) \neq (\ell,\ell)}^d \left( 
			\left| \left[ \widehat{D}_{\bh_{\bk}}^2 u_{\bh_{\bk}}(\bx_{\bk}) \right]_{ij} \right| 
			+ \left| \varphi_{x_i x_j}(\bx_0) \right| \right)  \\ 
		\nonumber & \qquad  
		+ \widehat{K} f_\ell \left( h_\ell^{(k_\ell)} \right) \sum_{i=1}^d \sum_{j=1 \atop (i,j) \neq (\ell,\ell)}^d \left( 
			\left| \left[ \widetilde{D}_{\bh_{\bk}}^2 u_{\bh_{\bk}}(\bx_{\bk}) \right]_{ij} \right| 
			+ \left| \left[ \widehat{D}_{\bh_{\bk}}^2 u_{\bh_{\bk}}(\bx_{\bk}) \right]_{ij} \right| \right) . 
\end{align}
(Note that the primary difficulty in showing \eqref{subseq_bound} is 
controlling the contributions of $\left[ \widetilde{D}_{\bh_{\bk}}^2 u_{\bh_{\bk}}(\bx_{\bk}) \right]_{\ell, j}$ 
and $\left[ \widehat{D}_{\bh_{\bk}}^2 u_{\bh_{\bk}}(\bx_{\bk}) \right]_{\ell, j}$ for $j \neq \ell$ 
due to the $\frac{1}{h_\ell^{(k_\ell)}}$ factor when approximating mixed derivatives.)  
Thus, scaling \eqref{conv_eq} by $f_\ell \left( h_\ell^{(k_\ell)} \right)$ 
and plugging in \eqref{rate_uh}, \eqref{momentii}, and \eqref{subseq_bound}, 
there exists an index $\bk_0$ with $\min \bk_0$ sufficiently large such that 
\begin{align}\label{conv_eq_neg}
	0 & \geq f_\ell \left( h_\ell^{([\bk_0]_\ell)} \right) F_* \left( D^2 \varphi(\bx_0) , \varphi(\bx_0) , \bx_{\bk_0} \right) 
		+ \frac{3 \lambda}{4} C_\ell  
		- \frac{\lambda}{4} C_\ell 
		- \frac{\lambda}{4} C_\ell \\ 
	\nonumber 
	& > f_\ell \left( h_\ell^{([\bk_0]_\ell)} \right) F_* \left( D^2 \varphi(\bx_0) , \varphi(\bx_0) , \bx_{\bk_0} \right) . 
\end{align}
The bound $0 \geq F_*[\varphi](\bx_0)$ follows since $f_\ell \left( h_\ell^{([\bk_0]_\ell)} \right) > 0$ and $\bx_{\bk} \to \bx_0$.  
Hence, \eqref{step3a} has been verified.  

The remainder of the proof is identical to Steps 4-6 in \cite{FengLewis21}, 
and the result follows.  
\end{proof}

\begin{remark} \
\begin{itemize}
\item[(a)] 
g-monotonicity allowed us to identify a sufficiently positive term when $\lambda > 0$.  
Consequently, we could strongly exploit the uniformly elliptic structure of the PDE operator $F$.  

\item[(b)] 
Theorem \ref{FD_convergence_nd} is proved under the assumption that
the numerical scheme is admissible and $\ell^\infty$-norm stable.  
The remainder of the paper verifies sufficient conditions under which the assumptions hold.  
\end{itemize}
\end{remark}

%%%%%%%%%%%%%%%%%%%%%%%%%%%%%%%%%%%%%%%%%%%%%%%%

\section{Admissibility analysis} \label{admissible_sec}

The goal of this section is to show that the proposed narrow stencil scheme \eqref{FD_method}
has a unique solution whenever the numerical operator $\hF$ is consistent, is g-monotone, 
and can be written in reduced form.  
For transparency, we will  only consider operators $F$ that have the form 
$F[u](\bx) = F \left( D^2 u, u, \bx \right)$ 
in \eqref{FD_pde}.  
The proofs can be adapted for the more general case 
using the techniques in \cite{FengLewisRapp21} 
by exploiting the fact that the matrix 
representation of $\overline{\nabla}_{\bh}$ is anti-symmetric.  

The idea for proving the well-posedness 
is to equivalently reformulate the proposed scheme as a fixed point problem 
and to prove the mapping is contractive in the $\ell^2$-norm. 
To this end, let $S(\cT_{\mathbf{h}}^\prime)$
denote the space of all grid functions on $\cT_{\mathbf{h}}^\prime$, and introduce the 
mapping $\cM_\rho : S(\cT_{\mathbf{h}}^\prime) \to S(\cT_{\mathbf{h}}^\prime)$ defined by 
\begin{equation}\label{M_rho}
	\widehat{U}
	\equiv \cM_\rho U, 
\end{equation}
where the grid function $\widehat{U}\in S(\cT_{\mathbf{h}}')$ is defined by 
\begin{subequations} \label{M_rho_matrix}
\begin{alignat}{2}  
	 \widehat{U}_\alpha &= U_\alpha - \rho \widehat{F} [ U_\alpha , \bx_\alpha] , \qquad 
		&& \text{if } \bx_\alpha \in \cT_{\bh} \cap \Omega , \label{M_rho_interior:1} \\ 
	 \widehat{U}_\alpha &= g(\bx_\alpha) , && \text{if } \bx_\alpha \in \cT_{\bh} \cap \partial \Omega , \label{M_rho_interior:2} \\ 
	\Delta_{\bh} \widehat{U}_\alpha & = 0 
&& \text{if } \mathbf{x}_\alpha\in \mathcal{S}_{\bh} 
		\label{M_rho_matrix_bc2}
\end{alignat}
\end{subequations} 
for $\rho > 0$ an undetermined constant.  
Clearly, the iteration defined in \eqref{M_rho_matrix}
is the standard forward Euler method with pseudo time-step $\rho$ 
complemented with a boundary condition consistent with \eqref{FD_method}.   
To show $\cM_\rho$ is a contraction, we will linearize the operator via the mean value theorem.  
As a preliminary result in Section~\ref{linear_sec}, we will first consider a simple case when $F$ 
is linear with constant-valued coefficients and a simple scheme based on $\overline{D}_{\bh}^2$ is used 
to discretize the Hessian. 
The general case will be considered in Section~\ref{admissible_nonlinear_sec}.  

We do require one additional structure assumption on the numerical operator $\hF$ 
to assist in the admissibility and stability proofs.  
The condition will ensure that the method based on using multiple Hessian operators 
is compatible with the uniform ellipticity property of the corresponding PDE problem.  
We first motivate the property before defining it.  
Note that the Lax-Friedrich's-like method in \cite{FengLewis21} and the examples in Section~\ref{examples_sec} 
satisfy the additional structure assumption.  

Suppose that $F = F(P,v,x)$ is uniformly elliptic and differentiable with respect to its first two arguments 
and $\hF = \hF(\widetilde{P}, \widehat{P}, v, x)$ is differentiable with respect to its first three arguments.  
Then, if $\hF$ is consistent with $F$, there holds 
\begin{subequations} \label{consistency_partials}
\begin{align}
	\frac{\partial F}{\partial P} 
	& = \frac{\partial }{\partial P} F(P,v,x) 
	= \frac{\partial }{\partial P} \hF(P,P,v,x) 
	= \frac{\partial \hF}{\partial \widetilde{P}} + \frac{\partial \hF}{\partial \widehat{P}} , \\ 
	\frac{\partial F}{\partial v} 
	& = \frac{\partial }{\partial v} F(P,v,x) 
	= \frac{\partial }{\partial v} \hF(P,P,v,x) 
	= \frac{\partial \hF}{\partial v} 
\end{align}
\end{subequations} 
for all $P \in \mathbb{R}^{d \times d}$, $v \in \mathbb{R}$, and $x \in \Omega$.  
Let $A, B \in \mathbb{R}^{d \times d}$, and suppose $P = \frac12 A + \frac12 B$.  
Then, for the Lax-Friedrich's-like method in \cite{FengLewis21} 
where 
\[
	\hF(A,B,v,x) = F \left( \frac12 A + \frac12 B , v , x \right) + \gamma \mathbf{1}_{d \times d} : (A-B) 
\]
for $\gamma$ sufficiently large, 
there holds 
\begin{align*}
	& \frac{\partial}{\partial \widetilde{P}} \hF(A,B,v,x) + \frac{\partial}{\partial \widehat{P}} \hF(A,B,v,x) \\ 
	& \qquad = \frac12 \frac{\partial}{\partial P} F(P,v,x) + \gamma \mathbf{1}_{d \times d} + \frac12 \frac{\partial}{\partial P} F(P,v,x) - \gamma \mathbf{1}_{d \times d} \\ 
	& \qquad = \frac{\partial}{\partial P} F(P,v,x) \leq - \lambda I , 
\end{align*}
and for the g-monotone method 
\[
	\hF(A,B,v,x) = F \left( B , v , x \right) + \gamma \mathbf{1}_{d \times d} : (A-B) , 
\]
there holds 
\begin{align*}
	\frac{\partial}{\partial \widetilde{P}} \hF(A,B,v,x) + \frac{\partial}{\partial \widehat{P}} \hF(A,B,v,x) 
	& = \gamma \mathbf{1}_{d \times d} + \frac{\partial}{\partial P} F(B,v,x) - \gamma \mathbf{1}_{d \times d} \\ 
	& = \frac{\partial}{\partial P} F(B,v,x) \leq - \lambda I .  
\end{align*}
Note that $\frac{\partial}{\partial P} F(B,v,x)$ may not equal $\frac{\partial}{\partial P} F(P,v,x)$ for $P \neq B$.  
However, the same uniform ellipticity bound holds.  
Similarly, for the g-monotone method 
\[
	\hF(A,B,v,x) = F \left( A , v , x \right) + \gamma \mathbf{1}_{d \times d} : (A-B) , 
\]
there holds 
\begin{align*}
	\frac{\partial}{\partial \widetilde{P}} \hF(A,B,v,x) + \frac{\partial}{\partial \widehat{P}} \hF(A,B,v,x) 
	& =  \frac{\partial}{\partial P} F(A,v,x) + \gamma \mathbf{1}_{d \times d} - \gamma \mathbf{1}_{d \times d} \\ 
	& = \frac{\partial}{\partial P} F(A,v,x) \leq - \lambda I , 
\end{align*}
and again the same uniform ellipticity bound holds.  
Thus, we assume the following compatibility condition when proving the admissibility and stability 
of our proposed narrow-stencil schemes.  

\begin{definition} \label{compatible_def}
Suppose $F = F(P,v,\bx)$ is proper elliptic and differentiable with respect to its first two arguments 
with 
$\frac{\partial}{\partial P} F(A,w,\bx) \leq - \lambda I$ and $\frac{\partial}{\partial v} F(A,w,\bx) \geq \kappa_0$ 
for all $A \in \mathcal{S}^{d \times d}$, $w \in \mathbb{R}$, and $\bx \in \Omega$.  
Suppose the numerical operator $\hF$ is consistent with $F$ and differentiable with respect to its first three arguments.  
The numerical operator $\hF = \hF(\widetilde{P},\widehat{P},v,\bx)$ is {\em elliptic compatible} if 
there exists a constant $c > 0$ independent of $\bh$ such that 
\[
	\frac{\partial}{\partial \widetilde{P}} \hF(A, B,w,\bx) + \frac{\partial}{\partial \widehat{P}} \hF(A, B,w,\bx) \leq - c \lambda I
\] 
and $\frac{\partial}{\partial v} \hF(A, B,w,\bx) \geq c \kappa_0$ 
for all $A, B \in \mathcal{S}^{d \times d}$, $w \in \mathbb{R}$, and $\bx \in \Omega$.  

\end{definition}

%%%%%%%%%%%%%%%%%%%%%%%%%%%%%%%%%%%%%%%%%%%%%%%%
%%%%%%%%%%%%%%%%%%%%%%%%%%%%%%%%%%%%%%%%%%%%%%%%

\subsection{Admissibility of a simple method for linear, constant coefficient PDEs} \label{linear_sec}

Consider the linear elliptic boundary value problem 
\begin{subequations} \label{bvp_linear}
\begin{alignat}{2}
	\mathcal{L}[u] \equiv -A:D^2 u = -\sum_{i=1}^d \sum_{j=1}^d a_{ij} u_{x_i x_j} & = f \qquad && \text{in } \Omega , \\ 
	u & = g \qquad && \text{on } \partial \Omega , 
\end{alignat}
\end{subequations} 
where 
$A$ is constant-valued and symmetric positive definite, 
and consider the simple FD scheme
corresponding to finding  
a grid function $U_\alpha: \NJ' \to \mathbb{R}$ such that 
\begin{subequations}\label{FD_method_linear}
\begin{alignat}{2} 
\mathcal{L}_{\bh} U_\alpha \equiv - A : \overline{D}_{\bh}^2 U_\alpha & = f(\bx_\alpha)  &&\qquad\mbox{for } \mathbf{x}_\alpha\in\mathcal{T}_{\mathbf{h}} \cap \Omega, \label{FD_method_linear:1}\\
U_\alpha &= g(\mathbf{x}_\alpha)  &&\qquad\mbox{for } \mathbf{x}_\alpha\in\mathcal{T}_{\mathbf{h}}\cap  \partial\Omega , \label{FD_method_linear:2} \\ 
\Delta_{\mathbf{h}} U_\alpha & = 0 
&& \qquad \mbox{for } \mathbf{x}_\alpha\in \mathcal{S}_{\mathbf{h}} \subset \mathcal{T}_{\mathbf{h}} \cap \partial\Omega 
\label{FD_method_linear:3} 
\end{alignat}
\end{subequations}
for all $\alpha \in \NJ$. 
We show that \eqref{FD_method_linear} is equivalent to solving a linear system $L \vec{U} = \mathbf{b}$ with the matrix 
$L$ symmetric positive definite.  

Let $\lambda_0 > 0$ denote the smallest eigenvalue of $A$.  
Define $A_0 \equiv A - \lambda_0 I$.  
Then, $A_0$ is symmetric nonnegative definite.  
Thus, there exists an eigenvalue decomposition 
$A_0 = Q \Lambda Q^T = \sum_{k=1}^d \lambda_k \mathbf{q}_k \mathbf{q}_k^T$, 
where $\lambda_k \geq 0$ and $\{ \mathbf{q}_k \}$ forms an orthonormal basis for $\mathbb{R}^d$.  
Define $\mathbf{q}_k = \sum_{i=1}^d b_i^{(k)} \mathbf{e}_i$, and observe that 
\begin{align*}
	\mathbf{q}_k \mathbf{q}_k^T 
	& = \sum_{i=1}^d \left( b_i^{(k)} \right)^2 \mathbf{e}_i \mathbf{e}_i^T 
		+ \sum_{i=1}^d \sum_{j = 1 \atop j \neq i}^d b_i^{(k)} b_j^{(k)} \mathbf{e}_i \mathbf{e}_j^T . 
\end{align*}
Thus, 
\begin{align*}
	A_0 & = \sum_{k=1}^d \lambda_k \left[ \sum_{i=1}^d \left( b_i^{(k)} \right)^2 \mathbf{e}_i \mathbf{e}_i^T 
		+ \sum_{i=1}^d \sum_{j = 1 \atop j \neq i}^d b_i^{(k)} b_j^{(i)} \mathbf{e}_i \mathbf{e}_j^T \right] \\ 
	& = \sum_{i=1}^d \left[ \sum_{k=1}^d \lambda_k \left( b_i^{(k)} \right)^2 \right] \mathbf{e}_i \mathbf{e}_i^T 
		+ \sum_{i=1}^d \sum_{j = 1 \atop j \neq i}^d \left[ \sum_{k=1}^d \lambda_k b_i^{(k)} b_j^{(k)} \right] \mathbf{e}_i \mathbf{e}_j^T , 
\end{align*}
and it follows that 
\begin{align*}
	[A_0]_{ii} = a_{ii} - \lambda_0 & = \sum_{k=1}^d \lambda_k \left( b_i^{(k)} \right)^2 , \qquad
	[A_0]_{ij} = a_{ij} = \sum_{k=1}^d \lambda_k b_i^{(k)} b_j^{(k)} 
\end{align*}
for all $i,j = 1,2,\ldots,d$ with $j \neq i$.  

Let  $J_0 = | \cT_{\bh} \cap \Omega |$, 
$D_i \in \mathbb{R}^{J_0 \times J_0}$ denote the matrix representation of $\overline{\delta}_{x_i, h_i}$ 
with the boundary condition \eqref{FD_method_linear:2} 
and $M \in \mathbb{R}^{J_0 \times J_0}$ denote the matrix representation of $-\Delta_{2 \bh}$ 
and the boundary conditions \eqref{FD_method_linear:2} and \eqref{FD_method_linear:3}.  
Then $\left(D_i \right)^T = - D_i$
and there exists diagonal matrices $B_i \in \mathbb{R}^{J_0 \times J_0}$ 
for $i=1,2,\ldots,d$ 
with nonnegative components such that 
\begin{align*}
	M & = - \sum_{i=1}^d \left( D_i D_i - B_i \right) 
	\geq - \sum_{i=1}^d D_i D_i 
	= \sum_{i=1}^d \left(D_i \right)^T D_i .  
	%\geq 0_{J_0 \times J_0} . 
\end{align*}

The positive components of $B_i$ correspond to nodes $\bx_\alpha$ near the boundary 
and increase the coefficients for $U_\alpha$.  
Such corrections are needed in the matrix form to account for the values of $\delta_{x_i, h_i} U_{\alpha'}$ 
when $\bx_{\alpha'} \in \cT_{\bh} \cap \partial \Omega$ 
and for the implementation of the auxiliary boundary condition \eqref{FD_method_linear:3}.  
Indeed, suppose $\bx_\alpha - h_i \mathbf{e}_i \in \partial \Omega$.  
Then 
\begin{align*}
	\delta_{x_i, 2 h_i}^2 U_{\alpha} 
	& = \delta_{x_i, h_i} \delta_{x_i, h_i} U_{\alpha} 
	= \frac{1}{2 h_i} \delta_{x_i, h_i} \left( U_{\alpha+\mathbf{e}_i} - U_{\alpha-\mathbf{e}_i} \right) \\ 
	& = \frac{1}{2 h_i} \delta_{x_i, h_i} U_{\alpha+\mathbf{e}_i} - \frac{1}{2 h_i} \delta_{x_i, h_i} U_{\alpha-\mathbf{e}_i} \\ 
	& = \frac{1}{2 h_i} \delta_{x_i, h_i} U_{\alpha+\mathbf{e}_i} 
		- \frac{1}{4 h_i^2} U_{\alpha} + \frac{1}{4 h_i^2} g(\bx_\alpha - h_i \mathbf{e}_i) . 
\end{align*}
However, when computing using the matrix representation, 
$D_i$ treats the boundary value $U_{\alpha-\mathbf{e}_i}$ as a known value in its representation 
and removes it when calculating $D_i U$.  
Consequently, the second application of $D_i$ does not act on the boundary node leading to a smaller coefficient for the 
adjacent interior node involved in the calculation of $\delta_{x_i, h_i} U_{\alpha-\mathbf{e}_i}$.  
Thus, 
$M$ would contain the contribution $\frac{1}{4 h_i^2}$ to the coefficient for $U_\alpha$ 
while $-D_i D_i$ would not.  

We can see that \eqref{FD_method_linear:3} ensures that the ghost value 
$U_{\alpha \pm 2 \mathbf{e}_i}$ satisfies 
\[
	\frac{1}{2 h_i^2} U_{\alpha \pm 2 \mathbf{e}_i} 
	= - \frac{1}{2 h_i^2} U_\alpha + \frac{1}{h_i^2} g(\bx_{\alpha \pm \mathbf{e}_i})  
		- \sum_{j=1 \atop j \neq i}^d \delta_{x_j, h_j}^2 g(\bx_{\alpha \pm \mathbf{e}_i}) 
\]
for $\bx_{\alpha \pm \mathbf{e}_i} \in \cT_{\bh} \cap \partial \Omega$, 
where $\frac{1}{2 h_i^2} U_{\alpha \pm 2 \mathbf{e}_i}$ is directly involved in the computation of 
$\overline{\delta}_{x_i, h_i}^2 U_\alpha$.  
Lastly, note that correction terms are not needed when considering the relationship 
of $\delta_{x_i, h_i} \delta_{x_j, h_j}$ to $D_i D_j$ for $i \neq j$ since the computation 
of $\delta_{x_i, h_i} U_{\alpha'}$ would only include boundary nodes 
whenever $\bx_\alpha \pm h_j \mathbf{e}_j \in \cT_{\bh} \cap \partial \Omega$.  

Observe that 
\begin{align*}
	- A_0 : \overline{D}_{\bh}^2 
	& = - \sum_{i=1}^d \sum_{j=1}^d [A_0]_{ij} \overline{\delta}_{x_i, h_i} \overline{\delta}_{x_j, h_j} \\ 
	& = - \sum_{i=1}^d (a_{ii} - \lambda_0) \overline{\delta}_{x_i, h_i} \overline{\delta}_{x_i, h_i} 
		- \sum_{i=1}^d \sum_{j = 1 \atop j \neq i}^d a_{ij} 
			\overline{\delta}_{x_i, h_i} \overline{\delta}_{x_j, h_j} \\ 
	& = - \sum_{i=1}^d \sum_{k=1}^d \lambda_k \left( b_i^{(k)} \right)^2 \overline{\delta}_{x_i, h_i} \overline{\delta}_{x_i, h_i} 
		- \sum_{i=1}^d \sum_{j = 1 \atop j \neq i}^d \left[ \sum_{k=1}^d \lambda_k b_i^{(k)} b_j^{(k)} \right] 
			\overline{\delta}_{x_i, h_i} \overline{\delta}_{x_j, h_j} \\ 
	& = - \sum_{k=1}^d \lambda_k \left[ 
		\sum_{i=1}^d \left( b_i^{(k)} \right)^2 \overline{\delta}_{x_i, h_i} \overline{\delta}_{x_i, h_i} 
		+ \sum_{i=1}^d \sum_{j = 1 \atop j \neq i}^d b_i^{(k)} b_j^{(k)} 
			\overline{\delta}_{x_i, h_i} \overline{\delta}_{x_j, h_j} \right] \\ 
	& = - \sum_{k=1}^d \lambda_k \left( \left[ \sum_{i=1}^d b_i^{(k)} \overline{\delta}_{x_i, h_i} \right] 
		\left[ \sum_{i=1}^d b_i^{(k)} \overline{\delta}_{x_i, h_i} \right] \right) , 
\end{align*}
and it follows that  
\begin{align*}
	0_{J_0 \times J_0} & 
	\leq \sum_{k=1}^d \lambda_k \left( \left[ \sum_{i=1}^d b_i^{(k)} D_i \right]^T \left[ \sum_{i=1}^d b_i^{(k)} D_i \right] \right) \\ 
	& = \sum_{i=1}^d \sum_{k=1}^d \lambda_k \left( b_i^{(k)} \right)^2 D_i^T D_i 
		+ \sum_{i=1}^d \sum_{j = 1 \atop j \neq i}^d \left[ \sum_{k=1}^d \lambda_k b_i^{(k)} b_j^{(k)} \right] 
			D_i^T D_j \\ 
	& = \sum_{i=1}^d (a_{ii} - \lambda_0) D_i^T D_i + \sum_{i=1}^d \sum_{j = 1 \atop j \neq i}^d a_{ij} D_i^T D_j 
	= \sum_{i=1}^d \sum_{j = 1}^d [A_0]_{ij} D_i^T D_j . 
\end{align*}
Using the fact that $D_i^T = - D_i$, there holds 
\begin{align} \label{M_0_bar}
	L 
	& \equiv -\sum_{i=1}^d \sum_{j=1}^d a_{ij} D_i D_j + \sum_{i=1}^d a_{ii} B_i  \\ 
	& \nonumber 
	= - \lambda_0 \sum_{i=1}^d \left( D_i^2 - B_i \right) 
		- \sum_{i=1}^d \sum_{j = 1}^d [A_0]_{ij} D_i D_j + \sum_{i=1}^d (a_{ii} - \lambda_0) B_i \\ 
	\nonumber & = \lambda_0 M 
		+ \sum_{i=1}^d \sum_{j = 1}^d [A_0]_{ij} D_i^T D_j  + \sum_{i=1}^d (a_{ii} - \lambda_0) B_i \\ 
	& \nonumber \geq \lambda_0 M > 0_{J_0 \times J_0} . 
\end{align}
Therefore, \eqref{FD_method_linear} has a unique solution 
and the matrix representation $L \vec{U} = \mathbf{b}$ yields a symmetric positive definite matrix $L$.

%%%%%%%%%%%%%%%%%%%%%%%%%%%%%%%%%%%%%%%%%%%%%%%%
%%%%%%%%%%%%%%%%%%%%%%%%%%%%%%%%%%%%%%%%%%%%%%%%

\subsection{Admissibility for fully nonlinear PDEs} \label{admissible_nonlinear_sec}

To show that the mapping $\cM_\rho$ has a unique fixed point in $S(\cT_{\mathbf{h}}^\prime)$, 
we first establish a lemma that specifies conditions under which $\cM_\rho$ is a contraction in $\ell^2$.  
The proof will assume $F$ is differentiable; however, the assumption is for ease of notation 
and the proof can be extended for $F$ Lipschitz but not differentiable.  
The result will utilize the following result found in \cite{FengLewis21}:  

\begin{lemma}\label{lemma_symmetrization}
Let $B, F \in \mathbb{R}^{J \times J}$ such that 
$B$ is symmetric nonnegative definite and $F$ is symmetric positive definite. 
Define $R \in \mathbb{R}^{J \times J}$ such that $R$ is upper triangular and $F = R^* R$.  
Then
\[
	\| \sigma I - FB \|_2 \leq \sigma 
\]
for all positive  constants $\sigma$ such that $\sigma I > R B R^*$.  

\end{lemma}

\begin{lemma} \label{lemma_contraction} 
Suppose the operator $F$ in \eqref{bvp} is proper and uniformly elliptic,  
differentiable, and Lipschitz continuous with respect to its first two arguments. 
Suppose $\widehat{F}$ is consistent, g-monotone, can be written in reduced form, 
and is differentiable with respect to its first three arguments.  
Choose $U,V\in S(\cT_{\mathbf{h}}')$ that satisfy the boundary conditions \eqref{FD_method_linear:2} 
and \eqref{FD_method_linear:3}, 
and let $\widehat{U} = \cM_\rho U$ and $\widehat{V} = \cM_\rho V$ 
for $\cM_\rho$ defined by \eqref{M_rho} and \eqref{M_rho_matrix}. 
Then, for $\hF$ elliptic compatible, there holds 
\[
	\| \widehat{U} - \widehat{V} \|_{\ell^2(\cT_{\mathbf{h}})} 
	\leq \left(1 - \rho \frac{c \lambda \kappa}{4} - \rho \frac{c k_0}{4} \right) \| U - V \|_{\ell^2(\cT_{\mathbf{h}})} 
\] 
for all $\rho > 0$ sufficiently small, 
where $c$ is from the definition of elliptic compatibility, 
$4 > \rho c \lambda \kappa + \rho c k_0$, 
$\frac{\partial F}{\partial D^2 u} \geq \lambda I$, $\frac{\partial F}{\partial u} \geq k_0$, 
and $M \geq \kappa I$ for $M$ the matrix representation of $-\Delta_{2 \bh}$.  
\end{lemma}

\begin{proof}
Let $W\equiv V-U$ and $\widehat{W}\equiv \widehat{V}-\widehat{U}$. 
Then, by the boundary conditions, 
$\widehat{W}_\alpha = W_\alpha = 0$ for all $\bx_\alpha \in \cT_{\bh} \cap \partial \Omega$ 
and $\Delta_{\bh} \widehat{W}_\alpha = \Delta_{\bh} W_\alpha = 0$ for all $\bx_\alpha \in \mathcal{S}_{\bh}$.  
Thus, by the mean value theorem for $\widehat{F} = \hF(\widetilde{P}, \widehat{P}, v, \bx)$, there holds 
\begin{align} \label{V_U} 
	\widehat{W}_\alpha 
	& = W_\alpha - \rho \left( \widehat{F}[V_\alpha, \bx_\alpha] - \widehat{F}[U_\alpha , \bx_\alpha] \right) \\ 
	& \nonumber = W_\alpha - \rho \left( \hF(\widetilde{D}_{\bh}^2 V_\alpha , \widehat{D}_{\bh}^2 V_\alpha , V_\alpha , \bx_\alpha ) - \hF (\widetilde{D}_{\bh}^2 U_\alpha , \widehat{D}_{\bh}^2 U_\alpha , U_\alpha , \bx_\alpha ) \right) \\ 
	& \nonumber = W_\alpha - \rho 
		\hF \left(\overline{D}_{\bh}^2 V_\alpha  + \frac12 \widetilde{D}_{\bh}^2 V_\alpha - \frac12 \widehat{D}_{\bh}^2 V_\alpha , \overline{D}_{\bh}^2 V_\alpha  - \frac12 \widetilde{D}_{\bh}^2 V_\alpha + \frac12 \widehat{D}_{\bh}^2 V_\alpha , V_\alpha , \bx_\alpha \right) \\ 
		& \nonumber \qquad  
		+ \rho \hF \left( \overline{D}_{\bh}^2 U_\alpha  + \frac12 \widetilde{D}_{\bh}^2 U_\alpha - \frac12 \widehat{D}_{\bh}^2 U_\alpha , \overline{D}_{\bh}^2 U_\alpha  - \frac12 \widetilde{D}_{\bh}^2 U_\alpha + \frac12 \widehat{D}_{\bh}^2 U_\alpha , U_\alpha , \bx_\alpha \right) \\ 
	& \nonumber = \left( 1 - \rho \frac{\partial \hF}{\partial v} \right) W_\alpha 
		- \rho \sum_{i=1}^d \sum_{j=1}^d \left( \frac{\partial \hF}{\partial \widetilde{P}_{ij}} + \frac{\partial \hF}{\partial \widehat{P}_{ij}} \right) \overline{\delta}_{x_i, x_j; h_i , h_j}^2 W_\alpha \\ 
		& \nonumber \qquad 
		- \rho \frac12 \sum_{i=1}^d \sum_{j=1}^d \left( \frac{\partial \hF}{\partial \widetilde{P}_{ij}} - \frac{\partial \hF}{\partial \widehat{P}_{ij}} \right) \widetilde{\delta}_{x_i, x_j; h_i , h_j}^2 W_\alpha \\ 
		& \nonumber \qquad 
		+ \rho \frac12 \sum_{i=1}^d \sum_{j=1}^d \left( \frac{\partial \hF}{\partial \widetilde{P}_{ij}} - \frac{\partial \hF}{\partial \widehat{P}_{ij}} \right) \widehat{\delta}_{x_i, x_j; h_i , h_j}^2 W_\alpha 
\end{align} 
for all $\bx_\alpha \in \cT_{\bh} \cap \Omega$.  

Let  $J_0 = | \cT_{\bh} \cap \Omega |$ and 
$\widehat{\mathbf{W}} , \mathbf{W} \in \mathbb{R}^{J_0}$ denote the vectorization of the 
grid functions $\widehat{W}$ and $W$ restricted to $\cT_{\bh} \cap \Omega$, respectively.   
We introduce several matrix operators in $\mathbb{R}^{J_0 \times J_0}$ 
that act on $\mathbf{W}$ and 
correspond to the FD operators with the boundary data naturally incorporated directly into the definition 
of the matrix.  
Then, using the notation of Section~\ref{discrete_Hessian_comparison_sec}, 
$\widetilde{D}_{ij,0}$ and $\widehat{D}_{ij,0}$ denote the matrix representations of $\widetilde{\delta}_{x_i, x_j; h_i , h_j}^2$ 
and $\widehat{\delta}_{x_i, x_j; h_i , h_j}^2$, respectively.  
Similarly, $\overline{D}_{ij,0}$ denotes the matrix representations of $\overline{\delta}_{x_i, x_j; h_i , h_j}^2$.  
We also let $\hF_0$ denote the diagonal matrix 
corresponding to the nodal values of $\frac{\partial \hF}{\partial v}$,  
$\widetilde{F}_{ij}$ denote the diagonal matrix corresponding to 
$\frac{\partial \hF}{\partial \widetilde{P}_{ij}}$, 
and $\widehat{F}_{ij}$ denote the diagonal matrix corresponding to 
$\frac{\partial \hF}{\partial \widehat{P}_{ij}}$.  
Then $\hF_0$, $\widetilde{F}_{ij}$, and $-\widehat{F}_{ij}$ are all nonnegative definite, 
and we have \eqref{V_U} becomes 
\begin{align*}
	\widehat{\mathbf{W}}  
	& = (I - \rho \hF_0) \bW 
		 - \rho \sum_{i=1}^d \sum_{j=1}^d \left( \widetilde{F}_{ij} + \widehat{F}_{ij} \right) \overline{D}_{ij,0} \bW \\ 
		& \qquad \nonumber 
		- \rho \frac12 \sum_{i=1}^d \sum_{j=1}^d \left( \widetilde{F}_{ij} - \widehat{F}_{ij} \right) \widetilde{D}_{ij,0} \bW  
		+ \rho \frac12 \sum_{i=1}^d \sum_{j=1}^d \left( \widetilde{F}_{ij} - \widehat{F}_{ij} \right) \widehat{D}_{ij,0} \bW . 
\end{align*}
Letting $M$ denote the matrix corresponding to $-\Delta_{2 \bh}$, 
it follows that 
\begin{align} \label{V_UmatF}
	\widehat{\mathbf{W}}  
	& = (I - \rho \hF_0) \bW - \rho \frac{1}{2} c \lambda M \bW 
		 - \rho \left[ \sum_{i=1}^d \sum_{j=1}^d \left( \widetilde{F}_{ij} + \widehat{F}_{ij}  \right) \overline{D}_{ij,0} - \frac12 c \lambda M \right] \bW \\ 
		& \qquad \nonumber 
		- \rho \frac12 \sum_{i=1}^d \sum_{j=1}^d \left( \widetilde{F}_{ij} - \widehat{F}_{ij} \right) \left( \widetilde{D}_{ij,0} - \overline{D}_{ij,0} \right) \bW \\ 
		& \qquad \nonumber
		+ \rho \frac12 \sum_{i=1}^d \sum_{j=1}^d \left( \widetilde{F}_{ij} - \widehat{F}_{ij} \right) \left( \widehat{D}_{ij,0} - \overline{D}_{ij,0} \right) \bW . 
\end{align}

Choose $\bx_\alpha \in \cT_{\bh} \cap \Omega$.  
Let $A(\bx_\alpha) \in \mathbb{R}^{d \times d}$ be defined by $A_{ij}(\bx_\alpha) = \frac{\partial \hF}{\partial \widetilde{P}_{ij}} + \frac{\partial \hF}{\partial \widehat{P}_{ij}}$.  
Then, $-A(\bx_\alpha) \geq c \lambda I$ by the elliptic compatible condition.  
Observe that 
\begin{align*}
	& \left( A(\bx_\alpha) - \frac{1}{2} c \lambda I \right) : \overline{D}_{\bh}^2 W_\alpha \\ 
	& \qquad = \sum_{i=1}^d \sum_{j=1}^d \left( \frac{\partial \hF}{\partial \widetilde{P}_{ij}} + \frac{\partial \hF}{\partial \widehat{P}_{ij}} \right) \overline{\delta}_{x_i , x_j ; h_i , h_j}^2 W_\alpha 
		- \frac{1}{2} c \lambda \sum_{i=1}^d \overline{\delta}_{x_i, h_i}^2 W_\alpha \\ 
	& \qquad = \sum_{i=1}^d \sum_{j=1}^d \left( \frac{\partial \hF}{\partial \widetilde{P}_{ij}} + \frac{\partial \hF}{\partial \widehat{P}_{ij}} \right) \overline{\delta}_{x_i , x_j ; h_i , h_j}^2 W_\alpha 
		- \frac{1}{2} c \lambda \Delta_{2 \bh} W_\alpha .  
\end{align*}
Then, the term $\left[ \sum_{i=1}^d \sum_{j=1}^d \left( \widetilde{F}_{ij} + \widehat{F}_{ij}  \right) \overline{D}_{ij,0} - \frac12 c \lambda M \right] \bW$ 
in \eqref{V_UmatF} is the matrix representation of 
$\left( A(\bx_\alpha) - \frac{1}{2} c \lambda I \right) : \overline{D}_{\bh}^2 W_\alpha$.  

We next utilize the frozen coefficient technique.  
Define the matrices $L_\alpha \in \mathbb{R}^{J_0 \times J_0}$ as the matrix representations of 
$\left( A(\bx_\alpha) - \frac{1}{2} c \lambda I \right) : \overline{D}_{\bh}^2$ 
for all $\bx_\alpha \in \cT_{\bh} \cap \Omega$.  
Then, by Section~\ref{linear_sec}, we have $L_\alpha$ is symmetric positive definite for all multi-indices $\alpha$ 
since the coefficient matrix is constant-valued for each fixed value of $\alpha$.  
Define $E_k \in \mathbb{R}^{J_0 \times J_0}$ by $[E_k]_{ij} = 1$ only if $i=j=k$ and $0$ otherwise.  
Notationally, we let $\alpha(k)$ be the multi-index corresponding to the single-index $k$.  
Then, \eqref{V_UmatF} can be rewritten as   
\begin{align*}
	\widehat{\mathbf{W}}  
	& = (I - \rho \hF_0) \bW - \rho \frac{1}{2} c \lambda M \bW 
		- \rho \sum_{k=1}^{J_0} E_k L_{\alpha(k)} \bW \\ 
		& \qquad \nonumber 
		- \rho \frac12 \sum_{i=1}^d \sum_{j=1}^d \left( \widetilde{F}_{ij} - \widehat{F}_{ij} \right) \left( \widetilde{D}_{ij,0} - \overline{D}_{ij,0} \right) \bW \\ 
		& \qquad \nonumber
		- \rho \frac12 \sum_{i=1}^d \sum_{j=1}^d \left( \widetilde{F}_{ij} - \widehat{F}_{ij} \right) \left( \overline{D}_{ij,0} - \widehat{D}_{ij,0} \right) \bW \\ 
	& \equiv G \mathbf{W} 
\end{align*}
for the iteration matrix $G \in \mathbb{R}^{J_0 \times J_0}$ defined by 
\begin{align*}
	G & \equiv (I - \rho \hF_0) - \rho \frac{1}{2} c \lambda M  
		- \rho \sum_{k=1}^{J_0} E_k L_{\alpha(k)} \\ 
		& \qquad \nonumber 
		- \rho \frac12 \sum_{i=1}^d \sum_{j=1}^d \left( \widetilde{F}_{ij} - \widehat{F}_{ij} \right) \left( \widetilde{D}_{ij,0} - \overline{D}_{ij,0} \right) \\ 
		& \qquad \nonumber
		- \rho \frac12 \sum_{i=1}^d \sum_{j=1}^d \left( \widetilde{F}_{ij} - \widehat{F}_{ij} \right) \left( \overline{D}_{ij,0} - \widehat{D}_{ij,0} \right)  .  
\end{align*}

Choose $\epsilon > 0$.  
Let $N = 2 + J_0 + 2d^2$. 
Observe that $E_k + \epsilon I$ and $\widetilde{F}_{ij} - \widehat{F}_{ij} + \epsilon I$ 
are symmetric positive definite for all $k$ and all $i,j$.  
Choose $K$ such that 
$0_{J_0 \times J_0} \leq \sum_{k=1}^{J_0} L_{\alpha(k)} \leq K I$, 
$0_{J_0 \times J_0} \leq \widetilde{D}_{ij,0} - \overline{D}_{ij,0} \leq K I$, 
and $0_{J_0 \times J_0} \leq \overline{D}_{ij,0} - \widehat{D}_{ij,0} \leq K I$ for all $i,j$, 
where the first bound uses the fact that $L_{\alpha(k)}$ is symmetric positive definite for all indices $k$ 
and the other bounds follow from Lemma~\ref{hessians_lemma}.  
Then, by Lemma~\ref{lemma_symmetrization}, there holds 
\begin{align}
	\| G \|_2 
	\nonumber 
	& \leq \frac12 + \left\| \frac{1}{2N} I - \rho \hF_0 \right\|_2 
		+ \left\| \frac{1}{2N} I -  \rho \frac{1}{2} c \lambda M  \right\|_2 \\ 
		& \qquad \nonumber 
		+ \sum_{k=1}^{J_0} \left\| \frac{1}{2N} I - \rho (E_k + \epsilon I ) L_{\alpha(k)} \right\|_2 
		+ \rho \epsilon \left\| \sum_{k=1}^{J_0} L_{\alpha(k)} \right\|_2 \\ 
		& \qquad \nonumber 
		+ \frac12 \sum_{i=1}^d \sum_{j=1}^d \left\| \frac{1}{N} I - \rho \left( \widetilde{F}_{ij} - \widehat{F}_{ij} + \epsilon I \right) \left( \widetilde{D}_{ij,0} - \overline{D}_{ij,0} \right) \right\|_2 \\ 
		& \qquad \nonumber 
		+ \frac12 \sum_{i=1}^d \sum_{j=1}^d \left\| \frac{1}{N} I - \rho \left( \widetilde{F}_{ij} - \widehat{F}_{ij} + \epsilon I \right) \left( \overline{D}_{ij,0} - \widehat{D}_{ij,0} \right) \right\|_2 \\ 
		& \qquad \nonumber 
		+ \rho \epsilon \frac12 \sum_{i=1}^d \sum_{j=1}^d \left\| \widetilde{D}_{ij,0} - \overline{D}_{ij,0} \right\|_2 
		+ \rho \epsilon \frac12 \sum_{i=1}^d \sum_{j=1}^d \left\| \overline{D}_{ij,0} - \widehat{D}_{ij,0} \right\|_2 \\ 
	& \nonumber 
	\leq \frac12 + \frac{1}{2N} - \rho c k_0 + \frac{1}{2N} - \rho \frac{1}{2} c \lambda \kappa 
		+ \sum_{k=1}^{J_0} \frac{1}{2N} 
		+ d^2 \frac{1}{N} 
		+ \rho \epsilon (d^2 + 1) K \\ 
	%& \nonumber 
%	= \frac12 + \frac{2 + J_0 + 2 d^2}{2N} - \rho c k_0 - \rho \frac{1}{2} c \lambda \kappa + \rho \epsilon (d^2 + 1) K \\ 
%	& \nonumber 
%	< \frac12 + \frac12 - \rho \left( c k_0 + \frac12 c \lambda \kappa - \frac14 (c k_0 + c \lambda \kappa) \right) \\ 
	& \nonumber 
	< 1 - \rho \frac{3}{4} c k_0 - \rho \frac14 c \lambda \kappa
\end{align}
for all $\rho > 0$ sufficiently small and $\epsilon < c \frac{k_0 + \lambda \kappa}{4 (d^2+1) K}$ 
chosen independently of $\rho$.  
The bound 
\[
	\| \widehat{W} \|_{\ell^2(\cT_{\mathbf{h}} \cap \Omega)} 
	\leq \left(1 - \rho \frac{c \lambda \kappa}{4} - \rho \frac{c k_0}{4} \right) \| W \|_{\ell^2(\cT_{\mathbf{h}}\cap \Omega)} 
\] 
follows since $\|\widehat{\bf{W}}\|_2\leq \|G\|_2 \| \mathbf{W} \|_2$, 
and the bound over $\cT_{\bh}$ follows since $\widehat{W}_\alpha = W_\alpha = 0$ over $\cT_{\bh} \cap \partial \Omega$.  
The proof is complete. 
\hfill 
\end{proof}

As an immediate corollary to Lemma~\ref{lemma_contraction}, 
we have the following well-posedness result by the contractive mapping theorem.  

\begin{theorem} \label{existence_thm}
Suppose the operator $F$ in \eqref{bvp} is proper and uniformly elliptic,  
differentiable, and Lipschitz continuous with respect to its first two arguments. 
Suppose $\widehat{F}$ is consistent, is g-monotone, can be written in reduced form, 
is differentiable with respect to its first three arguments, and is elliptic compatible.  
The scheme \eqref{FD_method} 
for approximating problem \eqref{bvp} has a unique solution. 
\end{theorem}

\begin{remark} 
We emphasize that the properties of $F$ and $\hF$ guarantee a unique solution whenever $k_0 > 0$ or $\lambda > 0$.  
\end{remark}

%%%%%%%%%%%%%%%%%%%%%%%%%%%%%%%%%%%%%%%%%%%%%%%%

\section{Stability analysis} \label{stability_sec}

For transparency and consistency with the results in Section~\ref{admissible_sec}, 
assume the operator $F$ in \eqref{FD_pde} has the form  $F[u](\bx) = F \left( D^2 u, u, \bx \right)$.  

\begin{theorem} \label{L2stable_thm}
Suppose the operator $F$ in \eqref{bvp} is proper and uniformly elliptic 
and is Lipschitz continuous and differentiable with respect to its first two arguments 
with $\lambda > 0$ or $k_0 > 0$.  
Suppose $\widehat{F}$ is consistent, is g-monotone, can be written in reduced form, 
is Lipschitz continuous and differentiable with respect to its first three arguments, 
and is elliptic compatible.  
Then the solution $U$ to the scheme \eqref{FD_method} 
for approximating problem \eqref{bvp} is $\ell^2$-norm stable in the sense that 
\[
	\Bigl( \prod_{i=1,2,\ldots,d} h_i^{\frac12} \Bigr) \left\| U \right\|_{\ell^2(\cT_{\mathbf{h}} \cap \Omega)} 
	\leq C , 
\]
where $C$ is a positive $\mathbf{h}$-independent constant which depends on $\Ome$, the lower (proper) 
ellipticity constants $k_0 $ and $\lambda$, $F$, and $g$. 
\end{theorem}

\begin{proof}
Define the function $v \in C^0(\overline{\Omega}) \cap H^2(\Omega)$ to be the solution to 
\begin{subequations} \label{bvp_aux}
\begin{alignat}{2}
	- \Delta v & = 0 \qquad && \text{in } \Omega , \\ 
	v & = g \qquad && \text{on } \partial \Omega , 
\end{alignat}
\end{subequations} 
and define $V : \cT_{\bh}' \to \mathbb{R}$ by $V_\alpha = v(\bx_\alpha)$ for all 
$\bx_\alpha \in \cT_{\bh} \cap \overline{\Omega}$ 
and introduce the ghost points so that 
$\Delta_{\bh} V_\alpha = 0$ for all $\bx_\alpha \in \mathcal{S}_{\bh}$.  
Then, by the mean value theorem, there exists a linear operator $\mathcal{L}_{\bh}$ such that 
\[
	\hF[U_\alpha, \bx_\alpha] - \hF[V_\alpha, \bx_\alpha] = \mathcal{L}_{\bh} [ U_\alpha - V_\alpha ] . 
\]
Furthermore, since $U$ is a solution to \eqref{FD_method}, there holds 
\[
	\hF[U_\alpha, \bx_\alpha] - \hF[V_\alpha, \bx_\alpha] = 0 - \hF[V_\alpha, \bx_\alpha] = - \hF[V_\alpha, \bx_\alpha] . 
\]
Thus, $U_\alpha - V_\alpha$ is a solution to 
\begin{subequations}\label{FD_method_stable}
\begin{alignat}{2} 
\mathcal{L}_{\bh} [U_\alpha - V_\alpha, \bx_\alpha] + \hF[V_\alpha, \bx_\alpha] & = 0  &&\qquad\mbox{for } \mathbf{x}_\alpha\in\mathcal{T}_{\mathbf{h}} \cap \Omega, \label{FD_method_stable:1}\\
U_\alpha - V_\alpha &= 0  &&\qquad\mbox{for } \mathbf{x}_\alpha\in\mathcal{T}_{\mathbf{h}}\cap  \partial\Omega , \label{FD_method_stable:2} \\ 
\Delta_{\mathbf{h}} (U_\alpha - V_\alpha) & = 0 
&& \qquad \mbox{for } \mathbf{x}_\alpha\in \mathcal{S}_{\mathbf{h}} \subset \mathcal{T}_{\mathbf{h}} \cap \partial\Omega , 
\label{FD_auxiliaryBC_stable} 
\end{alignat}
\end{subequations}
where 
$\mathcal{L}_{\bh}$ has the form 
\[
	\mathcal{L}_{\bh} \equiv A_{\bh} : \widetilde{D}_{\bh}^2 - B_{\bh} : \widehat{D}_{\bh}^2 + c_{\bh}
\]
for $A_{\bh}, B_{\bh} \in \mathcal{S}^{d \times d}$ with all nonnegative components 
and $c_{\bh} \in \mathbb{R}$ with $c_{\bh} \geq 0$ 
by the g-monotonicity of $\hF$.  
Furthermore, $\mathcal{L}_{\bh}$ is consistent with the linear elliptic boundary value problem 
\begin{subequations}\label{bvp_L}
\begin{alignat}{2}
	\mathcal{L} w \equiv -A : D^2 w + c w & = - F[v]  , && \qquad \text{in } \Omega , \\ 
	u & = g , && \qquad \text{on } \partial \Omega 
\end{alignat}
\end{subequations}
with the matrix $A$ symmetric positive definite with $\lambda I \leq A \leq \Lambda I$ and $c \geq k_0$, 
where $A_{\bh} - B_{\bh} \leq - c_0 \lambda I$ and $c_{\bh} \geq c_0 k_0$ 
for $c_0$ the constant based on the elliptic compatibility of $\hF$.  

Applying Theorem~\ref{existence_thm}, $U_\alpha - V_\alpha$ is the unique solution to \eqref{FD_method_stable}.  
Furthermore, by Lemma~\ref{lemma_contraction} and using the technique in Theorem 4.3 of \cite{FengLewis21}, 
there holds 
\begin{align*}
	\| U - V \|_{\ell^2(\cT_{\bh} \cap \Omega)} 
	& \leq \frac{4}{c_0 \lambda \kappa + c_0 k_0} \| \mathcal{L}_{\bh}[0] + \widehat{F}[V_\alpha , \bx_\alpha] \|_{\ell^2(\cT_{\bh} \cap \Omega)} \\ 
	& = \frac{4}{c_0 \lambda \kappa + c_0 k_0} \| \widehat{F}[V_\alpha , \bx_\alpha] \|_{\ell^2(\cT_{\bh} \cap \Omega)} . 
\end{align*}
Therefore, 
\begin{align*}
	\| U  \|_{\ell^2(\cT_{\bh} \cap \Omega)}  
	& \leq \| V \|_{\ell^2(\cT_{\bh} \cap \Omega)} 
		+ \frac{4}{c_0 \lambda \kappa + c_0 k_0} \| \widehat{F}[V_\alpha , \bx_\alpha] \|_{\ell^2(\cT_{\bh} \cap \Omega)} 
% \\ 
%	& \leq \| V \|_{\ell^2(\cT_{\bh} \cap \Omega)} 
%		+ \frac{4}{\lambda \kappa + k_0} \| \widehat{F}[V_\alpha , \bx_\alpha] - F[v](\bx_\alpha) \|_{\ell^2(\cT_{\bh} \cap \Omega)} \\
%		& \qquad 
%		+ \frac{4}{\lambda \kappa + k_0} \| F[v](\bx_\alpha) \|_{\ell^2(\cT_{\bh} \cap \Omega)}
\end{align*}
and the result follows by the properties of $v$.  
The proof is complete.   
\hfill
\end{proof}

Extending the techniques above and following the proofs in Section 5 of \cite{FengLewis21}, we can also prove the 
following results.  

\begin{lemma} \label{lemma_contraction2} 
Let $i \in \{1,2,\ldots,d \}$.  
Suppose the operator $F$ in \eqref{bvp} is proper and uniformly elliptic,  
differentiable, and Lipschitz continuous with respect to its first two arguments. 
Suppose $\widehat{F}$ is consistent, g-monotone, can be written in reduced form, 
is differentiable with respect to its first three arguments, 
and is elliptic compatible.  
Choose $U,V\in S(\cT_{\mathbf{h}}')$ that satisfy the boundary conditions \eqref{FD_method_linear:2} 
and \eqref{FD_method_linear:3}, 
and let $\widehat{U} = \cM_{\rho,2} U$ and $\widehat{V} = \cM_{\rho,2} V$, 
where $\widehat{W} \in S(\cT_{\mathbf{h}}')$ is defined by $\widehat{W} = \cM_{\rho,2} W$ 
for some $W \in S(\cT_{\mathbf{h}}')$ if 
\begin{alignat*}{2}  
	 - \delta_{x_i, 2 h_i}^2 \widehat{W}_\alpha &= - \delta_{x_i, 2 h_i}^2 W_\alpha - \rho \widehat{F} [ W_\alpha , \bx_\alpha] \qquad 
		&& \forall \bx_\alpha \in \cT_{\bh} \cap \Omega , \\ 
	 \widehat{W}_\alpha &= g(\bx_\alpha) && \forall \bx_\alpha \in \cT_{\bh} \cap \partial \Omega , \\ 
	\Delta_{\bh} \widehat{W}_\alpha & = 0 && \forall \mathbf{x}_\alpha\in \mathcal{S}_{\bh} . 
\end{alignat*}
Then there holds 
\[
	\| \delta_{x_i, 2 h_i}^2 ( \widehat{U} - \widehat{V} ) \|_{\ell^2(\cT_{\mathbf{h}} \cap \Omega)} 
	\leq \left(1 - \rho \frac{c \lambda \kappa}{4} \right) \| \delta_{x_i, 2 h_i}^2 ( U - V ) \|_{\ell^2(\cT_{\mathbf{h}} \cap \Omega)} 
\] 
for all $\rho > 0$ sufficiently small, 
where $4 > \rho c \lambda \kappa$, 
$\frac{\partial F}{\partial D^2 u} \geq \lambda I$, $\frac{\partial F}{\partial u} \geq k_0$, 
and $M \geq \kappa I$ for $M$ the matrix representation of $-\Delta_{2 \bh}$.  
\end{lemma}

\begin{theorem} \label{H2stable_thm}
Suppose the operator $F$ in \eqref{bvp} is proper and uniformly elliptic,  
and is Lipschitz continuous and differentiable with respect to its first two arguments 
with $\lambda > 0$.  
Suppose $\widehat{F}$ is consistent, is g-monotone, can be written in reduced form, 
is Lipschitz continuous and differentiable with respect to its first three arguments, 
and is elliptic compatible.  
Then the solution $U$ to the scheme \eqref{FD_method} 
for approximating problem \eqref{bvp} satisfies 
\[
	\Bigl( \prod_{i=1,2,\ldots,d} h_i^{\frac12} \Bigr) \left\| \delta_{x_i, 2 h_i}^2 U \right\|_{\ell^2(\cT_{\mathbf{h}} \cap \Omega)} 
	\leq C 
\]
for all $i=1,2,\ldots,d$, 
where $C$ is a positive $\mathbf{h}$-independent constant which depends on $\Ome$, the lower (proper) 
ellipticity constant $\lambda$, $F$, and $g$. 
\end{theorem}

Lastly, by \cite{FengLewis21}, combining Theorems~\ref{L2stable_thm} and \ref{H2stable_thm} 
using a novel numerical embedding technique 
yields the following $\ell^\infty$ stability result.  

\begin{theorem} \label{thm_stability} 
Under the assumptions of Lemma~\ref{lemma_contraction2}, the numerical 
solution $U$ is stable in the $\ell^\infty$-norm for $d\leq 3$; that is, $U$ satisfies
$\left\| U \right\|_{\ell^\infty(\cT_{\mathbf{h}})} 
	\leq C $
for $d\leq 3$, where $C$ is a positive constant independent of $\mathbf{h}$.  
\end{theorem}

%%%%%%%%%%%%%%
\section{Numerical experiments}\label{experiments_sec}

In this section we test the performance of the FD method based on the numerical operator 
$\hF_{\gamma,\sigma}$ defined by \eqref{Fgamma} for $\sigma \geq 0$ and $\gamma + \sigma \geq 0$.  
We consider test problems based on 
choosing $F$ to be 
linear and uniformly elliptic with non-divergence form, 
the Hamilton-Jacobi-Bellman operator, 
the Monge-Amp\`ere operator, 
and the operator coming from the equation of prescribed Gauss curvature.  
We will see that for the linear problem, $\hF_{0,0}$ yields a nonsingular sparse matrix.  
However, the nonlinear PDE problems considered are degenerate, 
and, consequently, the nonlinear solver `fsolve' in MATLAB has trouble finding a zero when 
using the zero function as the initial guess and $\sigma$ and $\gamma$ are small.  
For the Hamilton-Jacobi-Bellman problem, 
we can successfully solve the limiting cases by forming a sequence of 
approximations for decreasing values of $\gamma$ starting with an initial large 
value for $\gamma$ when initializing fsolve with the zero function.  
For the Monge-Amp\`ere operator and in the equation of prescribed Gauss curvature, 
we need to form the sequence of approximations with $\sigma$ large 
to more strictly impose g-monotonicity 
with respect to all of the components of the discrete Hessian operators.  
Once fsolve is in a neighborhood of the solution we can successfully find a zero 
for all $\sigma \geq 0$ and $\gamma \geq - \sigma$.  
For the Monge-Amp\'ere problem and the prescribed Gauss curvature problem, 
we choose an exact solution for which the Monge-Amp\`ere 
operator is locally uniformly elliptic.  
It is worth noting that fsolve does not find a false solution but instead always reported no solution found 
when staring with a poor initial guess.  
This is in direct contrast to the experiment in \cite{Feng_Glowinski_Neilan10} 
that found false solutions to the Monge-Amp\`ere problem when using the 
standard nine-point finite difference formula for approximating the discrete Hessian.  
 
Additional numerical tests for g-monotone FD methods can be found in 
\cite{Lewis_dissertation,FengKaoLewis13,FD_Long} 
and tests for the corresponding DG scheme can be found in 
\cite{Lewis_dissertation,FengLewis18}.  

%%%%%

\subsection{Test 1: Linear $F$ with non-aligned grids} 

In this test, we consider the linear uniformly elliptic problem $F[u] = -A : D^2 u - f$ 
for a discontinuous coefficient matrix $A$ and uniform grids chosen 
such that no monotone finite difference method exists 
(see \cite{MotzkinWasow53}).  
The matrix $A(\bx)$ will be uniformly symmetric positive definite, and $f$ will be a uniformly bounded 
function.  
We form a non-singular linear system $LU=F$ that is solved using MATLAB's backslash command. 
The matrix $L$ is formed using sparse storage. 

Let $\Omega = (-1,1)^2$.  
We form a sequence of uniform grids with 
$N_x = N_y = 10, 40, 80, 120, 180, 240, 300$.  
Let $h_x = h_y = \frac{2}{301}$ based on the finest mesh, 
define the directions $\bfv^{(i)} \in \mathbb{R}^2$ by 
\begin{align*}
    \bfv^{(1)} &= \left[ \begin{array}{cc} 1 \\ \frac{h_y}{2} \end{array}\right] , \qquad 
    \bfv^{(2)} = \left[ \begin{array}{cc} h_x \\ 5 h_y \end{array}\right] , \qquad
    \bfv^{(3)} = \left[ \begin{array}{cc} 10 h_x \\ h_y \end{array}\right] , \qquad
    \bfv^{(4)} = \left[ \begin{array}{cc} \frac{h_x}{2} \\ 2 \end{array}\right] , 
\end{align*}
and define the unit length vectors $\mathbf{q}_1^{(i)} = \frac{1}{\| \bfv^{(i)} \|_2} \bfv^{(i)}$ 
for all $i=1,2,3,4$.  
Note that, by construction, the directions $\bfv^{(1)}$ and $\bfv^{(4)}$ do not align with grid points in any 
of the meshes while $\bfv^{(2)}$ aligns with nodes at least 5 layers away 
and $\bfv^{(3)}$ aligns with nodes at least 10 layers away.  
For each $i=1,2,3,4$, choose unit length vectors $\mathbf{q}_2^{(i)}$ that are orthogonal to $\mathbf{q}_1^{(i)}$. 
%(which was automated using the qr command in MATLAB in the code).  
Define the orthogonal matrices $Q_i = [ \mathbf{q}_1^{(i)} , \mathbf{q}_2^{(i)} ]$ 
for each $i=1,2,3,4$, 
and define the diagonal matrix $\Lambda(x,y) \in \mathbb{R}^{2 \times 2}$ by 
 \begin{align*}
        \Lambda(x,y) 
        &= \left[ \begin{array}{cc}
                2 - \sin{(e^{5x})}\cos{(e^{-3y})} & 0 \\ 
                0 & 2 - \mbox{sign}\big(\cos{(6\pi x)}\sin{(6\pi y)}\big) 
                 \end{array}\right] 
\end{align*} 
for all $(x,y) \in \Omega$. 
Then, $\Lambda_{11} \in [1,3]$ oscillates several times and 
$\Lambda_{22} \in \{2,3,4\}$ 
has several discontinuities over $\Omega$.  
Finally, we define $A$ by 
    \[
    A(x,y) = 
    \begin{cases} 
        Q_1\, \Lambda(x,y)\, Q_1^T , & \mbox{for } x\geq0 ,\enskip y\geq0, \\
        Q_2\, \Lambda(x,y)\, Q_2^T , & \mbox{for } x<0 ,\enskip y\geq0, \\
        Q_3\, \Lambda(x,y)\, Q_3^T , & \mbox{for } x<0 ,\enskip y<0, \\
        Q_4\, \Lambda(x,y)\, Q_4^T , & \mbox{otherwise } \\
    \end{cases}
    \]
so that $A$ is strictly symmetric positive definite and discontinuous.  
No monotone method exists for this problem on the specified grids due to the choices for $Q_1$ and $Q_4$, 
and $Q_2$ and $Q_3$ would lead to wider-stencils.  

We consider the problem $-A:D^2 u = f$ for two different solutions 
\[
	u^{(1)}(x,y) = \sin \left( \frac{\pi}{2} (x+y)^2 \right)
\]
and 
\[
	u^{(2)}(x,y) = \frac{x^3}{18} \left(3\log(x^2)-11 \right) 
		+ \left(y-\frac{1}{2} \right)^{\frac{8}{3}} \sqrt{\bigg|x+\frac{1}{5}\bigg|^5} 
\]
so that $u^{(1)} \in C^{\infty}(\Omega)$ 
and $u^{(2)} \in C^2(\Omega) \setminus C^3(\Omega)$.  
The source function $f$ and boundary data $g$ are chosen so that the solution is given by 
either $u^{(1)}$ or $u^{(2)}$.  
Using $\hF_{0,0}$, we observe optimal second order rates when approximating $u^{(1)}$ 
and expected deteriorated rates when approximating $u^{(2)}$ 
in Table~\ref{table_test1}.  
 
 \begin{table}[htb] 
{\small 
\begin{center}
\begin{tabular}{| c | c | c | c | c |}
		\hline 
	& \multicolumn{2}{c |}{$u^{(1)}$} & \multicolumn{2}{c |}{$u^{(2)}$} \\ 
		\hline
	 $h$ & $\ell^\infty$ Error & Order & $\ell^\infty$ Error & Order \\ 
		\hline
	1.82e-01 & 5.77e-02 &  & 5.21e-03 &  \\ 
		\hline
	4.88e-02 & 3.86e-03 & 2.06 & 8.29e-04 & 1.40 \\ 
		\hline
	2.47e-02 & 9.79e-04 & 2.01 & 3.26e-04 & 1.37 \\ 
		\hline
	1.65e-02 & 4.40e-04 & 1.99 & 1.86e-04 & 1.40 \\ 
		\hline
	1.10e-02 & 1.96e-04 & 2.00 & 1.05e-04 & 1.41 \\ 
		\hline
	8.30e-03 & 1.11e-04 & 2.00 & 7.02e-05 & 1.42 \\ 
		\hline
	6.64e-03 & 7.10e-05 & 2.00 & 5.11e-05 & 1.43 \\ 
		\hline
\end{tabular}
\end{center}
}
\caption{
Approximation results for Test 1 using $\hF_{0,0}$.   
}
\label{table_test1}
\end{table}

%%%%%

\subsection{Test 2: Hamilton-Jacobi-Bellman equations} 

This example is adapted from \cite{Smears_Suli}, 
and it was considered in \cite{FengLewis21} using the Lax-Friedrich's-like method.  
Let $\Lambda = [0,\pi/3] \times SO(2)$, where $SO(2)$ is the set of $2 \times 2$ rotation matrices 
and define $\sigma^{\theta}$ by 
\[
	\sigma^\theta \equiv R^T \left[ \begin{array}{cc} 1 & \sin(\phi) \\ 0 & \cos(\phi) \end{array} \right] , 
	\qquad \theta = (\phi, R) \in \Lambda . 
\]
Consider the Hamilton-Jacobi-Bellman equation 
\begin{align*} 
F[u] &\equiv \inf_{\theta \in \Lambda} \bigl( L_\theta u - f_\theta \bigr) = 0, \\
L_\theta u &= -A^{\theta}(\bx) : D^2 u + b^{\theta}(\bx) \cdot \nabla u + c^{\theta}(\bx) \, u 
\end{align*}
with $\Omega = (0,1)^2$, 
$A^\theta = \frac12 \sigma^\theta \left( \sigma^\theta \right)^T$, 
$\beta^\theta = \vec{0}$, $c^\theta = \pi^2$, 
$f_\theta = \sqrt{3} \sin^2(\phi/\pi^2) + g$
for $g$ chosen independent of $\theta$, 
and Dirichlet boundary data 
chosen such that the exact solution 
is given by $u(x,y) = e^{xy} \sin(\pi x) \sin( \pi y)$.  
The optimal controls vary significantly throughout the domain and the corresponding 
diffusion coefficient is not diagonally dominant in parts of $\Omega$.  
Furthermore, the coefficient matrix is degenerate for certain choices of $\theta$.  

We approximate $u$ using $\hF_{\gamma,0}$ for 
$\gamma = 1000, 10, 1, 0$ in Table~\ref{test2_fig}. 
The case $\gamma = 1000$ appears to be in a pre-asymptotic regime 
with the error dominated by the numerical moment term.  
Otherwise, as expected, we see that the methods are around second order accuracy 
with the approximations increasing in accuracy as $\gamma$ decreases. 
In the implementation of $\hF_{\gamma,\sigma}$, we do not calculate derivatives of $F$ 
to define $M_\alpha$.  
Instead, we define $\hF_{\gamma,\sigma}^\theta$ by 
\begin{align*}
	\hF_{\gamma,\sigma}^\theta [U_\alpha] & \equiv A^\theta(\bx_\alpha) : D_{\bh}^{2,\theta} U_\alpha + c^\theta(\bx_\alpha) U_\alpha - f_\theta(\bx_\alpha) \\ 
	& \qquad + \left( \gamma I_{d \times d} + \sigma 1_{d \times d} \right) : \left( \widetilde{D}_{\bh}^2 U_\alpha - \widehat{D}_{\bh}^2 U_\alpha \right)  
\end{align*}
for $D_{\bh}^{2,\theta} U_\alpha$ defined analogously to \eqref{Dhlin} for each value of $\alpha$ and $\theta$.  
We then solve the optimization problem 
$\inf_{\theta \in \Lambda} \hF_{\gamma,\sigma}^\theta [U_\alpha] = 0$.  

When choosing $\gamma=0$, the solver does not successfully find a solution 
even for coarse meshes.  
The solver also takes several iterations for $\gamma = 1$ 
sometimes struggling to find a solution for the finer meshes.    
Thus, we first find the solution corresponding to $\gamma = 1000$, 
and then we sequentially find the solutions corresponding to $\gamma = 100, 10, 1, 0$ 
by using the solution for the next largest value of $\gamma$ as an initial guess.  
This iterative technique allows fsolve to successfully find a solution for each 
value of $\gamma$ 
and often allows fsolve to significantly decrease the number of iterations needed to converge.  
We also can form a better initial guess by first solving the problem corresponding to a small 
finite number of controls $\theta$ before optimizing over $\Lambda$.  

\begin{table}[htb] 
{\small 
\begin{center}
\begin{tabular}{| c | c | c | c | c | c | c |}
		\hline
	& \multicolumn{2}{c |}{$\gamma = 1000$} & \multicolumn{2}{c |}{$\gamma = 100$} 
		& \multicolumn{2}{c |}{$\gamma = 10$} \\ 
		\hline
	 $h$ & $\ell^\infty$ Error & Order & $\ell^\infty$ Error & Order & $\ell^\infty$ Error & Order  \\ 
		\hline
	1.57e-01 & 1.30e+00 & & 1.17e+00 & & 6.08e-01 &  \\ 
		\hline
	9.43e-02 & 1.27e+00 & 0.05 & 9.83e-01 & 0.35 & 3.34e-01 & 1.18 \\ 
		\hline
	6.15e-02 & 1.22e+00 & 0.09  & 7.51e-01 & 0.63 & 1.73e-01 & 1.54  \\ 
		\hline
	4.56e-02 & 1.15e+00 & 0.20 & 5.67e-01 & 0.94 & 1.01e-01 & 1.80  \\ 
		\hline
	3.63e-02 & 1.08e+00 & 0.29 & 4.34e-01 & 1.16 & 6.47e-02 & 1.95  \\ 
		\hline
	2.89e-02 & 9.78e-01 & 0.43 & 3.19e-01 & 1.35 & 4.05e-02 & 2.05  \\ 
		\hline
	2.21e-02 & 8.32e-01 & 0.61 & 2.12e-01 & 1.53 & 2.30e-02 & 2.12  \\ 
		\hline
\end{tabular}

\vspace{0.1in}

\begin{tabular}{| c | c | c | c | c | c | c |}
		\hline
	& \multicolumn{2}{c |}{$\gamma = 1$} & \multicolumn{2}{c |}{$\gamma = 0$} \\ 
		\hline
	 $h$ & $\ell^\infty$ Error & Order & $\ell^\infty$ Error & Order  \\ 
		\hline
	1.57e-01 & 1.32e-01 & & 2.35e-02 &  \\ 
		\hline
	9.43e-02 & 5.08e-02 & 1.87 & 1.03e-02 & 1.61 \\ 
		\hline
	6.15e-02 & 2.15e-02 & 2.01 & 4.96e-03 & 1.71 \\ 
		\hline
	4.56e-02 & 1.17e-02 & 2.03 & 2.92e-03 & 1.77 \\ 
		\hline
	3.63e-02 & 7.43e-03 & 1.99 & 1.94e-03 & 1.79 \\ 
		\hline
	2.89e-02 & 4.74e-03 & 1.97 & 1.28e-03 & 1.81 \\ 
		\hline
	2.21e-02 & 2.81e-03 & 1.96 & 7.92e-04 & 1.80 \\ 
		\hline
\end{tabular}
\end{center}
}
\caption{
Approximation results for Test 2 using $\hF_{\gamma,0}$ with $\gamma = 1000, 100, 10, 1, 0$.   
}
\label{test2_fig}
\end{table}
 
%%%%%

\subsection{Test 3: Monge-Amp\`ere equation}

Consider the Monge-Amp\`ere problem 
\[
	F[u] = - \text{det}(D^2 u) + f = - u_{xx} u_{yy} + |u_{xy}|^2 + f = 0 
\]
over $\Omega = (0,1)^2$.  
The problem has a unique convex viscosity solution whenever $f \geq 0$.  
We choose the source term $f$ and boundary function $g$ such that the 
exact solution is given by 
$u(x,y) = e^{\frac{x^2+y^2}{2}}$.  
The matrix $M_\alpha$ is easily found since 
$\left| \frac{\partial F}{\partial P_{ij}} \right| = | P_{k\ell} |$
for $(k,\ell) = (2,2)$ if $(i,j)=(1,1)$, 
$(k,\ell) = (1,1)$ if $(i,j)=(2,2)$, 
and $(k,\ell) = (1,2)$ if $(i,j)=(1,2)$ or $(i,j) = (2,1)$.  

This problem is degenerate and the uniformity of the g-monotonicity of $\hF_{\gamma,\sigma}[U]$ 
strongly depends upon $U$ similarly to how the ellipticity of $F[u]$ strongly depends upon the convexity of $u$.  
Consequently, we now form initial guesses for fsolve by solving a sequence of problems 
based on a decreasing sequence of values for $\sigma$ starting with $\sigma$ large.  
As such, we are using a stronger form of the numerical moment to overcome the conditional ellipticity 
and potential degeneracy of the problem.  

Rates of convergence for the various tests can be found in Table~\ref{test3_fig} where we 
consider the method $\hF_{\gamma,0}$ for $\gamma = 1000, 100, 10, 1, 0$
and Table~\ref{test3b_fig} where we consider the method $\hF_{-\sigma,\sigma}$ 
for $\sigma = 1000, 100, 10, 1$.  
We observe optimal / near optimal rates of convergence as $\gamma$ decreases in Table~\ref{test3_fig} 
and as $\sigma$ decreases in Table~\ref{test3b_fig}. 
For $\gamma$ and $\sigma$ large the rates appear to be suboptimal but improving towards a rate of 2 as 
$h$ decreases.  
We also note that the method $\hF_{-\sigma,\sigma}$ appears more accurate than the analogous method 
$\hF_{\gamma,0}$ when $\gamma = \sigma$.  
Such a relationship is expected based on the sensitivity of the Monge-Amp\`ere problem to the 
auxiliary boundary condition. 
Indeed, setting $\Delta_{\bh} U_\alpha = 0$ along the boundary and observing that $g$ must be convex along the 
tangential direction implies $U$ must be concave along the normal direction 
which brings a qualitative error into the interior of the domain.  
We can see this directly in Figure~\ref{test3_plots} 
where we plot the approximation for varying $h$, $\gamma$, and $\sigma$ values.  
For $\gamma > 0$ on a coarse mesh we can see that the convexity of the approximation 
is incorrect near the boundary and that for $\gamma$ large this forces the curvature to be incorrect throughout 
the interior.  
Consequently, for $\gamma$ large, the method does not appear to enforce the convexity of the underlying 
viscosity solution.  
Instead, based on the solver's performance and based on the tests in \cite{Lewis_dissertation}, 
the numerical moment for $\gamma$ large appears to minimize the number of times the function changes 
convexity over the domain by penalizing discontinuities in the second derivative 
and steering the approximation towards the correct viscosity solution as $h \to 0$.  
We also note that, in contrast, choosing $\sigma$ large and minimizing the effect of the auxiliary boundary condition by setting $\gamma = - \sigma$ 
appears to completely eliminate the convexity issue.  
Another way to decrease the convexity issue when $\gamma > 0$ is to 
set only $\delta_{x_i, h_i}^2 U_\alpha = 0$ along the normal direction instead of setting 
$\Delta_{\bh} U_\alpha = 0$ or to choose a more appropriate positive value for the auxiliary boundary condition 
consistent with the convex nature of the viscosity solution.  

\begin{table}[htb] 
{\small 
\begin{center}
\begin{tabular}{| c | c | c | c | c | c | c |}
		\hline
	& \multicolumn{2}{c |}{$\gamma = 1000$} & \multicolumn{2}{c |}{$\gamma = 100$} 
		& \multicolumn{2}{c |}{$\gamma = 10$} \\ 
		\hline
	 $h$ & $\ell^\infty$ Error & Order & $\ell^\infty$ Error & Order & $\ell^\infty$ Error & Order  \\ 
		\hline
	2.83e-01 & 5.59e-01 &  & 5.49e-01 &  & 4.02e-01 &  \\ 
		\hline
	1.29e-01 & 5.86e-01 & -0.06 & 5.14e-01 & 0.08 & 9.65e-02 & 1.81 \\ 
		\hline
	6.15e-02 & 5.64e-01 & 0.05 & 2.22e-01 & 1.14 & 2.09e-02 & 2.08 \\ 
		\hline
	3.01e-02 & 4.22e-01 & 0.40 & 5.07e-02 & 2.06 & 5.23e-03 & 1.94 \\ 
		\hline
	1.99e-02 & 2.32e-01 & 1.46 & 2.18e-02 & 2.04 & 2.39e-03 & 1.90 \\ 
		\hline
	1.49e-02 & 1.31e-01 & 1.97 & 1.22e-02 & 2.00 & 1.38e-03 & 1.89 \\ 
		\hline
	1.19e-02 & 8.13e-02 & 2.11 & 7.95e-03 & 1.90 & 8.98e-04 & 1.91 \\ 
		\hline
\end{tabular}

\vspace{0.1in}

\begin{tabular}{| c | c | c | c | c | c | c |}
		\hline
	& \multicolumn{2}{c |}{$\gamma = 1$} & \multicolumn{2}{c |}{$\gamma = 0$} \\ 
		\hline
	 $h$ & $\ell^\infty$ Error & Order & $\ell^\infty$ Error & Order  \\ 
		\hline
	2.83e-01 & 4.19e-02 &  & 1.57e-02 &  \\ 
		\hline
	1.29e-01 & 9.30e-03 & 1.91 & 3.41e-03 & 1.94 \\ 
		\hline
	6.15e-02 & 2.31e-03 & 1.89 & 7.87e-04 & 1.99 \\ 
		\hline
	3.01e-02 & 5.87e-04 & 1.91 & 1.88e-04 & 2.01 \\ 
		\hline
	1.99e-02 & 2.64e-04 & 1.94 & 8.21e-05 & 2.01 \\ 
		\hline
	1.49e-02 & 1.50e-04 & 1.93 & 4.58e-05 & 2.00 \\ 
		\hline
	1.19e-02 & 9.70e-05 & 1.94 & 2.91e-05 & 2.00 \\ 
		\hline
\end{tabular}
\end{center}
}
\caption{
Approximation results for Test 3 using $\hF_{\gamma,0}$ with $\gamma = 1000, 100, 10, 1, 0$.   
}
\label{test3_fig}
\end{table}

\begin{table}[htb] 
{\footnotesize 
\begin{center}
\begin{tabular}{| c | c | c | c | c | c | c | c | c |}
		\hline
	& \multicolumn{2}{c |}{$\sigma = 1000$} & \multicolumn{2}{c |}{$\sigma = 100$} 
		& \multicolumn{2}{c |}{$\sigma = 10$} & \multicolumn{2}{c |}{$\sigma = 1$} \\ 
		\hline
	 $h$ & $\ell^\infty$ Error & Order & $\ell^\infty$ Error & Order & $\ell^\infty$ Error & Order & $\ell^\infty$ Error & Order  \\ 
		\hline
	2.83e-01 & 3.65e-02 &  & 3.36e-02 &  & 1.95e-02 &  & 1.47e-02 &  \\ 
		\hline
	1.29e-01 & 3.56e-02 & 0.03 & 2.49e-02 & 0.38 & 6.22e-03 & 1.45 & 3.05e-03 & 1.99 \\ 
		\hline
	6.15e-02 & 3.07e-02 & 0.20 & 1.19e-02 & 1.00 & 1.63e-03 & 1.81 & 7.00e-04 & 2.00 \\ 
		\hline
	3.01e-02 & 1.97e-02 & 0.62 & 3.83e-03 & 1.59 & 4.07e-04 & 1.95 & 1.67e-04 & 2.00 \\ 
		\hline
	1.99e-02 & 1.24e-02 & 1.13 & 1.79e-03 & 1.84 & 1.79e-04 & 1.98 & 7.31e-05 & 2.00 \\ 
		\hline
	1.49e-02 & 8.14e-03 & 1.44 & 1.02e-03 & 1.92 & 1.00e-04 & 1.99 & 4.08e-05 & 2.00 \\ 
		\hline
	1.19e-02 & 5.66e-03 & 1.62 & 6.59e-04 & 1.95 & 6.41e-05 & 1.99 & 2.60e-05 & 2.00 \\ 
		\hline
\end{tabular}
\end{center}
}
\caption{
Approximation results for Test 3 using $\hF_{\gamma,\sigma}$ 
for $\sigma = 1000$, $100$, $10$, $1$ and $\gamma = - \sigma$.   
}
\label{test3b_fig}
\end{table}

\begin{figure}
\centerline{
\includegraphics[scale=0.1]{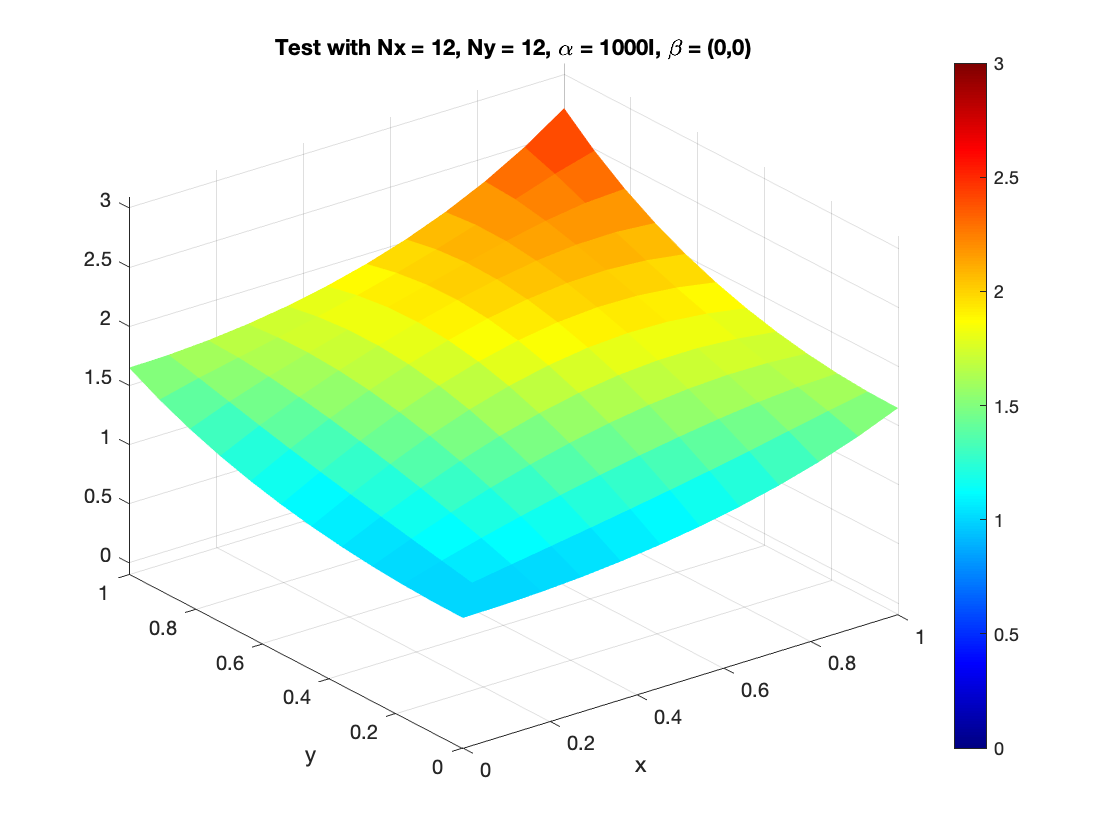}
\includegraphics[scale=0.1]{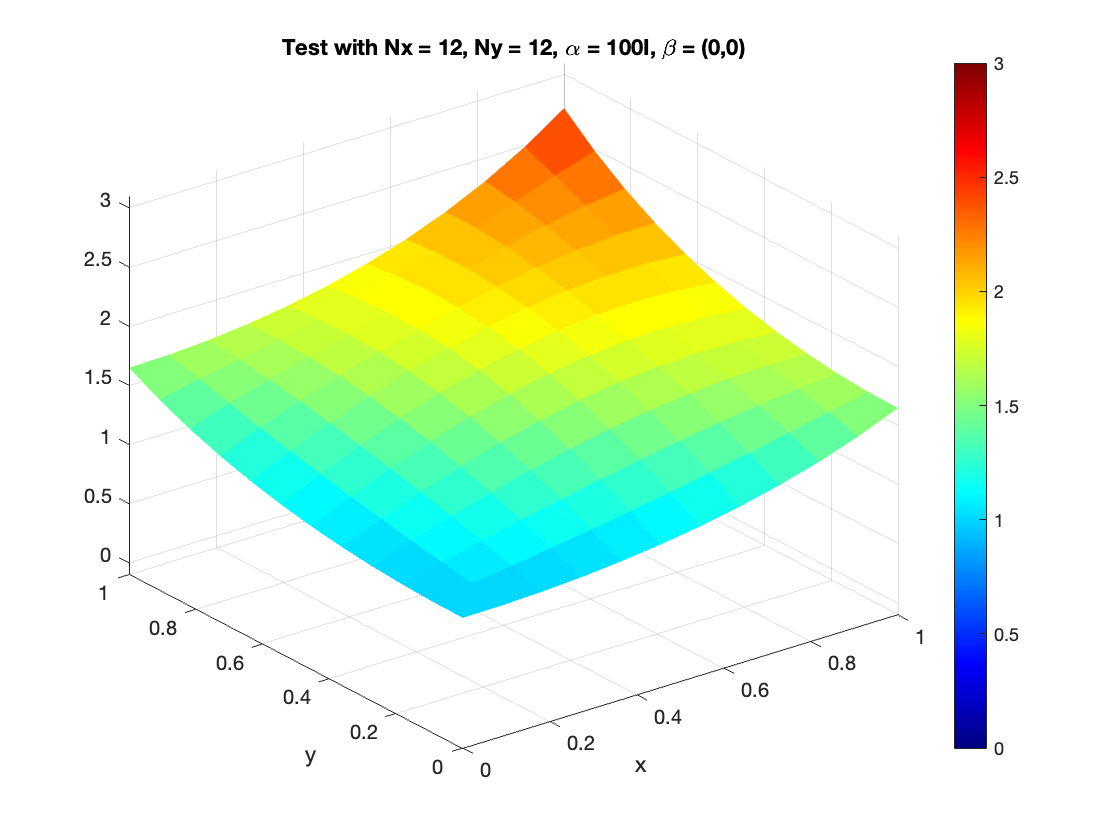}
\includegraphics[scale=0.1]{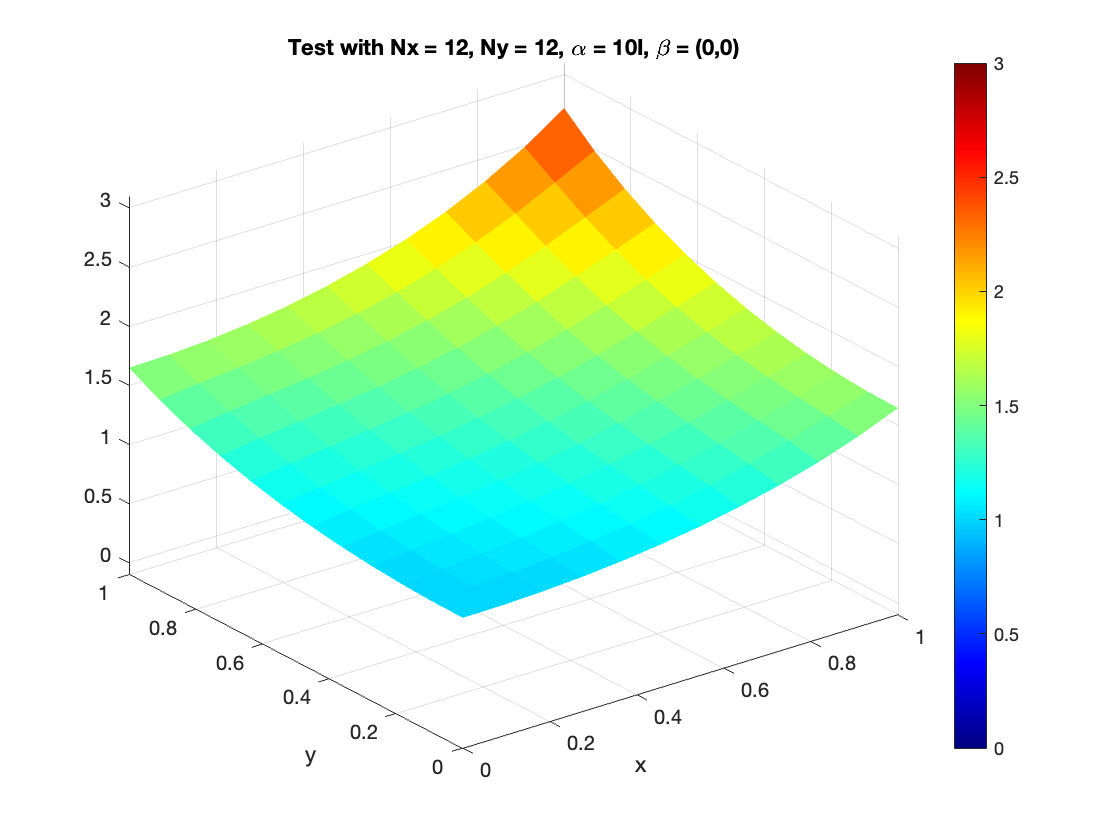}
}
\centerline{
\includegraphics[scale=0.1]{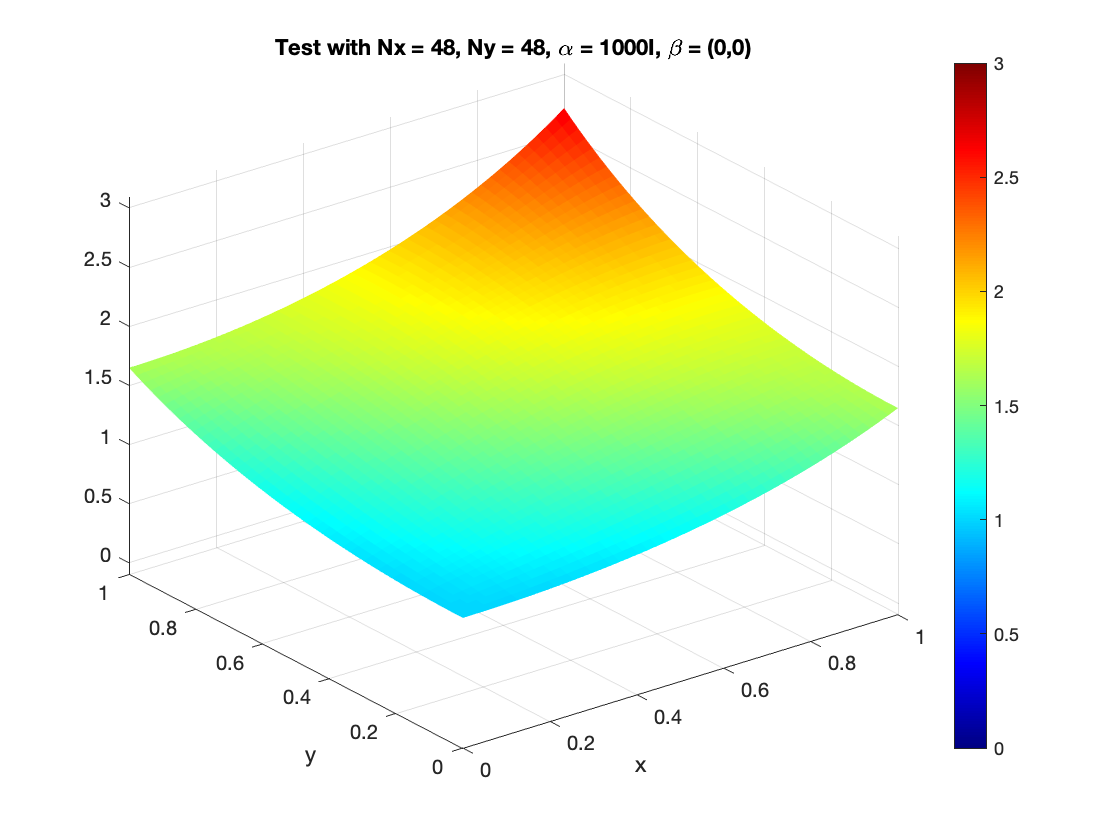}
\includegraphics[scale=0.1]{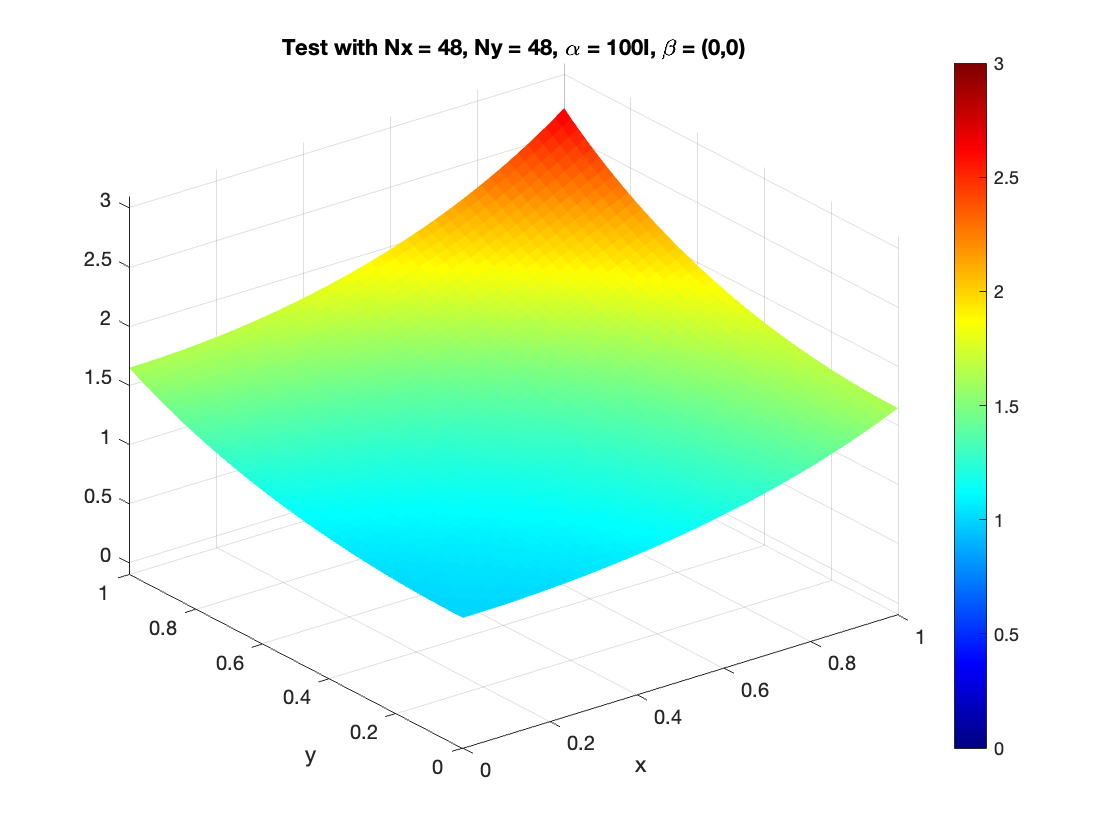}
\includegraphics[scale=0.1]{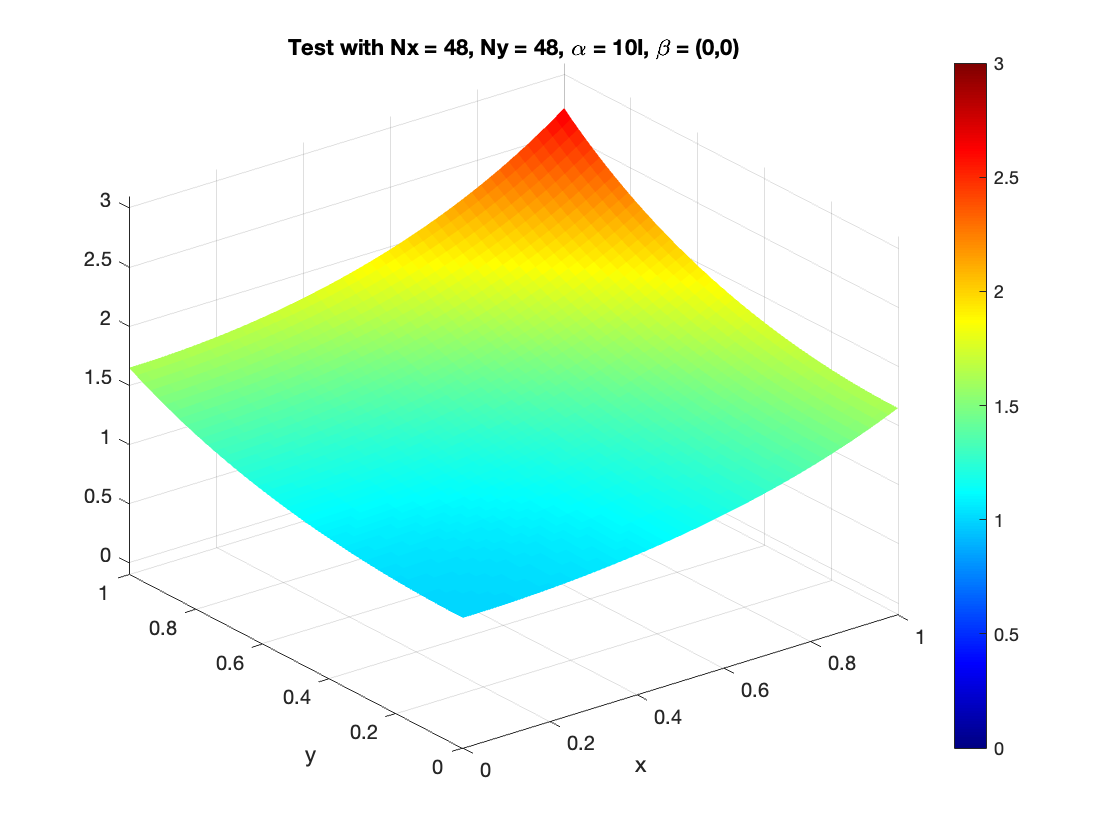}
}
\centerline{
\includegraphics[scale=0.1]{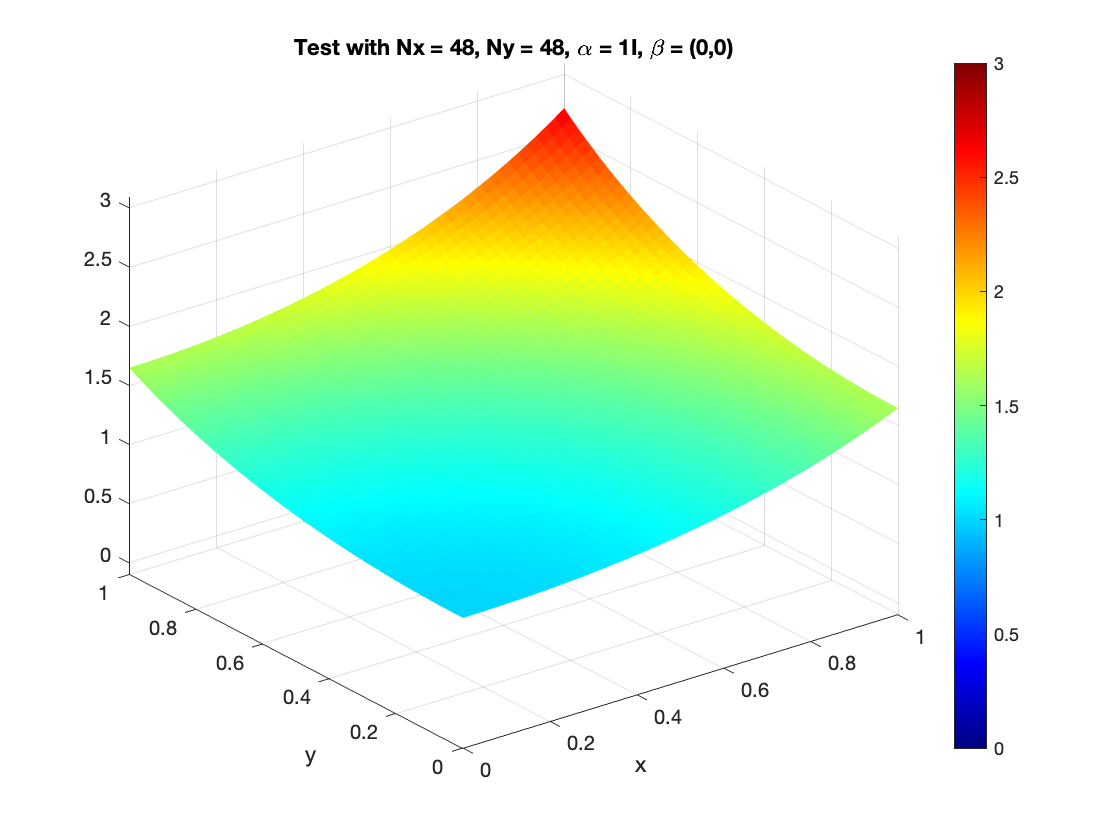}
\includegraphics[scale=0.1]{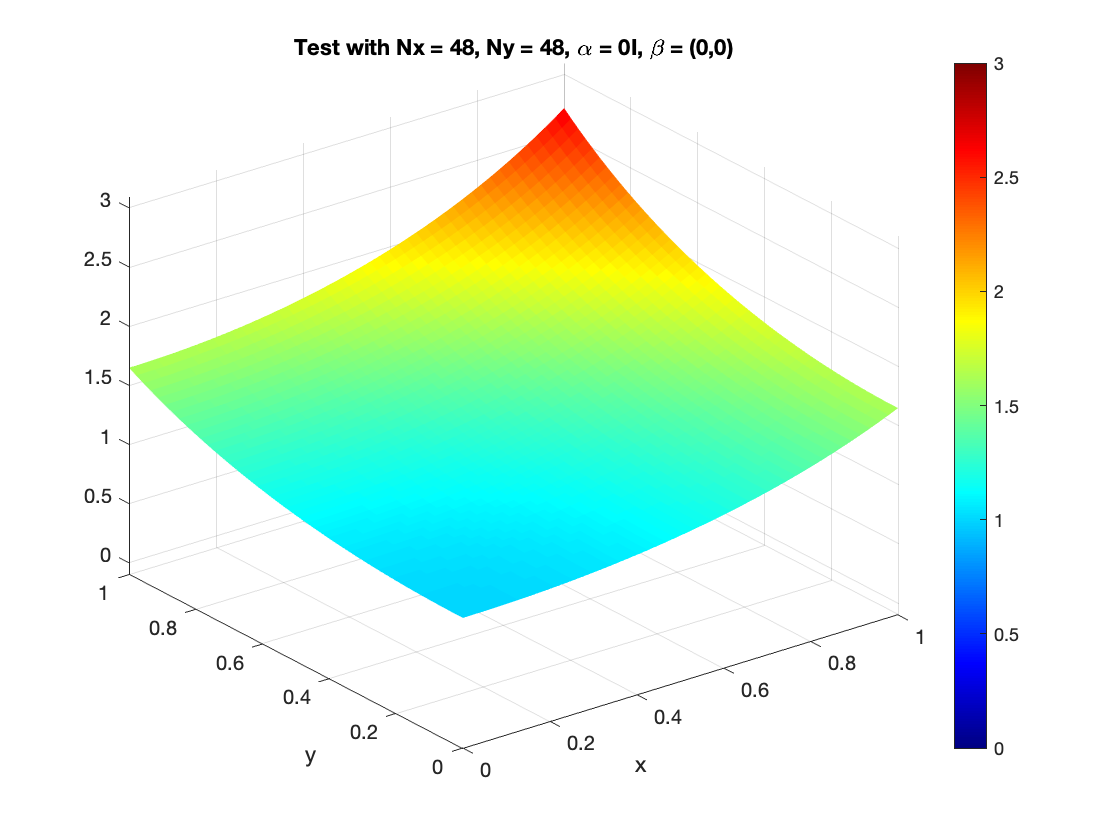}
}
\centerline{
\includegraphics[scale=0.1]{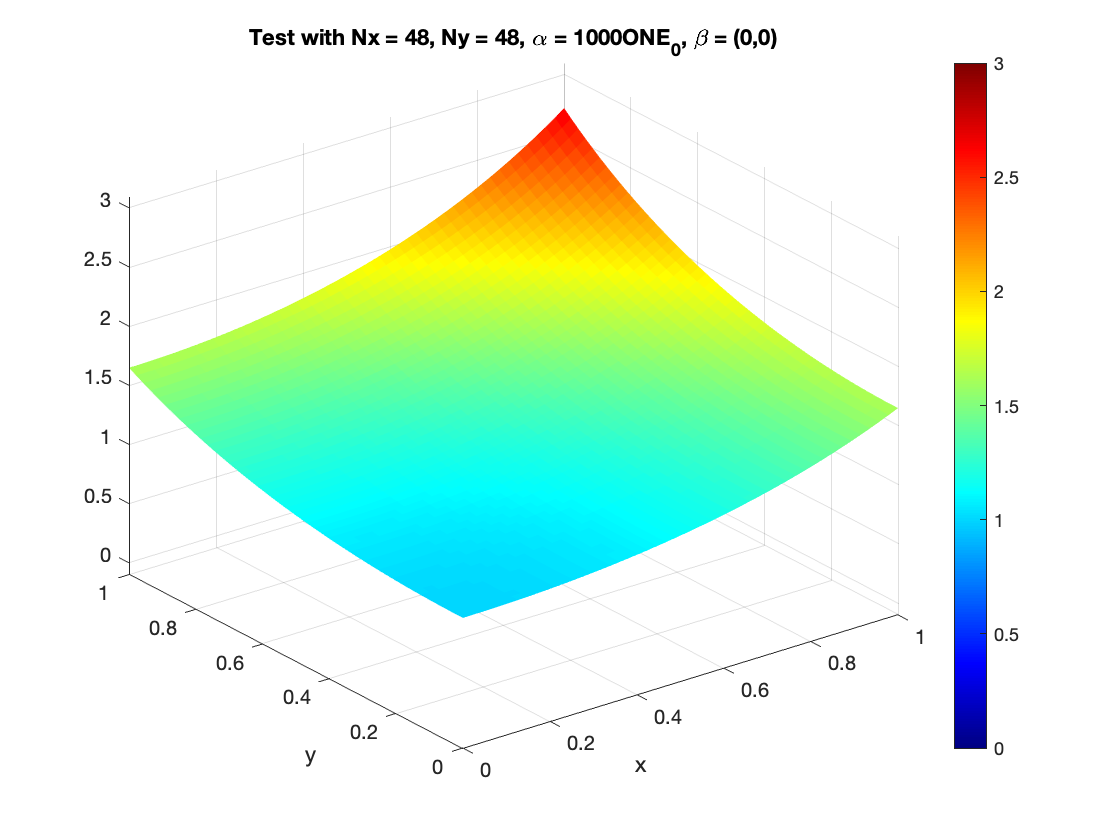}
\includegraphics[scale=0.1]{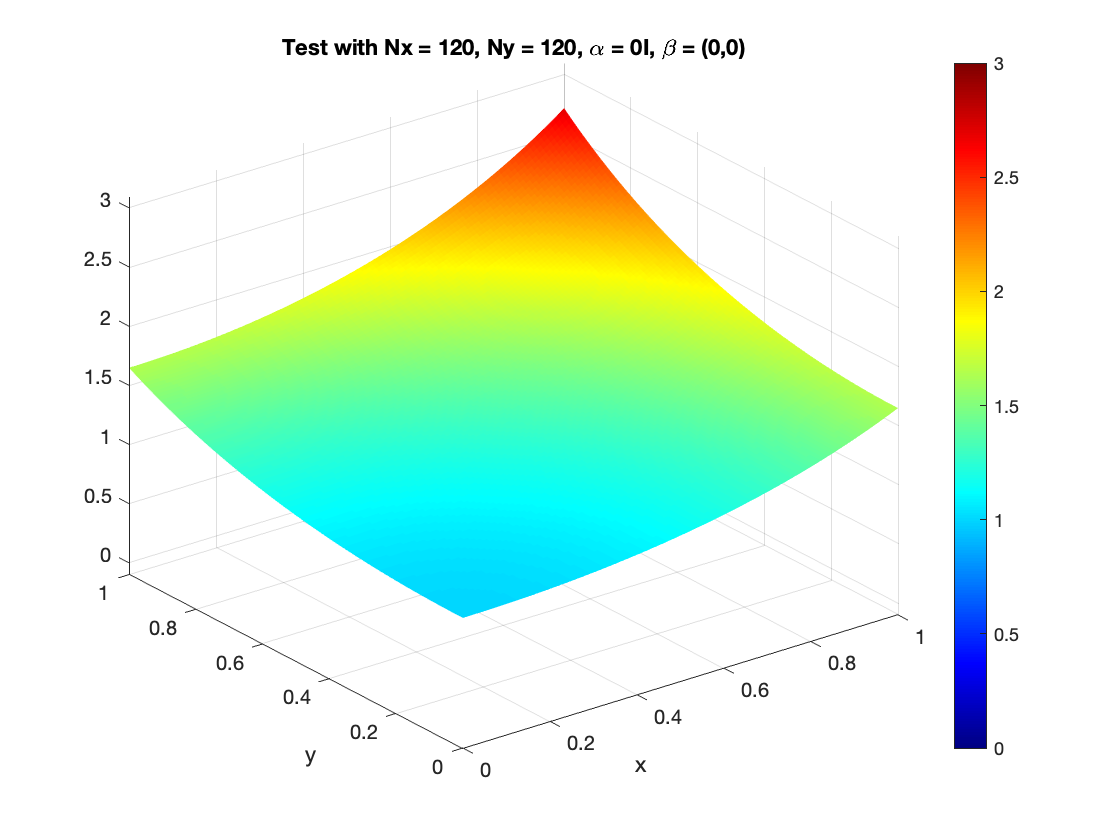}
}
\caption{
Plots of various approximations for Test 3 using $\hF_{\gamma,\sigma}$.  
The first row corresponds to $h=1.29$e-01 with $\sigma = 0$ and $\gamma = 1000, 100, 10$ 
from left to right.  
The second row corresponds to $h=3.01$e-02 with $\sigma = 0$ and $\gamma = 1000, 100, 10$ 
from left to right.  
The third row corresponds to $h=3.01$e-02 with $\sigma = 0$ and $\gamma = 1, 0$ 
from left to right.  
The first plot on the fourth row corresponds to $h=3.01$e-02 with $\sigma = 1000$ and $\gamma = -1000$.   
The second plot on the fourth row corresponds to $h=1.19$e-02 with $\gamma = \sigma = 0$.    
}
\label{test3_plots}
\end{figure}

%%%%%

\subsection{Test 4: Prescribed Gauss curvature equations}  

This example is adapted from \cite{FengNeilan14}. 
Let $K > 0$.  
The equation of prescribed Gauss curvature corresponds to the choice 
\[
	F[u] 
	= - \frac{\text{det}(D^2 u)}{\left(1 + |\nabla u|^2 \right)^{\frac{d+2}{2}}} 
	+ K f = 0 .  
\]
The problem is based on the Monge-Amp\`ere operator, 
and it is known that the problem when $f=1$ has a unique convex viscosity solution 
for each $K \in [0,K^*)$ for some positive constant $K^*$.  
We let $K=0.1$ and $\Omega = (0,1)^2$, and we choose $f$ and $g$ such that the exact solution is given by 
$u(x,y) = e^{\frac{x^2+y^2}{2}}$.  
The matrix $M_\alpha$ is easily found since 
$\left| \frac{\partial F}{\partial P_{ij}} \right| = | P_{k\ell} | / \left(1 + | \mathbf{v}|^2 \right)^2$
for $(k,\ell) = (2,2)$ if $(i,j)=(1,1)$, 
$(k,\ell) = (1,1)$ if $(i,j)=(2,2)$, 
and $(k,\ell) = (1,2)$ if $(i,j)=(1,2)$ or $(i,j) = (2,1)$.  

Rates of convergence for the various tests can be found in Table~\ref{test4_fig} where we 
consider the method $\hF_{\gamma,0}$ for $\gamma = 1000, 100, 10, 1, 0$
and Table~\ref{test4b_fig} where we consider the method $\hF_{-\sigma,\sigma}$ 
for $\sigma = 1000, 100, 10, 1$.  
Overall we observe similar behavior as Test 3 but less accuracy and lower rates 
when $\gamma > 0$ or $\sigma > 0$.  
The method $\hF_{0,0}$ does exhibit an optimal convergence rate of 2 
and is the most accurate of all of the methods tested.  
Since this problem involves the gradient operator, 
any monotone method that directly approximates the gradient 
would in general be limited to only first order accuracy 
until the mesh is fine enough to use the local uniform ellipticity 
to enforce the monotonicity with respect to $\nabla_{\bh}^\pm U_\alpha$.

\begin{table}[htb] 
{\small 
\begin{center}
\begin{tabular}{| c | c | c | c | c | c | c |}
		\hline
	& \multicolumn{2}{c |}{$\gamma = 1000$} & \multicolumn{2}{c |}{$\gamma = 100$} 
		& \multicolumn{2}{c |}{$\gamma = 10$} \\ 
		\hline
	 $h$ & $\ell^\infty$ Error & Order & $\ell^\infty$ Error & Order & $\ell^\infty$ Error & Order  \\ 
		\hline
	2.83e-01 & 1.30e+00 & & 5.57e-01 &  & 5.33e-01 &  \\ 
		\hline
	1.29e-01 & 1.27e+00 & 0.05 & 5.80e-01 & -0.05 & 4.05e-01 & 0.35 \\ 
		\hline
	6.15e-02 & 1.22e+00 & 0.09  & 5.24e-01 & 0.14 & 1.67e-01 & 1.20 \\ 
		\hline
	3.01e-02 & 1.15e+00 & 0.20 & 2.79e-01 & 0.88 & 8.64e-02 & 0.93 \\ 
		\hline
	1.99e-02 & 1.08e+00 & 0.29 & 1.72e-01 & 1.17 & 5.76e-02 & 0.98 \\ 
		\hline
	1.49e-02 & 9.78e-01 & 0.43 & 1.31e-01 & 0.96 & 4.24e-02 & 1.06 \\ 
		\hline
	1.19e-02 & 8.32e-01 & 0.61 & 1.06e-01 & 0.91 & 3.29e-02 & 1.12 \\ 
		\hline
\end{tabular}

\vspace{0.1in}

\begin{tabular}{| c | c | c | c | c | c | c |}
		\hline
	& \multicolumn{2}{c |}{$\gamma = 1$} & \multicolumn{2}{c |}{$\gamma = 0$} \\ 
		\hline
	 $h$ & $\ell^\infty$ Error & Order & $\ell^\infty$ Error & Order  \\ 
		\hline
	2.83e-01 & 2.56e-01 &  & 2.19e-02 &  \\ 
		\hline
	1.29e-01 & 1.08e-01 & 1.10 & 3.89e-03 & 2.19 \\ 
		\hline
	6.15e-02 & 5.52e-02 & 0.91 & 8.20e-04 & 2.11 \\ 
		\hline
	3.01e-02 & 2.51e-02 & 1.10 & 1.90e-04 & 2.05 \\ 
		\hline
	1.99e-02 & 1.49e-02 & 1.26 & 8.25e-05 & 2.02 \\ 
		\hline
	1.49e-02 & 1.01e-02 & 1.35 & 4.59e-05 & 2.01 \\ 
		\hline
	1.19e-02 & 7.30e-03 & 1.43 & 2.92e-05 & 2.01 \\ 
		\hline
\end{tabular}
\end{center}
}
\caption{
Approximation results for Test 4 using $\hF_{\gamma,0}$ with $\gamma = 1000, 100, 10, 1, 0$.   
}
\label{test4_fig}
\end{table}

\begin{table}[htb] 
{\footnotesize
\begin{center}
\begin{tabular}{| c | c | c | c | c | c | c | c | c |}
		\hline
	& \multicolumn{2}{c |}{$\sigma = 1000$} & \multicolumn{2}{c |}{$\sigma = 100$} 
		& \multicolumn{2}{c |}{$\sigma = 10$} & \multicolumn{2}{c |}{$\sigma = 1$} \\ 
		\hline
	 $h$ & $\ell^\infty$ Error & Order & $\ell^\infty$ Error & Order & $\ell^\infty$ Error & Order & $\ell^\infty$ Error & Order  \\ 
		\hline
	2.83e-01 & 3.68e-02 & & 3.63e-02 &  & 3.20e-02 &  & 1.81e-02 &  \\ 
		\hline
	1.29e-01 & 3.70e-02 & -0.01 & 3.40e-02 & 0.08 & 2.11e-02 & 0.53 & 6.83e-03 & 1.24 \\ 
		\hline
	6.15e-02 & 3.59e-02 & 0.04 & 2.70e-02 & 0.31 & 1.02e-02 & 0.98 & 2.09e-03 & 1.60 \\ 
		\hline
	3.01e-02 & 3.19e-02 & 0.17 & 1.61e-02 & 0.72 & 3.84e-03 & 1.38 & 5.75e-04 & 1.81 \\ 
		\hline
	1.99e-02 & 2.73e-02 & 0.38 & 1.04e-02 & 1.06 & 1.96e-03 & 1.62 & 2.61e-04 & 1.92 \\ 
		\hline
	1.49e-02 & 2.31e-02 & 0.57 & 7.22e-03 & 1.26 & 1.18e-03 & 1.74 & 1.47e-04 & 1.96 \\ 
		\hline
	1.19e-02 & 1.96e-02 & 0.73 & 5.28e-03 & 1.39 & 7.88e-04 & 1.81 & 9.45e-05 & 1.97 \\ 
		\hline
\end{tabular}
\end{center}
}
\caption{
Approximation results for Test 4 using $\hF_{\gamma,\sigma}$ 
for $\sigma = 1000$, $100$, $10$, $1$ and $\gamma = - \sigma$.   
}
\label{test4b_fig}
\end{table}

%%%%%%%%%%%%%

%\section{Conclusion} \label{conclusion_sec}

%Extension to $F(D^2 u , \nabla u , u , x)$.  $\ell^2$ approach follows by using the anti-symmetric approximation $\nabla_{\bh} u$ based on central difference (as in DG paper and in prep HJ paper for FD).  

%%%%%%%%%%%%%%%%%%%%%%%%%%%%%%%%%%%%%%%%%%%%%%%%

\end{document}